\documentclass[aos,a4paper,reqno,preprint]{imsart}
\usepackage[utf8]{inputenc} 
\usepackage[T1]{fontenc}    
\usepackage{url}            
\usepackage{booktabs}       
\usepackage{amsfonts,amssymb,amsthm,mathtools}
\usepackage{nicefrac,xfrac} 
\usepackage{microtype}      
\usepackage{graphicx,enumitem}
\usepackage{hyperref}
\usepackage[sort,comma,authoryear,round]{natbib}
\usepackage[ruled]{algorithm2e}
\usepackage{algorithmic}
\usepackage[table]{xcolor}
\usepackage{wrapfig}

\usepackage{titletoc}

\usepackage{subcaption}

\startlocaldefs

\newcommand{\myparagraph}[1]{\vspace{0.1cm}\noindent \textbf{\textit{#1.}}}

\newcommand{\bsX}{\boldsymbol{X}}
\newcommand{\bsomega}{\boldsymbol{\omega}}
\newcommand{\bsw}{\boldsymbol{w}}
\newcommand{\bsY}{\boldsymbol{Y}}
\newcommand{\bsy}{\boldsymbol{y}}
\newcommand{\bsv}{\boldsymbol{v}}
\newcommand{\bsxi}{\boldsymbol{\xi}}
\newcommand{\bsa}{\boldsymbol{a}}
\newcommand{\bsc}{\boldsymbol{c}}
\newcommand{\bstheta}{\boldsymbol{\theta}}
\newcommand{\bsx}{\boldsymbol{x}}
\newcommand{\bsb}{\boldsymbol{b}}
\newcommand{\bsB}{\boldsymbol{B}}
\newcommand{\bsV}{\boldsymbol{V}}
\newcommand{\bbeta}{\boldsymbol{\beta}}
\newcommand{\Probf}{\mathbf{P}}
\newcommand{\Qprobf}{\mathbf{Q}}
\newcommand{\bfW}{\mathbf{W}}
\newcommand{\bfX}{\mathbf{X}}
\newcommand{\bfA}{\mathbf{A}}
\newcommand{\bfO}{\mathbf{O}}
\newcommand{\bfPsi}{\mathbf{\Psi}}
\newcommand{\Prob}{\mathbb{P}}

\newcommand{\Exp}{\mathbb{E}}

\newcommand{\bbR}{\mathbb{R}}
\newcommand{\bbN}{\mathbb{N}}
\newcommand{\risk}{\mathcal{R}}
\newcommand{\class}[1]{\mathcal{#1}}
\newcommand{\Expf}{\mathbf{E}}
\newcommand{\1}{\boldsymbol 1}
\newcommand{\eqdef}{\vcentcolon=}
\newcommand{\bsdelta}{\boldsymbol \Delta}

\newcommand{\bszeta}{\boldsymbol{\zeta}}
\newcommand{\bfSigma}{\mathbf{\Sigma}}
\newcommand{\parent}[1]{\left( #1 \right)}
\newcommand{\ens}[1]{\left\{ #1\right\}}
\newcommand{\norm}[1]{\left\lVert#1\right\rVert}
\newcommand{\normin}[1]{\lVert#1\rVert}
\newcommand{\abs}[1]{\left\lvert #1\right\rvert}
\newcommand{\absin}[1]{\lvert #1\rvert}
\newcommand{\enscond}[2]{\left\{ #1 \, : \, #2\right\}}
\newcommand{\ind}[1]{\mathbb{I}{\left\{#1\right\}}}
\newcommand{\scalar}[2]{\left\langle #1, #2\right\rangle}
\newcommand{\scalarin}[2]{\langle #1, #2\rangle}
\newcommand{\data}{\class{D}}

\newcommand{\sW}{\mathsf{W}}

\newcommand{\hbsb}{\boldsymbol{\hat b}}

\newcommand{\hbeta}{{\boldsymbol{\hat \beta}}}
\newcommand{\hb}{\hat{b}}

\renewcommand{\d}{\,\mathrm{d}}

\newcommand{\hbsdelta}{\boldsymbol{\hat\Delta}}

\DeclareMathOperator*{\op}{op}
\DeclareMathOperator{\nur}{NUR}

\DeclareMathOperator{\test}{test}
\DeclareMathOperator{\train}{train}
\DeclareMathOperator*{\Var}{Var}

\DeclareMathOperator*{\diag}{diag}
\DeclareMathOperator*{\Id}{Id}
\DeclareMathOperator*{\unlab}{unlab}
\DeclareMathOperator*{\argmin}{arg\,min}
\DeclareMathOperator{\KL}{KL}
\DeclareMathOperator{\TV}{TV}
\DeclareMathOperator{\KS}{KS}
\DeclareMathOperator{\Law}{Law}
\DeclareMathOperator{\PF}{PF}
\newcommand{\ie}{{\em i.e.,~}}

\newcommand{\eg}{{\em e.g.,~}}

\newcommand{\resp}{{\em resp.~}}
\newcommand{\rhs}{{\em r.h.s.~}}

\newcommand{\wrt}{{\em w.r.t.~}}
\newcommand{\iid}{{\rm i.i.d.~}}

\newcommand{\simiid}{\overset{\text{\iid}}{\sim}}
\newtheorem{theorem}{Theorem}[section]
\newtheorem{definition}[theorem]{Definition}
\newtheorem{proposition}[theorem]{Proposition}
\newtheorem{lemma}[theorem]{Lemma}
\newtheorem{assumption}[theorem]{Assumption}
\newtheorem{remark}[theorem]{Remark}
\newtheorem{corollary}[theorem]{Corollary}
\newtheorem*{theorem*}{Theorem}
\newtheorem*{proposition*}{Proposition}
\newtheorem*{lemma*}{Lemma}

\definecolor{green}{rgb}{0.0, 0.5, 0.0}
\definecolor{red}{rgb}{0.8, 0.0, 0.0}
\definecolor{blue}{rgb}{0.01, 0.28, 1.0}
\definecolor{yellow}{rgb}{0.98, 0.93, 0.36}

\usepackage[textwidth=2.0cm, textsize=tiny]{todonotes} 

\definecolor{red}{rgb}{0.0, 0.0, 0.0}

\usepackage{ wasysym }

\definecolor{orange}{rgb}{0.76, 0.23, 0.13}

\hypersetup{
     colorlinks = true,
     linkcolor = blue,
     anchorcolor = blue,
     citecolor = blue,
     filecolor = blue,
     urlcolor = blue
}

\DeclareMathOperator*{\MSE}{MSE} 

    \newlength{\leftstackrelawd}
    \newlength{\leftstackrelbwd}
    \def\leftstackrel#1#2{\settowidth{\leftstackrelawd}%
    {${{}^{#1}}$}\settowidth{\leftstackrelbwd}{$#2$}%
    \addtolength{\leftstackrelawd}{-\leftstackrelbwd}%
    \leavevmode\ifthenelse{\lengthtest{\leftstackrelawd>0pt}}%
    {\kern-.5\leftstackrelawd}{}\mathrel{\mathop{#2}\limits^{#1}}}
\endlocaldefs

\begin{document}
\begin{frontmatter}
\title{A minimax framework for quantifying risk-fairness trade-off in regression}
\runtitle{Risk-fairness trade-off in regression}


\begin{aug}
\author{\fnms{Evgenii} \snm{Chzhen}\thanksref{orsay}\ead[label=e1]{evgenii.chzhen@universite-paris-saclay.fr}}
\and
\author{\fnms{Nicolas}
\snm{Schreuder}\thanksref{crest}\ead[label=e2]{nicolas.schreuder@ensae.fr}}

\runauthor{E. Chzhen and N. Schreuder}

\address[orsay]{Universit\'e Paris-Saclay, CNRS, Laboratoire de math\'ematiques d’Orsay\\ \printead{e1}}
\address[crest]{Università di Genova, MaLGa, DIBRIS\\
Institut Polytechnique de Paris, ENSAE, CREST\\
\printead{e2}}


\end{aug}

\begin{abstract}

We propose a theoretical framework for the problem of learning a real-valued function which meets fairness requirements.
This framework is built upon the notion of $\alpha$-relative (fairness) improvement of the regression function which we introduce using the theory of optimal transport.
Setting $\alpha = 0$ corresponds to the regression problem under the Demographic Parity constraint, while $\alpha = 1$ corresponds to the classical regression problem without any constraints.
For $\alpha \in (0, 1)$ the proposed framework allows to continuously interpolate between these two extreme cases and to study partially fair predictors.
Within this framework we precisely quantify the cost in risk induced by the introduction of the fairness constraint.
We put forward a statistical minimax setup and derive a general problem-dependent lower bound on the risk of any estimator satisfying $\alpha$-relative improvement constraint.
We illustrate our framework on a model of linear regression with Gaussian design and systematic group-dependent bias, deriving matching (up to absolute constants) upper and lower bounds on the minimax risk under the introduced constraint.
{\color{red}We provide a general post-processing strategy which enjoys fairness, risk guarantees and can be applied on top of any black-box algorithm.} Finally, we perform a simulation study of the {\color{red}linear model and numerical experiments of benchmark data, validating our theoretical contributions.}

\end{abstract}


\begin{keyword}
\kwd{Algorithmic fairness}
\kwd{risk-fairness trade-off}
\kwd{regressions}
\kwd{Demographic Parity}
\kwd{least-squares}
\kwd{optimal transport}
\kwd{minimax analysis}
\kwd{statistical learning}
\end{keyword}

\end{frontmatter}

\maketitle

\section{Introduction}
Data driven algorithms are deployed in almost all areas of modern daily life and it becomes increasingly more important to adequately address the fundamental issue of historical biases present in the data \citep{barocas-hardt-narayanan}.
The goal of algorithmic fairness is to bridge the gap between the statistical theory of decision making and the understanding of justice, equality, and diversity.
The literature on fairness is broad and its volume increases day by day, we refer the reader to~\citep{mehrabi2019survey,barocas-hardt-narayanan} for a general introduction on the subject and to~\citep{oneto2020fairness,del2020review} for reviews of the most recent theoretical advances.

Basically, the mathematical definitions of fairness can be divided into two groups~\citep{dwork2012fairness}: \emph{individual fairness} and \emph{group fairness}.
The former notion reflects the principle that similar individuals must be treated similarly, which translates into Lipschitz type constraints on possible prediction rules.
The latter defines fairness on population level via (conditional) statistical independence of a prediction from a sensitive attribute (\eg gender, ethnicity).
A popular formalization of such notion is through the \emph{Demographic Parity} constraint, initially introduced in the context of binary classification \citep{calders2009building}. Despite of some limitations~\citep{hardt2016equality}, the concept of Demographic Parity is natural and suitable for a range of applied problems~\citep{Koeppen_Yoshida_Ohnishi14,Zink_Rose19}.

In this work we study the regression problem of learning a real-valued prediction function, which complies with an approximate notion of Demographic Parity while minimizing expected squared loss.

Unlike its classification counterpart, the problem of fair regression has received far less attention in the literature.
However, as argued by~\cite{agarwal2019fair}, classifiers only provide binary decisions, while in practice final decisions are taken by humans based on predictions from the machine. In this case a continuous prediction is more informative than a binary one and justifies the need for studying fairness in the regression framework.

\myparagraph{Notation}
For any univariate probability measure $\mu$ we denote by $F_{\mu}$ (\emph{resp.} $F_{\mu}^{-1}$) the cumulative distribution function (\emph{resp.} the quantile function) of $\mu$.
For two random variables $U$ and $V$ we denote by $\Law( U \mid V {=} v)$ the conditional distribution of the random variable $U \mid V {=} v$ and we write $U \stackrel{d}{=} V$ to denote their equality in distribution. For any integer $K \geq 1$, we denote by $\Delta^{K - 1}$ the probability simplex in $\bbR^K$ and we write $[K] = \{1, \dots, K\}$. For any $a, b \in \bbR$ we denote by $a \vee b$ (\resp $a \wedge b$) the maximum (\resp the minimum) between $a, b$. We denote by $\mathcal{P}_2(\mathbb{R}^d)$ the space of probability measures on $\mathbb{R}^d$ with finite second-order moment.


\section{Problem statement and contributions}
\label{sec:problem_statement_and_contributions}

We study the regression problem when a sensitive attribute is available. The statistician observes triplets $(\bsX_1, S_1, Y_1), \ldots, (\bsX_n, S_n, Y_n) \in \bbR^p \times [K] \times \bbR$, which are connected by the following regression-type relation
\begin{align}
    \label{eq:model_general}
    Y_i = f^*(\bsX_i, S_i) + \xi_i\enspace,\qquad i \in [n]\enspace,
\end{align}
where $\xi_i \in \bbR$ is such that $\Expf[\xi_i \mid \bsX_i] = 0$ and $f^* : \bbR^p \times [K] \to \bbR$ is the regression function.
Here for each $i \in [n]$, $\bsX_i$ is a feature vector taking values in $\bbR^p$, $S_i$ is a sensitive attribute taking values in $[K]$, and $Y_i$ is a real-valued dependent variable.
A prediction is any measurable function of the form $f : \bbR^p \times [K] \to \bbR$.
We define the risk of a prediction function $f$ via the $\ell_2$ distance\footnote{{\color{red}The extension to $\ell_q$ losses is provided in Appendix~\ref{sec:q_losses}}.} to the regression function $f^*$ as
\begin{align*}
\tag{\textbf{Risk measure}}
    \risk(f) \eqdef \|f - f^*\|_{2}^2 \eqdef \sum_{s = 1}^K w_s\Exp\left[(f(\bsX, S) - f^*(\bsX, S))^2 \mid S = s \right]\enspace,
\end{align*}
where $\Exp[ \cdot \mid S {=} s]$ is the expectation \wrt the distribution of the features $\bsX$ in the group $S = s$ and $\bsw = (w_1, \ldots, w_K)^\top \in \Delta^{K-1}$ is a probability vector, which weights the group-wise risks.

For any $s\in[K]$ define $\nu^*_s$ as $\Law(f^*(\bsX, S) \mid S {=} s)$ --
the distribution of the optimal prediction inside the group $S = s$.
Throughout this work we make the following assumption on those measures, which is, for instance, satisfied in linear regression with Gaussian design.
\begin{assumption}\label{as:atomless}
    Measures $\{\nu^*_s\}_{s \in [K]}$ are non-atomic with finite second moments.
\end{assumption}

\subsection{Regression with fairness constraints}
\label{subsec:regression_with_fairness_constraints}

 Any predictor $f$ induces a group-wise distribution of the predicted outcomes
 $\Law(f(\bsX, S) \mid S{=}s)$ for $s \in [K]$.
 The high-level idea of \emph{group fairness} notions is to bound or diminish an eventual discrepancy between these distributions.

We define the \emph{unfairness} of a predictor $f$ as the sum of the weighted distances between $\{\Law(f(\bsX, S) \mid S{=}s)\}_{s \in [K]}$ and 
their common barycenter \wrt the Wasserstein-2 distance\footnote{See Appendix~\ref{sec:Wassersteinreminder} for a reminder on Wasserstein distances.}:
\begin{align}
        \class{U}(f) \eqdef \min_{\nu \in \class{P}_2(\bbR)} \sum_{s = 1}^K w_s \sW_2^2\big(\Law(f(\bsX, S) \mid S{=}s),\, \nu\big)\enspace.\tag{{\bf Unfairness measure}}
\end{align}
In particular, since the Wasserstein-2 distance is a metric on the space probability distributions with finite second-order moment $\mathcal{P}_2(\mathbb{R}^d)$, a predictor $f$ is such that $\class{U}(f)=0$ if and only if it satisfies the Demographic Parity (DP) constraint defined as
\begin{align}
    \big(f(\bsX, S) \mid S = s\big) \leftstackrel{d}{=} \big(f(\bsX, S) \mid S = s'\big), \quad \forall s, s' \in [K] \enspace.\tag{{\bf DP}}
\end{align}
Exact DP is not necessarily desirable in practice and it is common in the literature to consider \emph{relaxations} of this constraint. In this work we introduce the \emph{$\alpha$-Relative Improvement} ($\alpha$-RI) constraint -- a novel DP relaxation based on our unfairness measure.
We say that a predictor $f$ satisfies the $\alpha$-RI constraint for some $\alpha \in [0, 1]$ if its unfairness is at most an $\alpha$ fraction of the unfairness of the regression function $f^*$, that is, $\class{U}(f) \leq \alpha\,\class{U}(f^*)$.
Importantly, the fairness requirement is stated relatively to the unfairness of the regression function $f^*$, which allows to make a more informed choice of $\alpha$.

Formally, for a fixed $\alpha \in [0, 1]$, the goal of a statistician in our framework is to build an estimator $\hat f$ using data, which enjoys two guarantees (with high probability)
\begin{align*}
    &\textbf{$\alpha$-RI guarantee:}\quad\class{U}(\hat f) \leq \alpha\, \class{U}(f^*)\qquad\text{and}\qquad
    \textbf{Risk guarantee:}\quad\risk(\hat f) \leq r_{n, \alpha, f^*}\enspace.
\end{align*}
The former ensures that $\hat f$ satisfies the $\alpha$-RI constraint.
In the latter guarantee we seek the sequence $r_{n, \alpha, f^*}$ being as small as possible in order to quantify \emph{two effects}: the introduction of the $\alpha$-RI \emph{fairness constraint} and the \emph{statistical estimation}.
We note that $r_{n, \alpha, f^*}$ depends on the sample size $n$, the fairness parameter $\alpha$, as well as the regression function $f^*$ to be estimated, we clarify the reason for this dependency later in the text.
\subsection{Contributions}
\label{subsec:contributions}

The first natural question that we address is: assuming that the underlying distribution of $X \mid S$ and the regression function $f^*$ are known, which prediction rule $f_{\alpha}^*$ minimizes the expected squared loss under the $\alpha$-RI constraint $\class{U}(f^*_{\alpha}) \leq \alpha \, \class{U}(f^*)$?
To answer this question we shift the discussion to the population level and define a collection $\{f^*_{\alpha}\}_{\alpha \in [0, 1]}$ of \emph{oracle $\alpha$-RI} indexed by the parameter $\alpha$ as
\begin{align}
    \tag{{\bf Oracle $\alpha$-RI}}
    f^*_{\alpha} \in \argmin\enscond{\risk(f)}{\class{U}(f) \leq \alpha\,\class{U}(f^*)}
    \enspace,\qquad \forall \alpha \in [0, 1]\enspace.
\end{align}
For $\alpha=0$ the predictor $f^*_0$ corresponds to the optimal fair predictor in the sense of DP while for $\alpha=1$ the corresponding predictor $f^*_1$ coincides with the regression function $f^*$. Those two extreme cases have been previously studied but, up to our knowledge, nothing is known about those ``partially fair'' predictors.
Our study of the family $\{f^*_{\alpha}\}_{\alpha \in [0, 1]}$
serves as a basis for our statistical framework and analysis. It also reveals the intrinsic interplay of the fairness constraint with the risk measure.

The contributions of this work can be roughly split into three interconnected groups:
\begin{enumerate}
    \item We provide a theoretical study of the family of oracle $\alpha$-RI $\{f^*_{\alpha}\}_{\alpha \in [0, 1]}$ on the population level;
    \item We introduce {a minimax statistical} framework and derive a general problem-dependent minimax lower bound for the problem of regression under the $\alpha$-RI constraint;
    \item
    We derive minimax optimal rate of convergence for the statistical model of linear regression with systematic group-dependent bias and Gaussian design under the $\alpha$-RI constraint.
\end{enumerate}
\myparagraph{Properties of oracle $\alpha$-RI $\{f^*_{\alpha}\}_{\alpha \in [0, 1]}$}
\begin{figure}[t!]
\centering
\includegraphics[width=\textwidth]{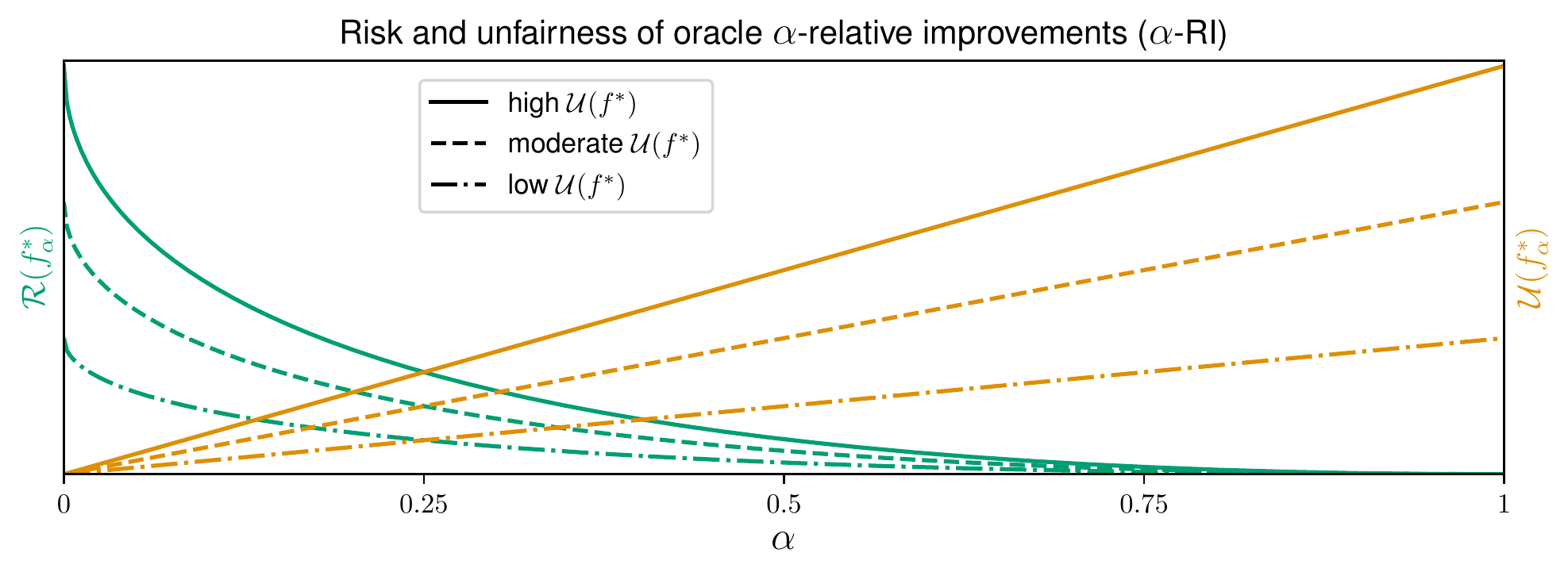}%
\caption{Risk $\risk$ and unfairness $\class{U}$ of $\alpha$-RI oracles $\{f^*_{\alpha}\}_{\alpha \in [0, 1]}$. Green curves (decreasing, convex) correspond to the risk, while orange curves (increasing, linear) correspond to the unfairness. Each pair of curves (solid, dashed, dashed dotted) corresponds to three regimes: high, moderate, and low unfairness of the regression function $f^*$ respectively.}
\label{fig:illustration1}
\end{figure}
It has been shown that, under the squared loss, the optimal fair predictor $f^*_0$ can be obtained as the solution of a Wasserstein-2 barycenter problem~\citep{gouic2020price,chzhen2020fair}.
In Section~\ref{SUBSEC:GENERAL} we study the whole family $\{f^*_{\alpha}\}_{\alpha \in [0, 1]}$ for arbitrary choice of $\alpha \in [0, 1]$.
To provide complete characterization of $\{f^*_{\alpha}\}_{\alpha \in [0, 1]}$ we derive Lemma~\ref{lem:geometric_general}, which could be of independent interest.
This result can be summarized as follows:
given a fixed collection of points $a_1, \ldots, a_K$ in an abstract metric space $(\class{X}, d)$, if one walks along the (constant speed) geodesics starting from $a_s$ and leading to their (weighted) barycenter until it reaches a proportion $\alpha$ of the full path, then these intermediate points $b_1, \ldots, b_K$ minimize the weighted distance to the initial points while being $\alpha$-closer to their own barycenter.
This abstract result enables us to characterize explicitly oracle $\alpha$-RI $\{f^*_\alpha\}_{\alpha \in [0, 1]}$. In particular, we show that the family of oracle $\alpha$-RI $\{f^*_{\alpha}\}_{\alpha \in [0, 1]}$ admits a simple structure: for any $\alpha \in [0, 1]$ the prediction $f^*_{\alpha}$ is the point-wise convex combination of the regression function $f^* \equiv f^*_1$ and the optimal fair predictor $f^*_{0}$, that is,
\begin{align*}
    f^*_{\alpha}(\bsx, s) = \sqrt{\alpha}f^*_{1}(\bsx, s) + (1 {-} \sqrt{\alpha})f_0^*(\bsx, s),\qquad \forall\,(\bsx, s) \in \bbR^p \times [K]\enspace.
\end{align*}
The final contribution of Section~\ref{SUBSEC:GENERAL} is the quantification of the risk-fairness trade-off on the population level. In particular, Lemma~\ref{lem:distance_fair_and_almost} establishes that for every $\alpha \in [0, 1]$ it holds that
\begin{align*}
    \risk(f^*_{\alpha}) = (1 {-} \sqrt{\alpha})^2\risk(f^*_{0})\quad\text{and}\quad \class{U}(f^*_{\alpha}) = \alpha\,\class{U}(f^*)\enspace.
\end{align*}
Observe that $f^*_0$, which is the optimal fair predictor in terms of DP, has the highest risk and the lowest unfairness, while the situation is reversed for $f^*_1 \equiv f^*$ -- the risk is the lowest and the unfairness is the highest.
Since the function $\alpha \to (1 {-} \sqrt{\alpha})^2$ grows rapidly in the vicinity of zero, even a mild relaxation of the exact fairness constraint $(\alpha = 0)$ yields a noticeable improvement in terms of the risk while having a low unfairness inflation.
For instance, the risk of $f^*_{\sfrac{1}{2}}$ is only around $8.5\%$ of the risk of $f^*_0$, while its fairness is two times better than that of $f^*$. This observation is illustrated in Figure~\ref{fig:illustration1}.


\myparagraph{Minimax framework}
In order to quantify the \emph{statistical} price of fairness, in Section~\ref{sec:minimax_setup} we propose a minimax framework and in Section~\ref{sec:general_lower} we derive a general problem-dependent lower bound on the minimax risk of estimators satisfying the $\alpha$-RI constraint.
%
Statistical study of the model in Eq.~\eqref{eq:model_general} typically requires additional assumptions to provide meaningful statistical guarantees.
Classically, one chooses a set $\class{F}$ of possible candidates for the regression function $f^*$ (\eg linear functions) and, possibly, introduces additional conditions on nuisance parameters of the model via some set $\Theta$ (\eg variance of the noise).
The goal of our lower bound is to understand fundamental limits of the problem of prediction under $\alpha$-RI constraint in arbitrary statistical model for Eq.~\eqref{eq:model_general}.
To this end, we show in Theorem~\ref{THM:GENERAL_LOWER} that \emph{any estimator $\hat f$ satisfying the $\alpha$-RI} constraint with high probability must incur
 \begin{align*}
     \risk(\hat f) \geq \delta_n(\class{F}, \Theta) \vee (1 {-} \sqrt{\alpha})^2\class{U}({f^*})\enspace,
 \end{align*}
where $\delta_n(\class{F}, \Theta)$ is the rate one would obtain \emph{without} restricting the set of possible estimators.

\myparagraph{Application to linear model}
The goal of Section~\ref{SEC:LINEAR} is to demonstrate that the general problem-dependent lower bound does indeed yield minimax optimal rates.
To this end, we apply our machinery to the problem of linear regression with systematic bias formalized by the following linear model
\begin{align*}
    Y_i = \scalar{\bsX_i}{\bbeta^*} + b_{S_i}^* + \xi_i,\quad i = 1, \dots, n\enspace,
\end{align*}
where the $\xi_i$'s are \iid zero mean Gaussian with variance $\sigma^2$ and the $p$-dimensional covariates $\{\bsX_i\}_{i = 1}^n$ are \iid Gaussian random vectors.
We propose an estimator $\hat f$ which, with probability at least $1 - \delta$, satisfies $\class{U}(\hat f) \leq \alpha\, \class{U}(f^*)$ and achieves the following minimax optimal rate
\begin{align*}
    &\risk(\hat f) \asymp \left\{\sigma^2\parent{{\frac{p + K}{n}} +
    {\frac{\log (\sfrac{1}{\delta})}{n}}} \right\} \bigvee \bigg\{(1 {-} \sqrt{\alpha})^2\class{U}(f^*) \bigg\}\enspace.
\end{align*}
Finally, we conduct a simulation study of the proposed estimator $\hat f$ and compare its performance with more straightforward approaches in terms of unfairness and risk.

\myparagraph{Ad-hoc procedure and experiments on \texttt{CRIME} dataset} {\color{red}The estimator that will be developed in the context of linear model with systematic bias relies heavily on the linear model and Gaussian features assumption. Thus, in Section~\ref{sec:estimation}, we propose a general post-processing estimator, which enjoys fairness and risk guarantees. Unlike the case of linear model, the optimality of these guarantees remains open. In Section~\ref{sec:real_exps}, we provide empirical study of estimators from Sections~\ref{SEC:LINEAR} and~\ref{sec:estimation}, validating our theoretical claims numerically.
}

\section{Prior and related works}
\label{sec:prior_and_related_works}
Until very recently, contributions on  fair regression were almost exclusively focused on the practical incorporation of proxy fairness constraints in classical learning methods, such as random forest, ridge regression, kernel based methods to name a few~\citep{calders2013controlling,komiyama2017two,berk2017convex,perez2017fair,raff2017fair,fitzsimons2018equality}.
Several works empirically study the impact of (relaxed) fairness constraints on the risk~\citep{bertsimas2012efficiency,zliobaite2015relation,haas2019price,Wick_Panda_Tristan19,zafar2017fairness}.
Yet, the problem of precisely quantifying the effect of such constraints on the risk has not been tackled.

More recently, statistical and learning guarantees for fair regression were derived \citep{agarwal2019fair,gouic2020price,chzhen2020fair,chiappa2020general,fitzsimons2019general,plevcko2019fair,chzhen2020fairTV}.
The closest works to our contribution are that of~\cite{gouic2020price,chzhen2020fair,chiappa2020general}, who draw a connection between the problem of exactly fair regression of demographic parity and the multi-marginal optimal transport formulation~\citep{gangbo1998optimal,agueh2011barycenters}.

As already mentioned in the previous section, considering predictors which satisfy the DP constraint incurs an unavoidable price in terms of the risk.
Depending on the application at hand, this price might or might not be reasonable.
However, since the notion of DP is completely fairness driven, it does not allow to quantify the price of considering ``fairer'' predictions than the regression function $f^*$.
For this reason, several contributions relax this constraint, forcing a milder fairness requirement.
A natural idea is to define a functional $\class{U}$ which quantifies the violation of the DP constraint and to declare a prediction approximately fair if this functional does not exceed a user pre-specified threshold.
In recent years a large variety of such relaxations has been proposed: correlation based~\citep{baharlouei2019r,mary2019fairness,komiyama2018nonconvex}; Kolmogorov-Smirnov distance~\citep{agarwal2019fair}; Mutual information~\citep{steinberg2020fast,steinberg2020fairness}; Total Variation distance~\citep{oneto2019general,oneto2019learning}; Equality of means and higher moment matching~\citep{raff2017fair,fitzsimons2019general,calders2013controlling,berk2017convex,olfat2020covariance,Donini_Oneto_Ben-David_Taylor_Pontil18}; Maximum Mean Discrepancy~\citep{quadrianto2017recycling,madras2018learning}; Wasserstein distance~\citep{chiappa2020general,gouic2020price,chzhen2020fair,gordaliza2019obtaining}.

\subsection{Other notions of unfairness}
The most common relaxations of the Demographic Parity constraint are based on the Total Variation (TV) and the Kolmogorov-Smirnov (KS) distances~\citep{agarwal2019fair, oneto2019general,Agarwal_Beygelzimer_Dubik_Langford_Wallach18,chzhen2020fairTV}. There are various ways to use the TV or KS in order to build a functional $\class{U}$, which quantifies the violation of the DP constraint.
To compare those measures of discrepancy with the one that we introduce in our work, we define $\class{U}_{\TV}$ and $\class{U}_{\KS}$ as follows
\begin{align*}
    &\textbf{TV unfairness:} &&\class{U}_{\TV}(f) \eqdef \sum_{s \in [K]} \TV\left(\Law(f(\bsX, S) \mid S = s),\, \Law(f(\bsX, S))\right)\enspace,\\
    &\textbf{KS unfairness:} &&\class{U}_{\KS}(f) \eqdef \sum_{s \in [K]} \KS\left(\Law(f(\bsX, S) \mid S = s),\, \Law(f(\bsX, S))\right)\enspace.
\end{align*}
Using these notions, one wishes to study those predictors $f$ which satisfy relaxed fairness constraint $\class{U}_{\square}(f) \leq \varepsilon$, where $\square$ is $\KS$ or $\TV$ and $\varepsilon \geq 0$ is a user specified parameter.
Note that since both $\KS$ and $\TV$ are metrics, setting $\varepsilon = 0$ is equivalent to the DP constraint.  Meanwhile, for $\varepsilon > 0$ these formulations allow some slack.
It is known that the TV distance is rather strong and extremely sensitive to small changes in distributions which is the major drawback of the TV unfairness.
This limitation can be addressed by the KS unfairness due to an obvious relation $\class{U}_{\KS}(f) \leq \class{U}_{\TV}(f)$.

In our work we argue that the introduced notion of unfairness $\class{U}$ is better suited for the problem of regression with squared loss under fairness constraint. 
Indeed, we prove in Lemma~\ref{lem:distance_fair_and_almost} that $\class{U}$ can be naturally connected to the squared risk and allows to give a precise quantification of the risk-fairness trade-off.
This result is the major advantage of $\class{U}$ over both $\class{U}_{\KS}$ and $\class{U}_{\TV}$. Nevertheless, it is still interesting to understand whether a more popular $\KS$ unfairness can be related to $\class{U}$ that we introduce.
In Appendix we prove the following connection.
\begin{proposition}
    Fix some predictor $f : \bbR^p \times [K] \to \bbR$.
    Assume that $\Law(f(\bsX, S) \mid S {=} s) \in \class{P}_2(\bbR)$ and it admits density bounded by $C_{f, s} > 0$ for all $s \in [K]$, then\footnote{One can erase $\|\sfrac{1}{\bsw}\|_{\infty}$ from the bound introducing these weights into the definition of $\class{U}_{\KS}(f)$.}
    \begin{align*}
        \class{U}_{\KS}(f) \leq  \|\sfrac{1}{\bsw}\|_{\infty}\sqrt{ 8\bar{C}_{f} }\cdot\class{U}^{1/4}(f)\enspace,
    \end{align*}
    where $\bar{C}_f = \sum_{s = 1}^Kw_sC_{f, s}$ and $\sfrac{1}{\bsw} = (\sfrac{1}{w_1}, \ldots, \sfrac{1}{w_K})^\top$.
\end{proposition}
The latter result indicates that if one can control the unfairness $\class{U}$ introduced in this work, one also has some control over the $\KS$ unfairness.
Note that the leading constant of the previous bound depends on the predictor $f$. More precisely, this constant corresponds to the upper bound on the density of $f(\bsX, S)$.

Another advantage of the introduced unfairness measure, and, in particular, the notion of $\alpha$-relative improvement is the fact that the parameter $\alpha$ has a clear practical interpretation, while the interpretation of $\varepsilon$ is not intuitive. Of course, using $\class{U}_{\KS}$ or $\class{U}_{\TV}$ one can also define unfairness of a predictor $f$ relatively to the regression function $f^*$.
{\color{red}However, the interpretation of $\KS$ or $\TV$ unfairness relative to the unfairness of the Bayes rule is less meaningful. Indeed, intuitively, if a prediction function $f : \bbR^p \times [K] \to \bbR$ introduces some group-wise disparities, then $c \cdot f$ for $c \gg 1$ should be even further \emph{amplifying} these disparities.
Yet, for all $c > 0$ we have $\class{U}_{\KS/\TV}(c\cdot f) = \class{U}_{\KS/\TV}(f)$, while the introduced notion of unfairness satisfies $\class{U}(c\cdot f) = c^2 \class{U}(f)$ for all $c > 0$.}
Due to completely different geometries induced by $\risk$ in the space of functions and by $\class{U}_{\KS {/} \TV}$ in the space of distributions, precise theoretical study of such formulations is notoriously complicated if possible.

\subsection{Optimal transport and fair regression}
    The use of optimal transport tools in the study of fairness is relatively recent. Initially, contributions in this direction were mainly dealing with the problem of binary classification~\citep{gordaliza2019obtaining,jiang2019wasserstein}.
    Later on, the tools of the optimal transport theory migrated to the setup of fair regression~\citep{chiappa2020general,chzhen2020fair,gouic2020price}.
The main theoretical motivation to consider $\class{U}$ instead of the KS and TV unfairnesses lies in the following recent result.
\begin{theorem}[\cite{gouic2020price,chzhen2020fair}]
    \label{thm:basic}
    Let Assumption~\ref{as:atomless} be satisfied, then
    \begin{align}
        \label{eq:motivation1}
         \min\enscond{\risk(f)}{\big(f(\bsX, S) \mid S = s\big) \leftstackrel{d}{=} \big(f(\bsX, S) \mid S = s'\big)\,\,\forall s, s' \in [K]} = \class{U}(f^*)\enspace.
    \end{align}
    Moreover, the distribution of the minimizer of the problem on the \emph{l.h.s.} is given by
    \begin{align*}
        \argmin_{\nu \in \class{P}_2(\bbR)} \sum_{s = 1}^K w_s \sW_2^2\left(\Law(f^*(\bsX, S) \mid S {=} s),\, \nu\right)\enspace.
    \end{align*}
\end{theorem}
An important consequence of Theorem~\ref{thm:basic} is that it puts the risk $\risk$ and the unfairness $\class{U}$ -- two conflicting quantities -- on the same scale. In particular, it allows to measure both fairness and risk using the same unit measurements, hence, study the trade-off between the two.
In order to build our framework, we remark that since $\sW_2$ is a metric then the problem on the \emph{l.h.s.} of Eq.~\eqref{eq:motivation1} can be equivalently written as $\min\enscond{\risk(f)}{\mathcal{U}(f) \leq 0 \times \mathcal{U}(f^*)}$. Moreover, one can observe that the regression function $f^* \in \min\enscond{\risk(f)}{\mathcal{U}(f) \leq 1 \times \mathcal{U}(f^*)}$.
Thus, a natural relaxation of the above formulation
is the introduced notion of \emph{$\alpha$-relative improvement}, which interpolates between the exactly fair predictor $f^*_0$ and the regression function $f^*_1 \equiv f^*$. In this retrospect, the result of~\cite{gouic2020price,chzhen2020fair} provides characterization of $f^*_0$ but says nothing about the whole family of oracle $\alpha$-RI $\{f^*_\alpha\}_{\alpha \in [0, 1]}$.
\section{Oracle \texorpdfstring{$\alpha$}{Lg}-relative improvement}
\label{SUBSEC:GENERAL}
This section is devoted to the study of the $\alpha$-relative improvement $f^*_\alpha$ on population level, that is, in this section we study 
\begin{align}
    \label{eq:alpha_improvement}
    f^*_{\alpha} \in \argmin\enscond{\risk(f)}{\class{U}(f) \leq \alpha\,\class{U}(f^*)}\enspace,\qquad\forall\alpha\in[0,1]
    \enspace.
\end{align}
The next result establishes a closed form solution to the minimization Problem~\eqref{eq:alpha_improvement} under Assumption~\ref{as:atomless} for any value of $\alpha \in [0, 1]$.
\begin{proposition}
    \label{prop:optimal_alpha}
    Let Assumption~\ref{as:atomless} be satisfied, then for all $\alpha \in [0, 1]$ and all $(\bsx, s) \in \bbR^p \times [K]$ (up to a set of null measure) it holds that
    \begin{align*}
        f_{\alpha}^*(\bsx, s)
        &=
        \sqrt{\alpha} f^*(\bsx, s) + \big(1{-}\sqrt{\alpha}\big) \sum_{s' = 1}^K w_{s'} F_{\nu^*_{s'}}^{-1} \circ F_{\nu^*_s} \circ f^*(\bsx, s)\nonumber\\
        &= \sqrt{\alpha}f^*_{1}(\bsx, s) + (1 {-} \sqrt{\alpha})f_0^*(\bsx, s)\enspace.
    \end{align*}
\end{proposition}
Recall that $f^* = f^*_1$, hence the $\alpha$-relative improvement $f^*_\alpha$ is the point-wise convex combination of exactly fair prediction $f^*_0$ and the regression function $f^*_1$.
Besides, setting $\alpha = 0$ we recover the result of~\cite{chzhen2020fair,gouic2020price} as a particular case of our framework.
\begin{figure}[!t]
\centering
\includegraphics[width=0.48\textwidth]{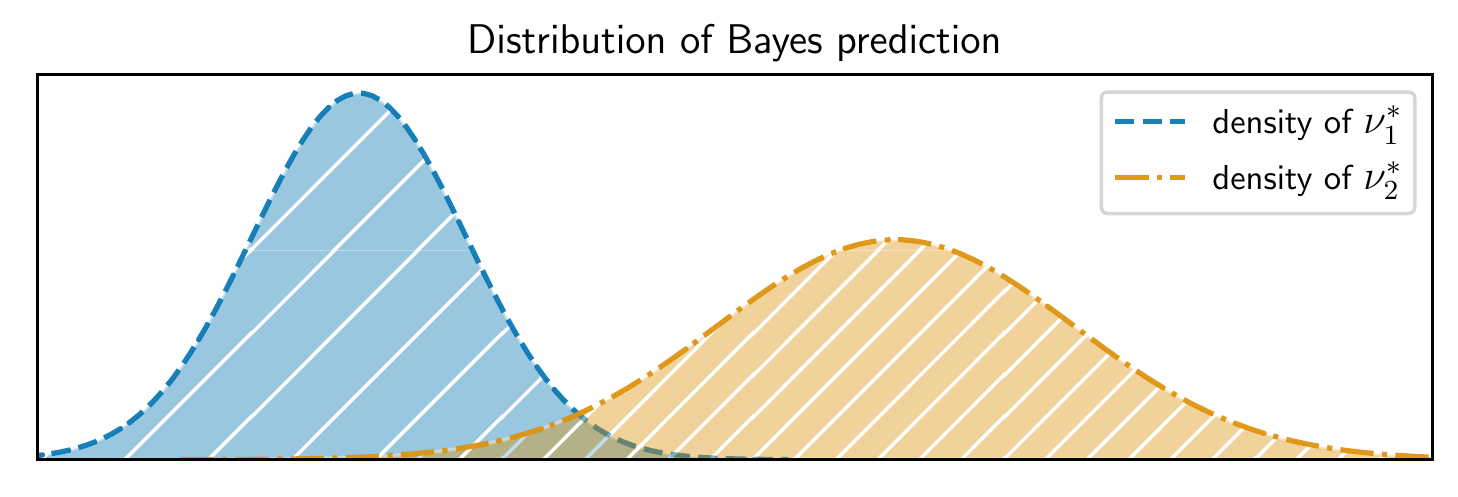}\hfill
\includegraphics[width=0.48\textwidth]{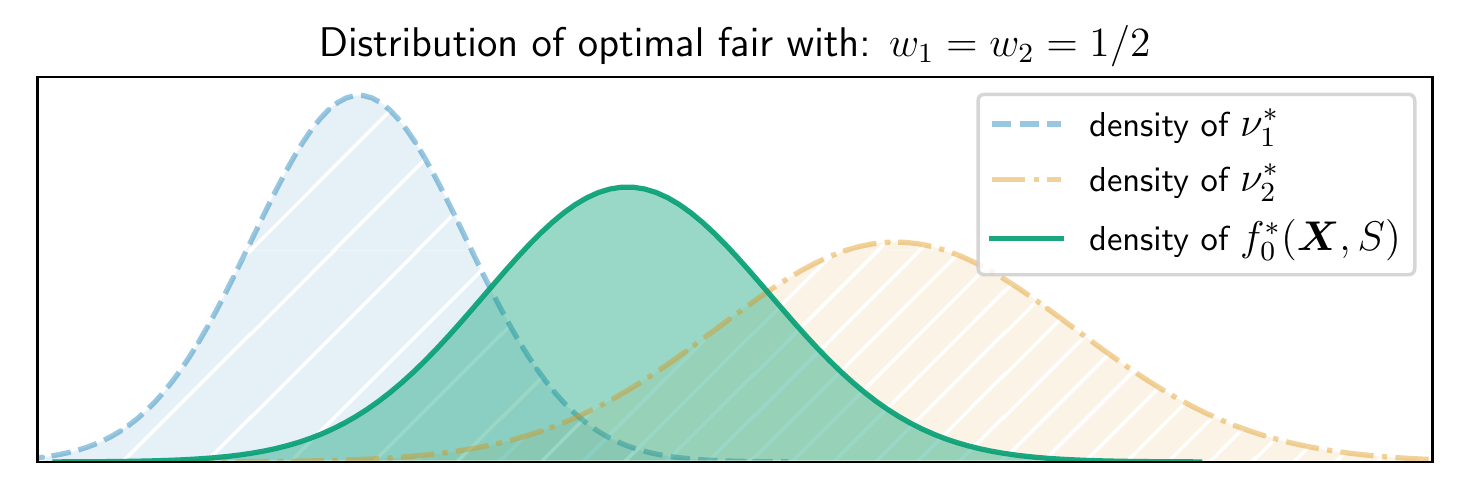}\\
\includegraphics[width=0.48\textwidth]{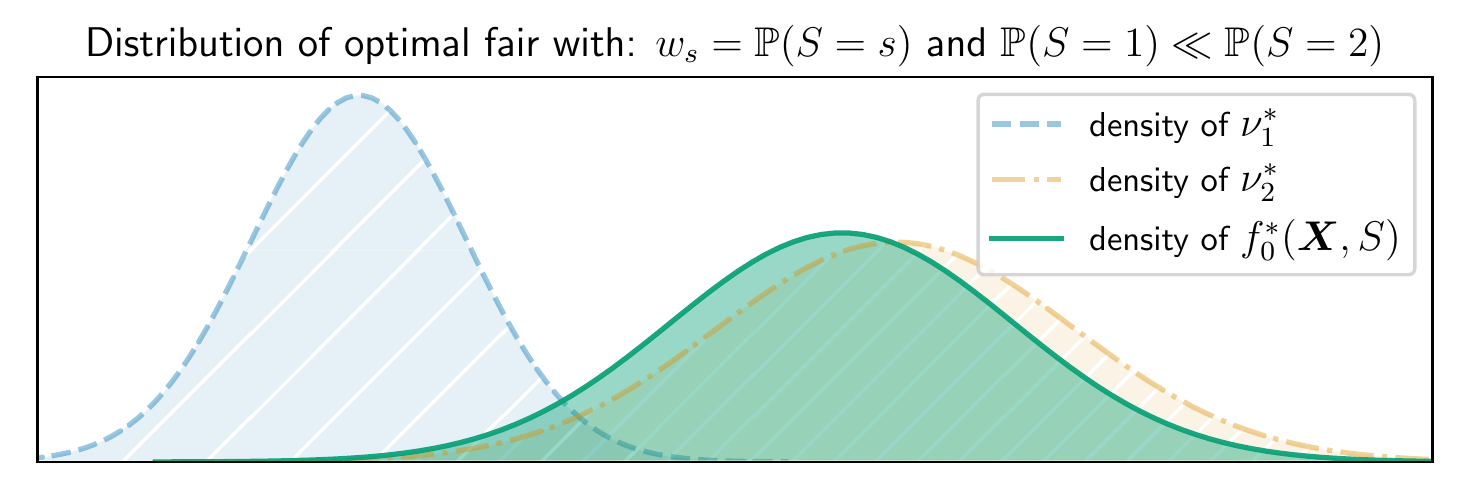}\hfill
\includegraphics[width=0.48\textwidth]{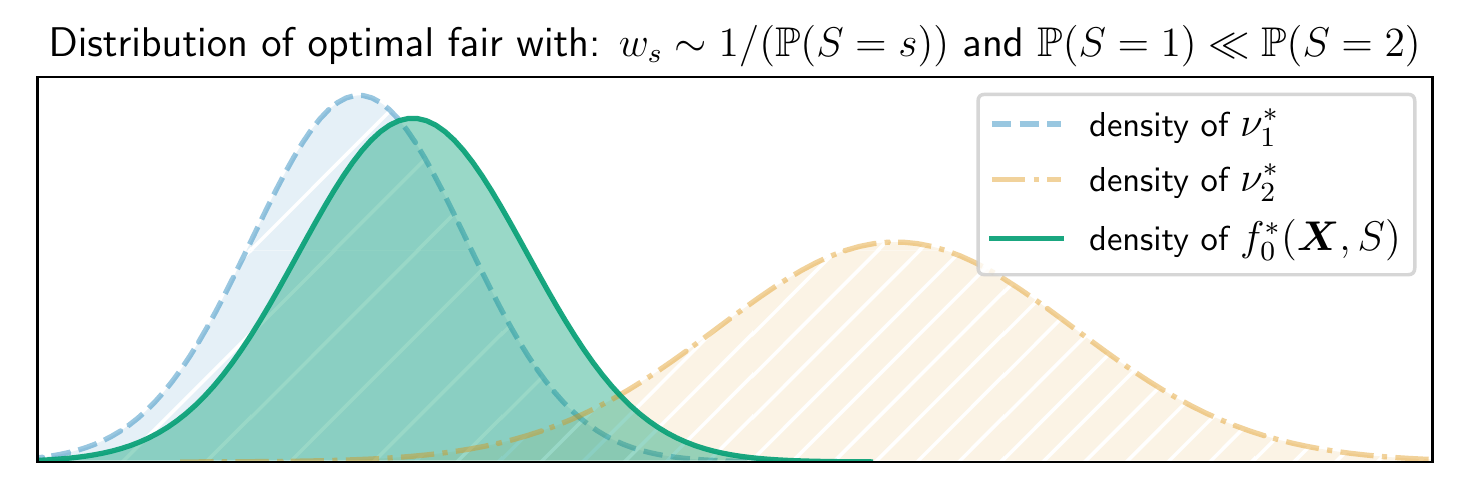}
\caption{Impact of the weights $\bsw \in \Delta^{K-1}$ on the distribution of $f^*_{0}$}
\label{fig:weights}
\end{figure}
{\color{red}Let Assumption~\ref{as:atomless} be satisfied, then} the set of oracle $\alpha$-RI $\{f^*_{\alpha}\}_{\alpha \in [0, 1]}$ satisfies the following properties.
    \begin{enumerate}
        \item{\bf Risk and fairness monotonicity:} if $\alpha \leq \alpha'$, then $\risk(f_{\alpha}^*) \geq \risk(f^*_{\alpha'})$ and $\class{U}(f^*_{\alpha}) \leq \class{U}(f^*_{\alpha'})$.
        \item{\bf Point-wise convexity:} for all $\alpha, \alpha' \in [0, 1]$ and all $\tau \in [0, 1]$ it holds that $ \tau f^*_{\alpha} + (1 - \tau)f^*_{\alpha'} \in \{f^*_{\alpha}\}_{\alpha \in [0, 1]}$. Moreover $\tau f^*_{\alpha} + (1 - \tau)f^*_{\alpha'} = f^*_{\bar{\alpha}}$ with $\bar\alpha = (\tau \sqrt{\alpha} + (1 - \tau) \sqrt{\alpha'})^2$.
        \item{\bf Order preservation:} for all $s \in [K], \bsx, \bsx' \in \bbR^p$, if $f^*(\bsx, s) \geq f^*(\bsx', s)$, then for all $\alpha \in [0, 1]$ it holds that $f^*_{\alpha}(\bsx, s) \geq f^*_{\alpha}(\bsx', s)$.
        {\color{red}
        \item{\bf Average stability:} let $\mu^*_s$ be the expected value of $\nu^*_s$.
        For all $\alpha \in [0, 1]$ and all $\bsw' = (w_1', \ldots, w_K')^\top \in \Delta^{K - 1}$ it holds that
        \begin{align*}
            \sum_{s = 1}^K w_s'\Exp[f^*_{\alpha}(\bsX, S) \mid S = s] = \sqrt{\alpha}\parent{\sum_{s = 1}^Kw_{s}' \mu^*_s} + (1 {-} \sqrt{\alpha})\parent{\sum_{s = 1}^Kw_{s}\mu_{s}^*}\enspace.
        \end{align*}
        In particular, setting $\bsw' = \bsw$, we get for all $\alpha \in [0, 1]$ the average stability:
        \begin{align*}
            \sum_{s = 1}^K w_s\Exp[f^*_{\alpha}(\bsX, S) \mid S = s] = \sum_{s = 1}^K w_s\Exp[f^*(\bsX, S) \mid S = s]\enspace.
        \end{align*}
        }
    \end{enumerate}
The first property is intuitive and does not require the result of Proposition~\ref{prop:optimal_alpha}.
The second property can be directly derived using the expression of $f^*_{\alpha}$ and it describes additional algebraic structure of the family $\{f^*_{\alpha}\}_{\alpha \in [0, 1]}$ .
The third group-wise order preserving property of $f^*_{\alpha}$ is particularly attractive.
Its proof is straightforward after the observation that $F_{\nu_{s}^*}$ and $\sum_{s' = 1}^Kw_{s'}F^{-1}_{\nu^*_{s'}}$ are non-decreasing functions and the fact that the composition of two non-decreasing functions is non-decreasing.
For the special case of $\alpha = 0$, this observation has already been made in~\citep{chzhen2020fair} and a practical algorithm that follows the group-wise order preservation property was proposed by~\cite{plevcko2019fair}.
{\color{red}In the context of classification~\cite{lipton2018does} refer to this property as ``rational ordering''.}
In words, this property says: given any two individuals $\bsx, \bsx' \in \bbR^p$ from the same sensitive group $s \in [K]$, if the optimal prediction $f^*(\bsx, s)$ for $\bsx$ is larger than that for $\bsx'$, then across all levels $\alpha$ of fairness parameter the oracle $\alpha$-RI $f^*_{\alpha}$ is not changing this order. {\color{red} The last property of average stability admits an interesting interpretation. Let us interpret $f : \bbR^p \times \class{S} \to \bbR_+$ as a salary assignment function (making it naturally positive). In that case setting $\bsw' = \bsw = (\sfrac{1}{K}, \ldots, \sfrac{1}{K})^\top$, we can interpret $\Exp[f^*(\bsX, S) \mid S = s]$ as the average amount of money allocated for salaries within group $s \in \class{S}$. Thus, in this context, the average stability property states that by enforcing fairness improvement property with equally distributed weights across groups $\bsw = (\sfrac{1}{K}, \ldots, \sfrac{1}{K})^\top$, we do not need to augment the budget allocated for these salaries.
The proof of the last property is also rather straightforward: for every $s \in \class{S}$ define the mean of the distribution $\nu^*_s$ as $m_s^* \eqdef \Exp[f^*(\bsX, S) \mid S = s]$, then
\begin{align*}
    \Exp[f_{\alpha}^*(\bsX, S) \mid S = s]
        &=
        \sqrt{\alpha} m^*_s + \big(1{-}\sqrt{\alpha}\big) \sum_{s' = 1}^K w_{s'} \Exp \left [ F_{\nu^*_{s'}}^{-1} \circ F_{\nu^*_s} \circ f^*(\bsX, S) \mid S = s \right]\enspace.
\end{align*}
To conclude, it suffices to notice that under Assumption~\ref{as:atomless} the random variable $(F_{\nu^*_{s'}}^{-1} \circ F_{\nu^*_s} \circ f^*(\bsX, S) \mid S = s)$ is distributed according to $\nu^*_{s'}$.

}

{\color{red}
\subsection{Influence of the choice of weights}\label{sec:weights_choice} Note that the considered framework permits the statistician to pick different $\bsw \in \Delta^{K-1}$.
To study this additional level of flexibility and freedom, we provide in this section the intuition for three natural choices:
\begin{itemize}[noitemsep,topsep=0.2ex]
     \item \textbf{Proportional}: $w_s = \Prob(S = s)$;
     \item \textbf{Inverse}: $w_s \sim 1/\Prob(S = s)$;
     \item \textbf{Equal weights}: $w_s = \tfrac{1}{K}$.
\end{itemize}
Since all of the $\alpha$-RI prediction functions can be obtained as convex combination of $f^*_0$ and the Bayes rule, it is sufficient to understand the underlying principle behind $f^*_0$. We focus on the case of two groups ($K=2$), with the first group representing the minority $(\Prob(S = 1) \ll \Prob(S=2))$.
As already mentioned, the Bayes optimal prediction $f^*$ induces two distributions $\nu^*_1$ and $\nu^*_2$ (\eg distribution of salaries for minority and majority sub-populations). We schematically illustrate these distributions on top left plot of Fig.~\ref{fig:weights}.

The choice of $w_s = \Prob(S = s)$ (bottom left on Fig.~\ref{fig:weights}) leads to equalization of minority to majority---that is, the situation of the majority group is modified only slightly, while the minority gets pushed towards the majority; the choice $w_s \sim 1 / \Prob(S = s)$ (bottom right on Fig.~\ref{fig:weights}) leads to an inverse situation---majority is equalized to minority; the choice $w_s = 1/K$ (top right on Fig.~\ref{fig:weights}) leads to a ``middle ground'' compromise, where both the majority and the minority are pushed toward their barycenter.
The actual choice of the weights clearly depends on the given application and the social aspects of thereof. We hope that the provided intuition in conjunction with domain expertise can help the practitioner to make a more informed choice for the weights.

}

\subsection{An abstract geometric lemma}
The proof of Proposition~\ref{prop:optimal_alpha} relies on an abstract geometric result, Lemma~\ref{lem:geometric_general}, which might be interesting on its own. First, let us introduce the following definition, which asks for existence of finitely supported barycenters in a metric space $(\class{X}, d)$.
\begin{definition}[Barycenter property]
    \label{def:barycenter_space}
    We say that a metric space $(\class{X}, d)$ satisfies the barycenter property if for any weights $\bsw \in \Delta^{K - 1}$ and tuple $\bsa = (a_1, \ldots, a_K) \in \class{X}^K$ there exists a barycenter
    \begin{align*}
        C_{\bsa_{\bsw}} \in \argmin_{C \in \class{X}} \sum_{s = 1}^K w_s d^2(a_s, C)\enspace.
    \end{align*}
    Moreover, for any tuple $\bsa = (a_1, \ldots, a_K) \in \class{X}^K$ we denote\footnote{When there is no ambiguity in the weights $\bsw$ we simply write $C_{\bsa}$.} by $C_{\bsa_{\bsw}}$ a barycenter of $\bsa$ weighted by $\bsw \in \Delta^{K - 1}$.
\end{definition}

\begin{figure}[!t]
\centering
\includegraphics[width=0.32\textwidth]{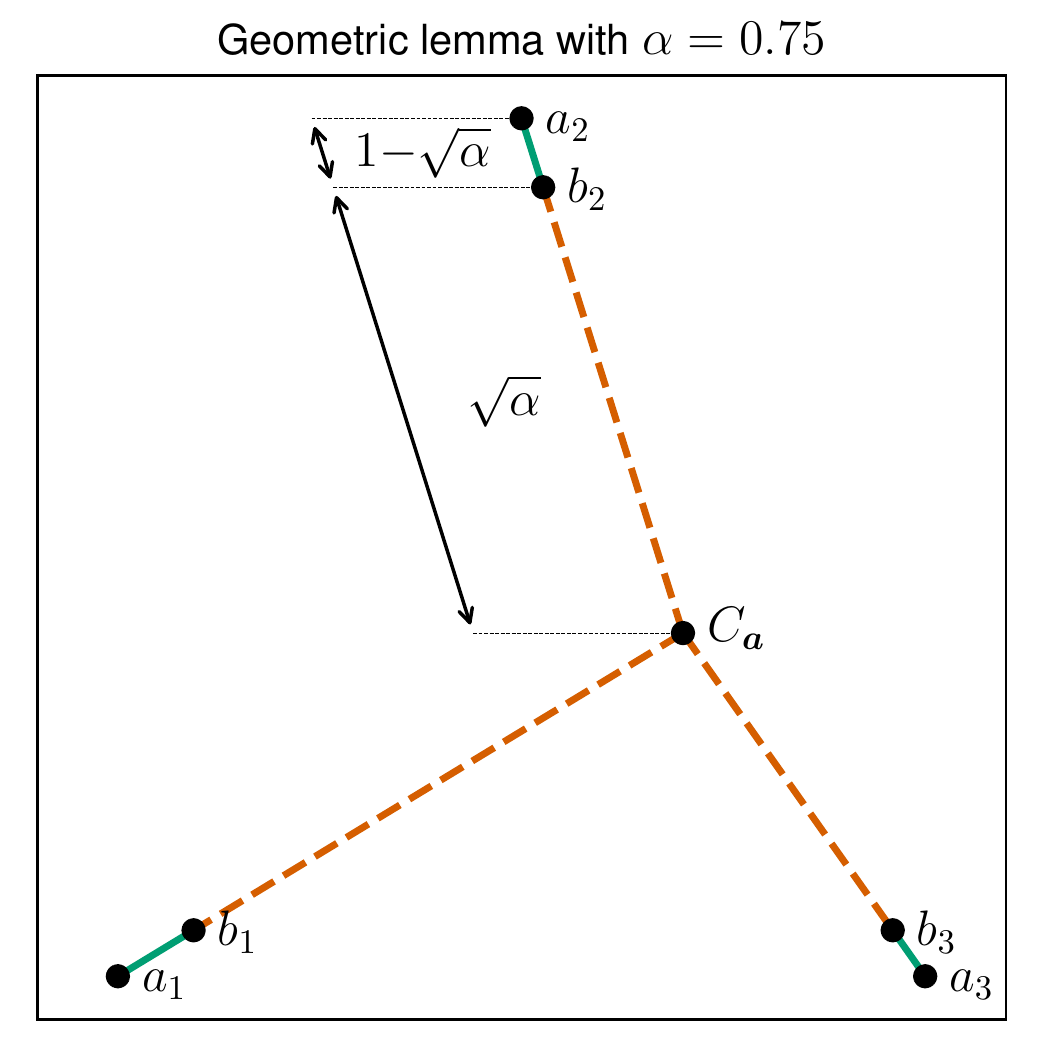}\hfill
\includegraphics[width=0.32\textwidth]{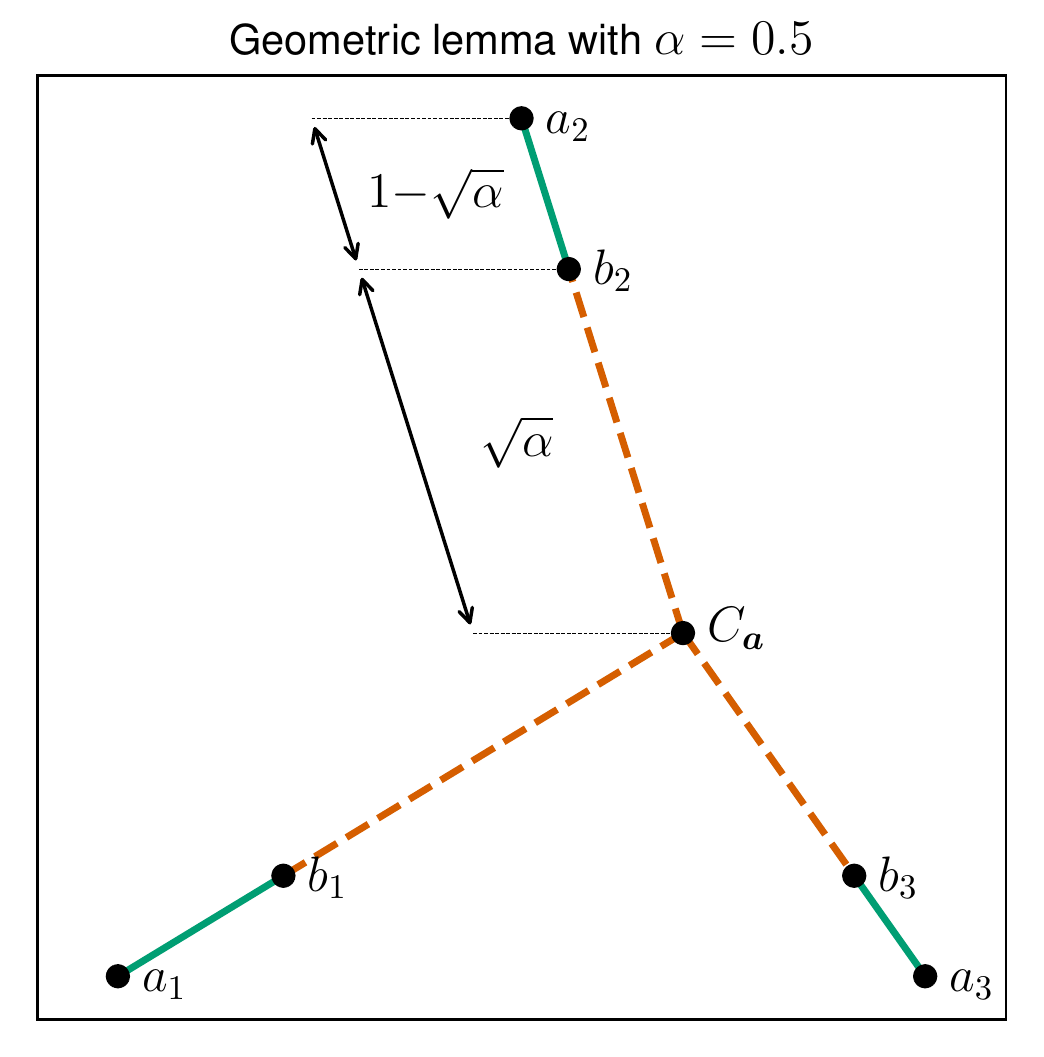}\hfill
\includegraphics[width=0.32\textwidth]{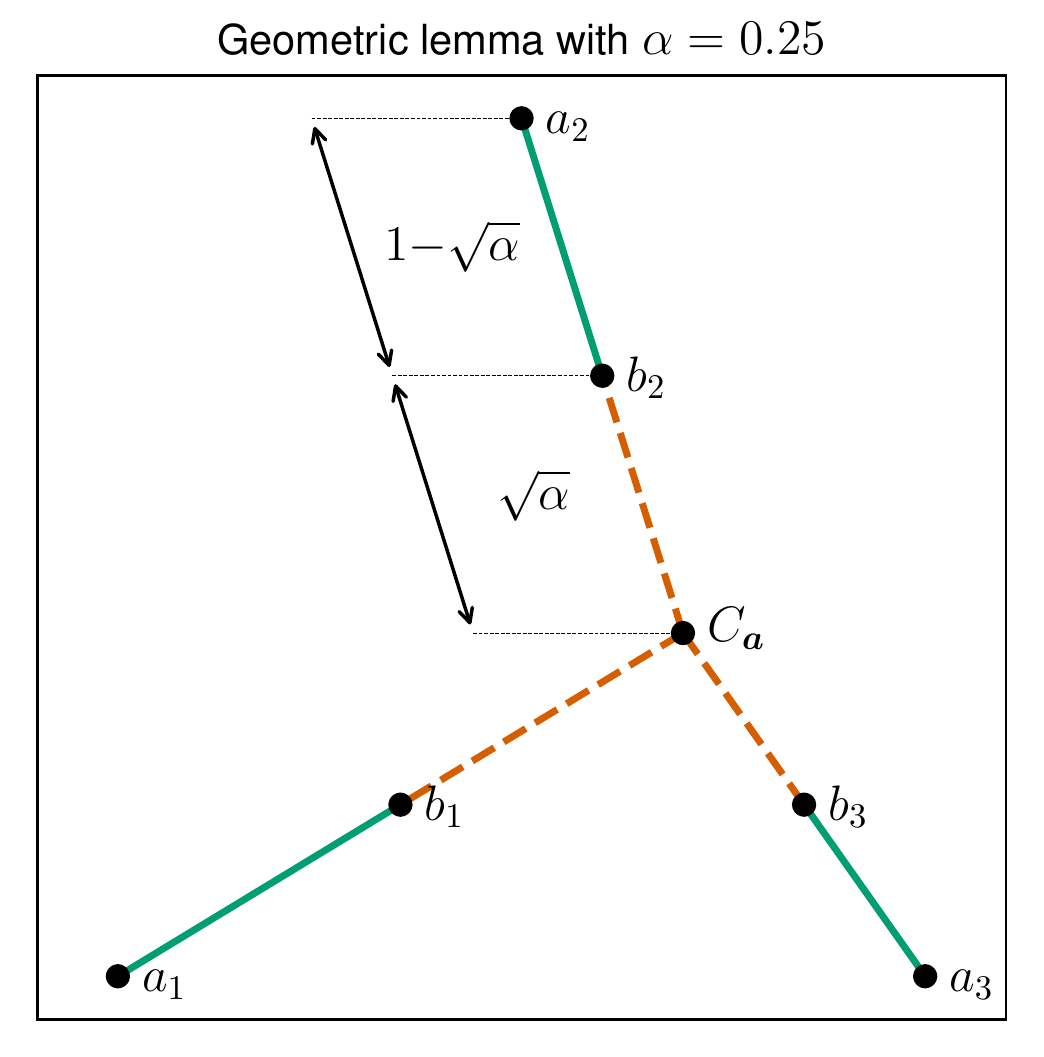}
\caption{Illustration of Lemma~\ref{lem:geometric_general} for $(\class{X}, d) = (\bbR^2, \|\cdot\|_2)$ and $\alpha \in \{0.25, 0.5, 0.75\}$. The initial points $a_1, a_2, a_3$ are the vertices of an isosceles triangle. The weights are set as follows: $w_1 = 0.1$, $w_2 = 0.4$ and $w_3 = 0.5$.}
\label{fig:geom}
\end{figure}
{
\begin{lemma}[Abstract geometric lemma]\label{lem:geometric_general}
    Let $(\class{X}, d)$ be a metric space satisfying the $q$-barycenter property.
    Let $\bsa = (a_1, \ldots, a_K) \in \class{X}^K$, $\bsw = (w_1, \ldots, w_K)^\top \in \Delta^{K - 1}$ and let $C_{\bsa}$ be a barycenter of $\bsa$ with respect to weights $\bsw$.
    For a fixed $\alpha \in [0, 1]$ assume that there exists $\bsb = (b_1, \ldots, b_K) \in \class{X}^K$ which satisfies
    \begin{alignat}{3}
        &d(a_s, C_{\bsa}) = d(a_s, b_s) + d(b_s, C_{\bsa})\enspace, \qquad &&s = 1,\dots,K\enspace, \tag{$P_1$}\label{eq:prop_1}\\
        &d(b_s, a_s) = \big(1 {-} \alpha^{\sfrac{1}{2}}\big)d(a_s, C_{\bsa})\enspace, \qquad &&s = 1,\dots,K\enspace. \tag{$P_2$}\label{eq:prop_2}
    \end{alignat}
    Then, $\bsb$ is a solution of
    \begin{align}
        \label{eq:geom_lemma_the_problem}
        \inf_{\bsb \in \class{X}^K}\enscond{\sum_{s = 1}^K w_s d^2(b_s, a_s)}{\sum_{s = 1}^K w_s d^2(b_s, C_{\bsb}) \leq {\alpha} \sum_{s = 1}^K w_s d^2(a_s, C_{\bsa}) }\enspace.
    \end{align}
\end{lemma}
}
\begin{remark}
    Property~(\ref{eq:prop_1}) essentially requires that each $b_i$ lies on the geodesic between $a_i$ and $C_{\bsa}$ while Property~(\ref{eq:prop_2}) specifies the location of $b_i$ on this geodesic: $b_i$ should be $(1 {-} \alpha^{\sfrac{1}{q}})$ times closer to $a_i$, than $C_{\bsa}$ to $a_i$. An illustration provided on Figure~\ref{fig:geom} describes these properties in Euclidean geometry. For general case, the straight lines should be replaced by geodesics. {\color{red}In Appendix~\ref{sec:q_losses} we extend our framework to $\ell_q$-risks and provides an extension of Lemma~\ref{lem:geometric_general} to handle losses other than $\ell_2$.}
\end{remark}
The setting of Lemma~\ref{lem:geometric_general} is quite general and only requires existence of barycenters (also known as the Fr\'echet means) for any finite weighted combination of points in accordance with Definition~\ref{def:barycenter_space}.
For our purposes, Lemma~\ref{lem:geometric_general} will be applied to the metric space $(\class{X}, d) = (\class{P}_2(\bbR), \sW_2)$. We refer to~\citep{agueh2011barycenters,le2017existence} who investigate and prove the existence of Wasserstein barycenters of random probabilities defined on geodesic spaces.

\begin{proof}[Proof of Lemma~\ref{lem:geometric_general}]
      Fix some $\bsa = (a_1, \ldots, a_K) \in \class{X}^K$,  $\bsw = (w_1, \ldots, w_K)^\top \in \Delta^{K - 1}$ and let $C_{\bsa}$ be a barycenter of $\bsa$ with respect to weights $\bsw$. Fix $\alpha \in [0, 1]$ and any $\bsb = (b_1, \ldots, b_K) \in \class{X}^K$ which satisfies properties~\eqref{eq:prop_1}--\eqref{eq:prop_2}.
    Let $\bsb_k = (b_1^k, \ldots, b_K^k) \in \class{X}^K$ be a minimizing sequence of the problem~\eqref{eq:geom_lemma_the_problem} and for any $\bsb' = (b_1',\ldots,b'_K) \in \class{X}^K$ denote by $G(\bsb') = \sum_{s = 1}^K w_s d^2(b_s', a_s)$ the objective function of the problem~\eqref{eq:geom_lemma_the_problem}.
    Then, by the definition of a minimizing sequence, the following two properties hold
    \begin{align}
        &\lim_{k \rightarrow \infty} G(\bsb_k) = \inf_{\bsb \in \class{X}^K}\enscond{G(\bsb)}{\sum_{s = 1}^K w_s d^2(b_s, C_{\bsb}) \leq {\alpha} \sum_{s = 1}^K w_s d^2(a_s, C_{\bsa}) }\enspace,\\
        &\sum_{s = 1}^K w_s d^2(b^k_s, C_{\bsb_k}) \leq \alpha\sum_{s = 1}^K w_s d^2(a_s, C_{\bsa})\enspace,\quad\forall k\in \bbN\enspace.\label{eq:geom_proof_contradiction_1.2}
    \end{align}
    Furthermore, using properties~\eqref{eq:prop_1}--\eqref{eq:prop_2} we deduce that
    \begin{align*}
        \sum_{s = 1}^K w_s d^2(b_s, C_{\bsb}) \stackrel{(a)}{=} \sum_{s = 1}^K w_s d^2(b_s, C_{\bsa})
        &\stackrel{\eqref{eq:prop_1}}{=} \sum_{s = 1}^K w_s \left(d(a_s, C_{\bsa}) - d(a_s, b_s)\right)^2\\ &\stackrel{\eqref{eq:prop_2}}{=} \alpha \sum_{s = 1}^K w_s d^2(a_s, C_{\bsa})\enspace,
    \end{align*}
    where $(a)$ follows from Lemma~\ref{lem:barycenters_coincide} in appendix.
    Therefore, $\bsb = (b_1, \ldots, b_s) \in \class{X}^K$ is feasible for the  problem~\eqref{eq:geom_lemma_the_problem}.
    By Lemma~\ref{lem:fancy_triangle} it holds for all $k \in \bbN$ that
    \begin{align*}
        \left\{\sum_{s = 1}^K w_s d^2(a_s, C_{\bsb_k})\right\}^{\sfrac{1}{2}}
        &\leq
        \left\{\sum_{s = 1}^K w_s d^2(a_s, b^k_s)\right\}^{\sfrac{1}{2}} + \left\{\sum_{s = 1}^K w_s d^2(b^k_s, C_{\bsb_k})\right\}^{\sfrac{1}{2}}\\
        &=
        {G^{\sfrac{1}{2}}(\bsb_k)} +  \left\{\sum_{s = 1}^K w_s d^2(b^k_s, C_{\bsb_k})\right\}^{\sfrac{1}{2}}\enspace.
    \end{align*}
    We continue using the definition of $C_{\bsa}$ and Eq.~\eqref{eq:geom_proof_contradiction_1.2} to obtain for all $k \in \bbN$
    \begin{align*}
        \left\{\sum_{s = 1}^K w_s d^2(a_s, C_{\bsa})\right\}^{\sfrac{1}{2}}
        \leq
        \left\{\sum_{s = 1}^K w_s d^2(a_s, C_{\bsb_k})\right\}^{\sfrac{1}{2}}
        \leq
        { G^{\sfrac{1}{2}}(\bsb_k)} {+} {\alpha}^{\sfrac{1}{2}}\left\{\sum_{s = 1}^K w_s d^2(a_s, C_{\bsa})\right\}^{\sfrac{1}{2}}\enspace,
    \end{align*}
    which after rearranging implies that
    \begin{align*}
        (1 {-} {\alpha}^{\sfrac{1}{2}})\left\{\sum_{s = 1}^K w_s d^2(a_s, C_{\bsa})\right\}^{\sfrac{1}{2}} \leq { G^{\sfrac{1}{2}}(\bsb_k)},\qquad\forall k \in \bbN\enspace.
    \end{align*}
    Finally, using property~\eqref{eq:prop_2} we deduce that $ G(\bsb) \leq  G(\bsb_k)$ for all $ k \in \bbN$.
    Recall that we have already shown that $\bsb$ is feasible for the problem~\eqref{eq:geom_lemma_the_problem}, hence taking the limit \wrt to $k$ concludes the proof of Lemma~\ref{lem:geometric_general}.
\end{proof}
The complete proof of Proposition~\ref{prop:optimal_alpha} is omitted in the main body.
We only provide a short intuition.
\begin{proof}[Sketch of the proof]
The idea of the proof is to apply Lemma~\ref{lem:geometric_general} with $(\mathcal{X}, d) = (\mathcal{P}_2(\bbR), \sW_2)$ and with measures $a_s  \eqdef \nu_s^*$, which belong to $\class{P}_2(\bbR)$ due to Assumption~\ref{as:atomless}.
    Then, we need to construct {measures} $\bsb = (b_1, \ldots, b_K)^\top \in \class{P}_2^K(\bbR)$, which satisfy the properties~\eqref{eq:prop_1}--\eqref{eq:prop_2}.
    To this end, let $\gamma_s$ be the (constant-speed) geodesic between $a_s$ and $C_{\bsa}$ \ie $\gamma_s(0) = a_s$, $\gamma_s(1)=C_{\bsa}$.
    We define $b_s \eqdef \gamma_s(1 {-} \sqrt{\alpha})$ for $s \in[K]$, similarly to the intuition provided by Figure~\ref{fig:geom}.
    One can verify that that $\bsb = (b_s)_{s \in [K]}$ satisfies \eqref{eq:prop_1} and \eqref{eq:prop_2}. Then, by Lemma~\ref{lem:geometric_general} we know that $\bsb$ solves the minimization problem in Eq.~\eqref{eq:geom_lemma_the_problem}. For the final part of the proof we propagate the optimality of $\bsb$ in the space of distributions to the optimality of $f^*_{\alpha}$ in the space of predictions using the assumption that $\bsa$ admits a density and an explicit construction of the geodesic $\gamma_s$.
\end{proof}


\subsection{Risk-fairness trade-off on the population level}
The next key result of our framework establishes the risk-fairness trade-off provided by the parameter $\alpha \in [0, 1]$ on the population level.
In particular, it establishes a simple user-friendly relation between the risk and unfairness of $\alpha$-relative improvement. Note that such a result is not available neither for $\class{U}_{\TV}$ nor for $\class{U}_{\KS}$, due to fundamentally different geometries of the squared risk and the aforementioned distances.
\begin{lemma}
    \label{lem:distance_fair_and_almost}
     Let Assumption~\ref{as:atomless} be satisfied, then for any $\alpha \in [0, 1]$ it holds that
    \begin{align}
        \label{eq:risk_alpha}
        \risk(f^*_{\alpha}) = (1 {-} \sqrt{\alpha})^2\risk(f^*_0) = (1 {-} \sqrt{\alpha})^2\class{U}(f^*)\enspace.
    \end{align}
\end{lemma}
\begin{proof}
    Proposition~\ref{prop:optimal_alpha} gives the following explicit expression for the best $\alpha$-improvement of $f^*$:
    \begin{align*}
        f_{\alpha}^*(\bsx, s) = \sqrt{\alpha}f^*(\bsx, s) + (1{-}\sqrt{\alpha})f_0^*(\bsx, s)\enspace.
    \end{align*}
    Plugging it in the risk gives
    \begin{align*}
        \risk(f_{\alpha}^*) = \lVert f_{\alpha}^* - f^*\rVert_2^2 = (1{-}\sqrt{\alpha})^2 \lVert f_0^* - f^*\rVert_2^2 = (1{-}\sqrt{\alpha})^2 \risk(f_0^*)\enspace.
    \end{align*}
    This proves the first equality. Given the definition of $f_0^*$, the second equality is exactly the result stated in Theorem~\ref{thm:basic}.
\end{proof}

Recall that thanks to Theorem~\ref{thm:basic} we have $\risk(f^*_0) = \class{U}(f^*)$.
Hence, the $\alpha$-relative improvement $f^*_{\alpha}$ enjoys the following two properties
\begin{align*}
    \risk(f^*_{\alpha}) = (1 {-} \sqrt{\alpha})^2\risk(f^*_{0})\quad\text{and}\quad \class{U}(f^*_{\alpha}) = \alpha\,\class{U}(f^*)\enspace.
\end{align*}
For instance, if $\alpha = 1/2$, that is, we want to half the unfairness of $f^*$, it incurs the risk which is equal to $\approx 8.5 \%$ of the risk of exactly fair predictor $f^*_0$.
We illustrate this general behaviour in Figure~\ref{fig:illustration1} (Section~\ref{sec:problem_statement_and_contributions}), where the risk and the unfairness of $f^*_{\alpha}$ are shown for different levels of $\class{U}(f^*)$.
A striking observation we can make from this plot is that, letting $\alpha$ vary between $0$ and $1$, the risk of $f^*_{\alpha}$ growth rapidly in the vicinity of zero, while it behaves almost linearly in a large neighbourhood of one.
That is, one can find a prediction $f$ whose unfairness $\class{U}(f)$ is smaller than that of $f^*$ by a constant multiplicative factor, without a large increase in risk.


\subsection{Pareto efficiency: a systematic way to select \texorpdfstring{$\alpha$}{Lg}}
\label{sec:pareto}
\begin{figure}[!t]
\centering
\includegraphics[width=0.32\textwidth]{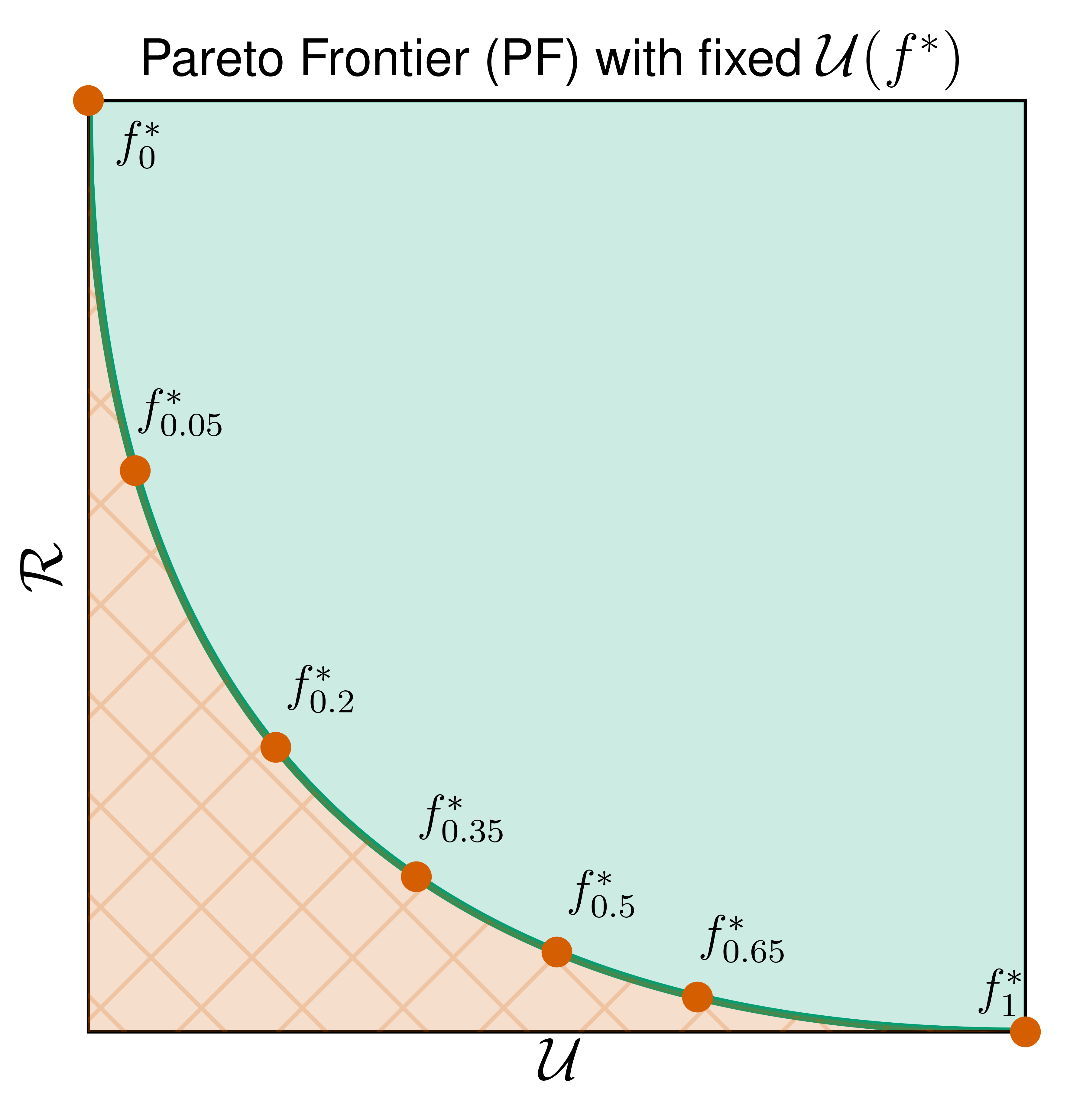}
\includegraphics[width=0.32\textwidth]{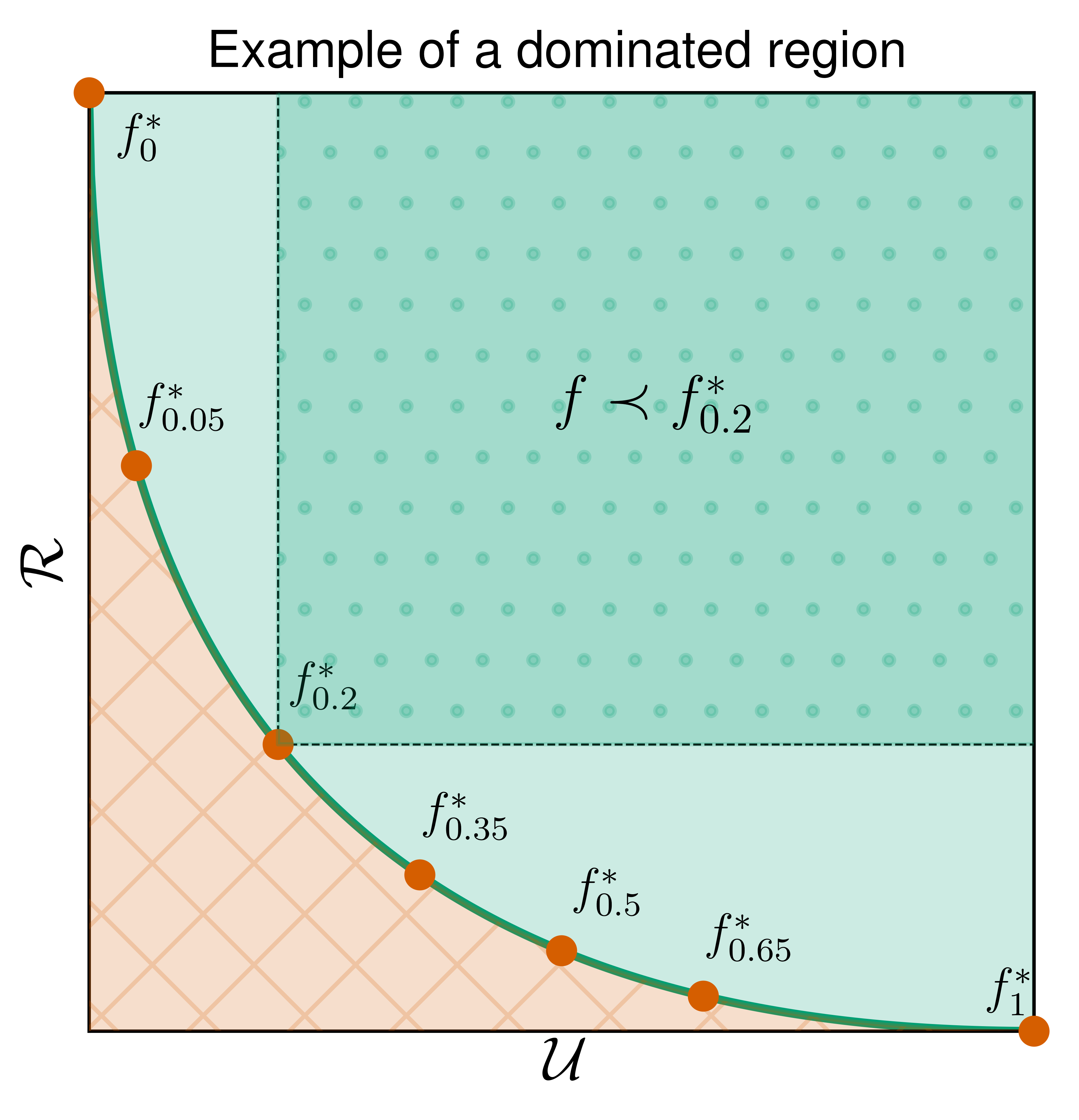}
\includegraphics[width=0.32\textwidth]{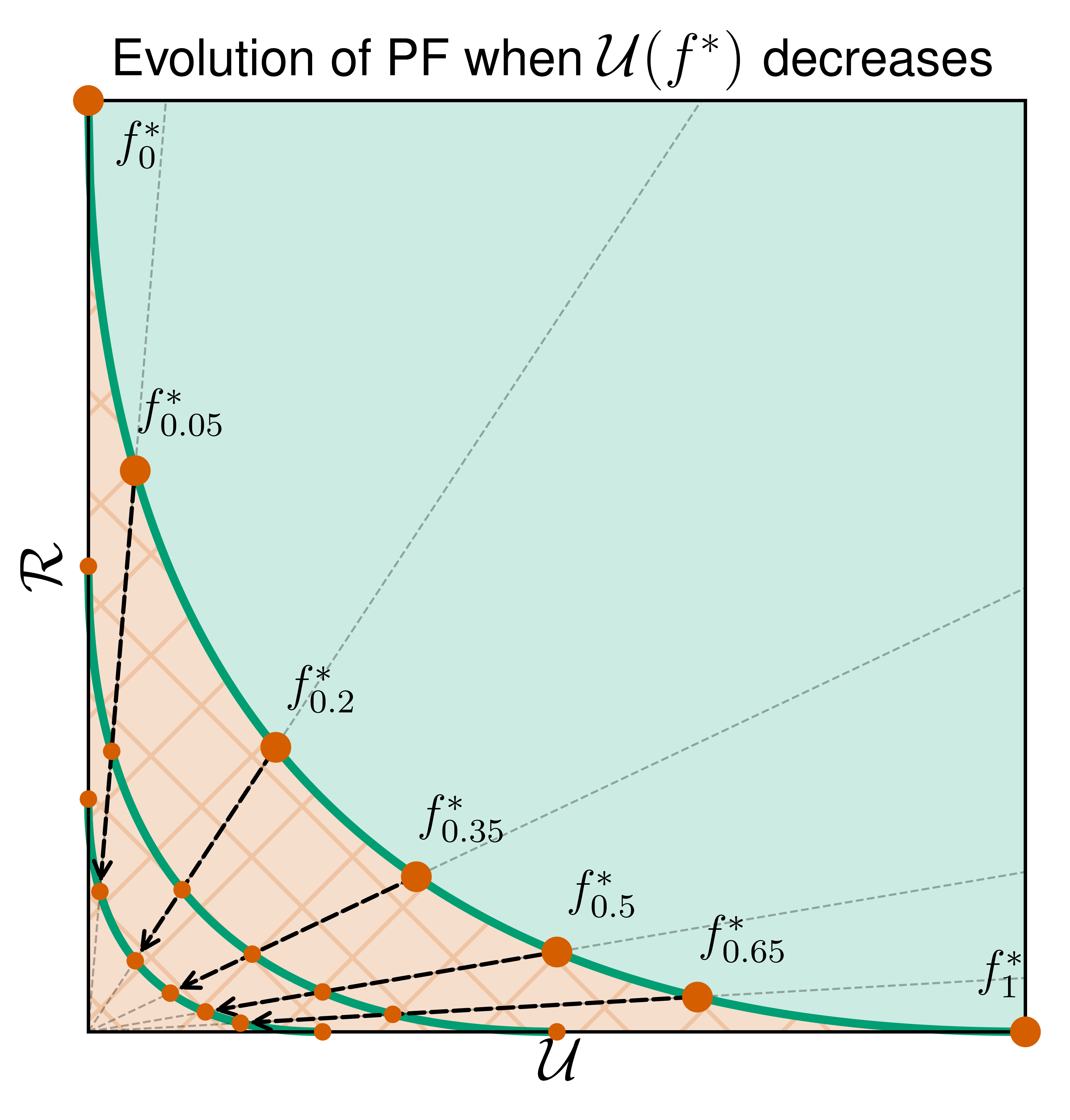}
\caption{Illustration of Pareto frontiers and Pareto dominance. \emph{Left}: Orange (hatched) part is not realisable by any prediction $f$; Each point of green (not hatched) part is realizable by some prediction $f$; The curve that separates the two is the Pareto frontier. \emph{Center}: The darker green (dotted) rectangle in the upper right corner is the set of predictors dominated by $f^*_{0.2}$. \emph{Right}: Evolution of the Pareto frontier when $\class{U}(f^*)$ decreases.}
\label{fig:pareto}
\end{figure}
Even though the parameter $\alpha \in [0, 1]$ has a clear interpretation in our framework, one still might have to figure out which $\alpha$ to pick in practice.
The ultimate theoretical goal is to find a prediction $f$ which simultaneously minimizes the risk $\risk$ and the unfairness $\class{U}$.
Yet, unless $f^*$ satisfies $\class{U}(f^*) = 0$, this goal is unreachable and some trade-offs must be examined.
A standard approach to study such \emph{multi-criteria optimization problems} is via the notion of \emph{Pareto dominance} and \emph{Pareto efficiency}~\citep{osborne1994course}.
In words, the idea of Pareto analysis is to restrict the attention of a practitioner to some set of ``good'' predictors, termed \emph{Pareto frontier} of the multi-criteria optimization problem, instead of considering all possible predictions.
In this section, we show that the set of oracle $\alpha$-RI $\{f^*_{\alpha}\}_{\alpha \in [0, 1]}$ is the Pareto frontier of the multi-criteria minimization problem with target functions $f \mapsto \risk(f)$ and $f \mapsto \class{U}(f)$.

Let us first introduce the terminology of the Pareto analysis specified for our setup.
 We say that a prediction $f$ \emph{Pareto dominates} a prediction $f'$ if one of the following holds
    \begin{itemize}
        \item $\risk(f) \leq \risk(f')$ and $\class{U}(f) < \class{U}(f')$;
        \item $\risk(f) < \risk(f')$ and $\class{U}(f) \leq \class{U}(f')$.
    \end{itemize}
To denote the fact that $f'$ is dominated by $f$ we write $f' \prec f$. Moreover, we say that $f'$ and $f$ are \emph{comparable} if either $f' \prec f$ or $f \prec f'$.
Intuitively, whenever $f' \prec f$, the prediction $f$ is strictly preferable, since it is at least as good as $f'$ for both criteria and it is strictly better for at least one of them.

Note that not every two predictions are actually comparable, that is, the relation $\prec$ only defines a partial-order. It is a known fact that partially ordered sets can be partitioned into well-ordered chains, that is, every pair within the chain is comparable and the restriction of $\prec$ on this chain defines an order relation.
In this set-theoretic terminology, a prediction $f$ is \emph{Pareto efficient} if it is \emph{maximal} within some chain in the sense of the partial order $\prec$.
In other words, a prediction $f$ is \emph{Pareto efficient} if it is not dominated by any other prediction. The set of all Pareto efficient predictions is called the \emph{Pareto frontier} and is denoted by $\PF$.

Note that it would be more accurate to say that $f'$ $\Prob$-Pareto dominates $f$ and $f'$ is $\Prob$-Pareto efficient, since the above definitions are acting on the level of population and they do depend on the underlying distribution. We omit this notation for simplicity.
In general, an analytic description of the Pareto frontier $\PF$ is not necessarily feasible. However, in our case, thanks to the analysis of the previous section, we can precisely describe the Pareto frontier of this problem.

\begin{proposition}
    \label{prop:Pareto}
    Let Assumption~\ref{as:atomless} be satisfied. Then, the Pareto frontier for the multi-criteria minimization problem with objective functions $\risk(f)$ and $\class{U}(f)$ is given by $\{f^*_{\alpha}\}_{\alpha \in [0, 1]}$.
\end{proposition}
\begin{proof}
    On the one hand, by definition of $f^*_{\alpha}$ it holds that $\{f^*_{\alpha}\}_{\alpha \in [0, 1]} \subset \PF$.
    On the other hand, let $f \in \PF$ with $\class{U}(f) \neq 0$ and let $\alpha_f \eqdef \class{U}(f) / \class{U}(f^*)$. Then by the definition of $\alpha_f$ it holds that $\class{U}(f) = \alpha_f\, \class{U}(f^*)$.
    Furthermore, by definition of $f^*_{\alpha_f}$ it holds that $\risk(f^*_{\alpha_f}) \leq \risk(f)$ and by Lemma~\ref{lem:distance_fair_and_almost} it holds that $\class{U}(f^*_{\alpha_f}) = \alpha_{f}\,\class{U}(f^*)$.
    Finally, since $f$ is Pareto efficient it holds that $\risk(f) \leq \risk(f^*_{\alpha_f})$. If $f \in \PF$ is such that $\class{U}(f)=0$, then it is as good as $f^*_0$ in the sense of Pareto. The proof is concluded\footnote{To be more precise, one needs to introduce the equivalence relation $\sim$ defined as $f \sim f'$ iff $\risk(f) = \risk(f')$ and $\class{U}(f) = \class{U}(f')$ and to perform the exact same proof on the quotient space. For the sake of presentation we omit this benign technicality.}.
\end{proof}
Note that any predictor $f$ defines a point $(\class{U}(f), \risk(f))$ in the coordinate system $(\class{U}, \risk)$. The left plot of Figure~\ref{fig:pareto} illustrates the Pareto frontier and those values of $(\class{U}, \risk)$ that are attainable by some prediction $f$. We remark that the convexity of the Pareto frontiers curve is due to the specific trade-off provided by the parameter $\alpha$. For a general multi-criteria optimization problem this convexity is not ensured. The right plot of Figure~\ref{fig:pareto} demonstrates the evolution of the Pareto frontier when $\class{U}(f^*)$ decreases.

Finally, Proposition~\ref{prop:Pareto} provides simple practical guidelines for the study of the trade-off given by $\alpha$.
Note that since thanks to Lemma~\ref{lem:distance_fair_and_almost} it holds that $\class{R}(f^*_0) = \class{U}(f^*)$ and $\class{U}(f^*_1) = \class{U}(f^*)$, then the practitioner needs to estimate only one quantity $\class{U}(f^*)$ and trace the curve of Pareto frontier in order to establish the desired trade-off for the problem at hand.

{
\color{red}
\subsection{Relation with fairness regularized problem}
In this section we present one of the possible applications of the Pareto interpretation of the $\alpha$-relative improvements.
For each $\lambda \geq 0$, define (un)fairness regularized minimizer as
\begin{align*}
    f^{*, \lambda} \in \argmin_{f : \bbR^p \times [K] \to \bbR}\ens{\risk(f) + \lambda \cdot \class{U}(f)}\enspace.
\end{align*}
Note that the level sets of $f \mapsto \risk(f) + \lambda \cdot \class{U}(f)$ induce an affine function in the coordinates $(\risk, \class{U})$. This simple observation allows to establish the connection between minimizers of the penalized objective and $\alpha$-relative improvements---minimizers of the constrained problem.
\begin{proposition}
\label{prop:regular_pareto}
Let Assumption~\ref{as:atomless} be satisfied. Then, for any $\lambda \geq 0$ it holds that
\begin{align*}
    f^{*, \lambda} \equiv f^*_{\alpha(\lambda)}\qquad\text{with}\qquad \alpha(\lambda) = {(1 + \lambda)^{-2}}\enspace.
\end{align*}
\end{proposition}
As a sanity check, we observe that for $\lambda = 0$ we recover the Bayes optimal prediction $f^* \equiv f^*_1 \equiv f^{*, 0}$ and with $\lambda$ approaching $+\infty$, we obtain the expression for the fair optimal prediction $f^{*, +\infty} \equiv f^*_0$.

}





\section{Minimax setup}
\label{sec:minimax_setup}

While the previous section was dealing with the general framework on the population level, the goal of this section is to put forward a minimax setup for the statistical problem of regression with the introduced fairness constraints.

Let $(\bsX_1, S_1, Y_1), \ldots, (\bsX_n, S_n, Y_n)$ be \iid sample with joint distribution $\Probf_{(f^*, \bstheta)}$, where the pair $(f^*, \bstheta) \in \class{F} \times \Theta$ for some class $\class{F}$ and $\Theta$.
In this notation $f^*$ is the regression function and $\bstheta$ is a nuisance parameter.
For example $\class{F}$ can be the set of all affine or Lipschitz continuous functions and $\Theta$ defines additional assumptions on the model in Eq.~\eqref{eq:model_general}
(see Section~\ref{SEC:LINEAR} for a concrete example). For a given fairness parameter $\alpha \in [0, 1]$ and a given confidence parameter $t > 0$, the goal of the statistician is to construct an estimator\footnote{As usual, an estimator $\hat f$ is a measurable mapping of data to the space of predictions.} $\hat f$, which simultaneously satisfies the following two properties
\begin{enumerate}
    \item Uniform fairness guarantee:
    \begin{align}
        \label{eq:uniform-fairness}
        \forall(f^*, \bstheta) \in \class{F} \times \Theta\quad \Probf_{(f^*, \bstheta)}\parent{\class{U}(\hat f) \leq \alpha\,\class{U}(f^*)}  \geq 1 - t\enspace,
    \end{align}
    \item Uniform risk guarantee:
    \begin{align}
        \label{eq:uniform-risk}
        \forall(f^*, \bstheta) \in \class{F} \times \Theta \quad \Probf_{(f^*, \bstheta)}\parent{\risk(\hat f) \leq r_{n, \alpha, f^*}(\class{F}, \Theta, t)} \geq 1 - t\enspace.
    \end{align}
\end{enumerate}
Eq.~\eqref{eq:uniform-fairness} states that the constructed estimator satisfies the fairness requirement with high probability uniformly over the class $\class{F} \times \Theta$.
Meanwhile, in Eq.~\eqref{eq:uniform-risk} we seek for the smallest rate $r_{n, \alpha, f^*}(\class{F}, \Theta, t)$ to quantify the statistical price of being $\alpha$-relatively fair.
Note that $r_{n, \alpha, f^*}(\class{F}, \Theta, t)$ depends explicitly on $f^*$. This is explained by the fact that the fairness of $\hat f$ is measured relatively to $f^*$, hence the price of this constraint also depends on the initial unfairness level of the regression function $f^*$.

The actual construction of the estimator $\hat f$ is problem dependent and the proving that it satisfies Eqs.~\eqref{eq:uniform-fairness}--\eqref{eq:uniform-risk} requires a careful case-by-case study.
In Section~\ref{SEC:LINEAR} we provide an example of such analysis for a simple statistical model of linear regression with systematic group-dependent bias.

\subsection{Generic lower bound}
\label{sec:general_lower}
While the upper bounds of Eqs.~\eqref{eq:uniform-fairness}--\eqref{eq:uniform-risk} require a problem dependent analysis, a general problem dependent lower bound can be derived.
In this section we develop such lower bound.
Let us first introduce some useful definitions.
{
\begin{assumption}[Unconstrained rate]
    \label{ass:lower_no_fair}
    For a fixed confidence level $t \in (0, 1)$ and a class $(\class{F}, \Theta)$, there exists a positive sequence $\delta_n(\class{F}, \Theta, t)$ such that
    \begin{align*}
       &\inf_{\hat f}\sup_{(f^*, \bstheta) \in \class{F} \times \Theta} \Probf_{(f^*, \bstheta)}\parent{\risk(\hat f) \geq \delta_n(\class{F}, \Theta, t)} \geq t\enspace,
    \end{align*}
    where the infimum is taken over all estimators.
\end{assumption}
}

Assumption~\ref{ass:lower_no_fair} can be used with any sequence $\delta_{n}(\class{F}, \Theta, t)$, however, we implicitly assume that $\delta_{n}(\class{F}, \Theta, t)$ corresponds to the minimax optimal rate of estimation of $f^*$ by any estimator (without constraints) in expected squared loss.

\begin{definition}[Valid estimators]\label{def:valid_estimators}
    For some $\alpha \in [0, 1]$ and confidence level $t' \in (0, 1)$ we say that an estimator $\hat f$ is $(\alpha, t')$-valid \wrt the class $(\class{F}, \Theta)$ if
    \begin{align*}
        \inf_{{(f^*, \bstheta)} \in \class{F} \times \Theta} \Probf_{(f^*, \bstheta)}\parent{\class{U}(\hat f) \leq \alpha\,\class{U}({f^*})} \geq 1 - t'\enspace.
    \end{align*}
    The set of all $(\alpha, t')$-valid estimators \wrt the class $(\class{F}, \Theta)$ is denoted by $\widehat{\class{F}}_{(\alpha, t')}$.
\end{definition}
Definition~\ref{def:valid_estimators} characterizes estimators which satisfy the $\alpha$-RI constraint at least with constant probability uniformly over the class $(\class{F}, \Theta)$.

Equipped with Assumption~\ref{ass:lower_no_fair} and Definition~\ref{def:valid_estimators} we are in position to state the main result of this section, which establishes the statistical risk-fairness trade-off.
As we will see in Section~\ref{SEC:LINEAR}, supported by appropriate upper bounds, Theorem~\ref{THM:GENERAL_LOWER} yields optimal rates of convergence up to a multiplicative factor.
\begin{theorem}
    \label{THM:GENERAL_LOWER}
     Let Assumption~\ref{as:atomless} be satisfied. Let $\delta_{n}(\class{F}, \Theta, t)$ be a sequence that satisfies Assumption~\ref{ass:lower_no_fair}. Then 
     \begin{align*}
         \inf_{\hat f \in \widehat{\class{F}}_{(\alpha, t')}}\sup_{{(f^*, \bstheta)} \in \class{F} \times \Theta} \Probf_{{(f^*, \bstheta)}}\parent{\risk^{\sfrac{1}{2}}(\hat f) \geq \delta_n^{\sfrac{1}{2}}(\class{F}, \Theta, t) \vee (1 {-} \sqrt{\alpha})\class{U}^{\sfrac{1}{2}}({f^*})} \geq t \wedge (1 - t')\enspace.
     \end{align*}
\end{theorem}
Drawing an analogy with Lemma~\ref{lem:distance_fair_and_almost}, the two terms of the derived bound have natural interpretations: the first term $\delta_n(\class{F}, \Theta, t)$ is the price of statistical estimation; the second term $(1 {-} \sqrt{\alpha})^2\class{U}({f^*})$ is the price of fairness.
Consequently, the rate $r_{n, \alpha, f^*}(\class{F}, \Theta, t)$ in Eq.~\eqref{eq:uniform-risk} is lower bounded (up to a multiplicative constant factor) by $\delta_n^{\sfrac{1}{2}}(\class{F}, \Theta, t) \vee (1 {-} \sqrt{\alpha})\class{U}^{\sfrac{1}{2}}({f^*})$.
The confidence parameter on the \rhs of the bound is $t \wedge (1 - t')$. The reasonable choice of $t'$ is in the vicinity of zero, which corresponds to estimators satisfying the fairness constraint with high probability.
Finally, observe that this bound is not conventional in the sense of classical statistics, where the bound would converge to zero with the growth of sample size.
This behavior is not surprising, since the infimum is taken \wrt to $(\alpha, t')$-valid estimators and not \wrt all possible estimators.
One can draw an analogy of the obtained bound with recent results in robust statistics~\citep{chen2016general,chen2018robust}, where the minimax rate converges to a function of the proportion of outliers, which might be different from zero.

\section{Application to linear model with systematic bias}
\label{SEC:LINEAR}
\myparagraph{Additional notation}
 We denote by $\|\cdot\|_2$ and by $\|\cdot\|_n = (\sfrac{1}{\sqrt n})\|\cdot\|_2$ the Euclidean and the normalized Euclidean norm. The standard scalar product is denoted by $\scalar{\cdot}{\cdot}$. We denote by $\1_{p}$ the vector of all ones of size $p$. 
 For square matrix $\mathbf{A} \in \mathbb{R}^{n \times n}, n\geq1$, we write $\mathbf{A} \succ 0$ if $\mathbf{A}$ is symmetric positive-definite.

The goal of this part is to provide an example of a complete statistical analysis for a regression problem under the $\alpha$-RI constraint. In particular, we show how to apply the plug-and-play results of Section~\ref{sec:general_lower} in order to derive minimax rate optimal bounds under the $\alpha$-RI constraint.
To this end we apply the developed theory to the following model of linear regression with systematic group-dependent bias
\begin{align}
    \label{eq:model_linear}
    Y = \scalar{\bsX}{\bbeta^*} + b_S^* + \xi\enspace,
\end{align}
where $\bsX \sim \class{N}(\boldsymbol{0}, \bfSigma)$ is a feature vector independent from the sensitive attribute $S$ with $\bfSigma \succ 0$; $\xi \sim \class{N}({0}, \sigma^2)$ is an additive independent noise; and the vector $\bsb^* = (b_1^*, \ldots, b_K^*)$ is the vector of systematic bias.
{We assume that the noise level $\sigma$ is known to the statistician.}
Note that in this case the regression function $f^*$ is given by the expression $f^*(\bsx, s) = \scalar{\bsx}{\bbeta^*} + b_s^*$ and Assumption~\ref{as:atomless} is satisfied.
We assume that the observations are
\begin{align}
    \label{eq:model_linear_vector}
   \bsY_s = \bfX_s \bbeta^* + b_s^*\1_{n_s} + \bsxi_s,\quad s = 1, \ldots, K\enspace,
\end{align}
with $\bsY_s, \bsxi_s \in \bbR^{n_s}$, $\bfX_s \in \bbR^{n_s \times p}$, and $\1_{n_s}$ is the vector of all ones of size $n_s$.
The rows of $\bfX_s$ are \iid realization of $\bsX$, the components of $\bsxi_s$ are \iid from $\mathcal{N}(0, \sigma^2)$.
Additionally, we set $n = n_1 + \ldots + n_K$ and $w_s = \sfrac{n_s}{n}$.
The risk of $f: \bbR^p \times [K] \to \bbR$ is then defined as
\begin{align*}
    \risk(f) = \sum_{s = 1}^K w_s\Exp\parent{\scalar{\bsX}{\bbeta^*} + b_s^* - f(\bsX, s)}^2\enspace.
\end{align*}

\begin{remark}
    We set $w_s = \sfrac{n_s}{n}$ instead of $w_s = \Prob(S = s)$ to simplify the presentation and proofs of the main results.
    Note that if $S_1, \ldots, S_n$ is an \iid sample, then $n_s = \sum_{i = 1}^n\ind{S_i = s}$ and $\Expf[\sfrac{n_s}{n}] = \Prob(S=s)$, that is our choice of weights essentially corresponds to the scenario of \iid sampling of sensitive attribute.
\end{remark}
Using the terminology of Section~\ref{sec:general_lower} the joint distribution of data sample $\Probf_{(f^*, \bstheta)}$ is uniquely defined by $(\bbeta^*, \bsb^*)$ and $(\bfSigma, \sigma)$. That is, $(\bbeta^*, \bsb^*)$ defines the regression function $f^*$ and $(\bfSigma, \sigma)$ is the nuisance parameter $\bstheta$.
To simplify the notation we write $\Probf_{(\bbeta^*, \bsb^*)}$ instead of $\Probf_{(\bbeta^*, \bsb^*, \bfSigma, \sigma)}$.

The following result is the application of Proposition~\ref{prop:optimal_alpha} to the model in Eq.~\eqref{eq:model_linear}.
\begin{proposition}
    For all $\alpha \in [0, 1]$, the $\alpha$-relative improvement of $f^*$ is given for all $(\bsx, s) \in \bbR^p \times [K]$ by
    \begin{align*}
        f^*_{\alpha}(\bsx, s) = \scalar{\bsx}{\bbeta^*} + \sqrt{\alpha}b_s^* + (1 {-} \sqrt{\alpha})\sum_{s = 1}^Kw_sb_s^*\enspace.
    \end{align*}
\end{proposition}
In order to build an estimator $\hat f$, which improves the fairness of $f^*$, while providing minimal risk among such predictions, we first estimate parameters of model in Eq.~\eqref{eq:model_linear} using least-squares estimators
\begin{align}
    \label{eq:LS}
    (\hbeta, \hbsb) \in \argmin_{(\bbeta, \bsb) \in \bbR^p \times \bbR^K} \sum_{s = 1}^K w_s \norm{\bsY_s - \bfX_s \bbeta - b_s\1_{n_s}}_{n_s}^2\enspace.
\end{align}
Based on the above quantities we then define a family of linear estimators $\hat f_{\tau}$ parametrized by $\tau \in [0, 1]$ as
\begin{align}
    \label{eq:estimator_general}
    \hat f_{\tau}(\bsx, s) = \scalarin{\bsx}{\hbeta} + \sqrt{\tau}\hb_s + (1 {-} \sqrt{\tau})\sum_{s = 1}^Kw_s\hb_s\,\,, \qquad (\bsx, s) \in \bbR^p \times [K]\enspace.
\end{align}
We would like to find a value of $\tau = \tau_n(\alpha)$ such that Eqs.~\eqref{eq:uniform-fairness}--\eqref{eq:uniform-risk} are satisfied.
Note that the choice of $\tau = \alpha$ would not yield the desired fairness guarantee stated in Eq.~\eqref{eq:uniform-fairness}.
As it will be shown later, $\tau$ should be smaller than $\alpha$, in order to account for finite sample effects and derive high confidence fairness guarantee.
The next result shows that under the model in Eq.~\eqref{eq:model_linear}, the unfairness of $\hat f_{\tau}$ can be computed in a data-driven manner, which is crucial for the consequent choice of $\tau$.

\begin{lemma}
    \label{lem:unfairness_estimator1}
    For any $\tau \in [0, 1]$, the unfairness of $\hat f_{\tau}$ is given by
    \begin{align*}
        \class{U}(\hat f_{\tau}) = \tau\sum_{s = 1}^Kw_s\parent{\hb_s - \sum_{s' = 1}^Kw_{s'}\hb_{s'}}^2\qquad\text{almost surely}\enspace.
    \end{align*}
\end{lemma}
Apart from being computable in practice, Lemma~\ref{lem:unfairness_estimator1} provides an intuitive result that $\class{U}(\hat f_{\tau})$ is the variance of the bias term $\hbsb$.

\subsection{Upper bound}
Linear regression is one of the most well-studied problems of statistics~\citep{nemirovski2000topics,tsybakov2003optimal,gyorfi2006distribution,mourtada2019exact,catoni:hal-00104952,hsu2012random,audibert2011robust}.
In the context of fairness, linear regression is considered in~\citep{calders2013controlling,berk2017convex,Donini_Oneto_Ben-David_Taylor_Pontil18}, where the fairness constraint is formulated via the approximate equality of group-wise means. In this section we establish a statistical guarantee on the risk and fairness of $\hat f_{\tau}$ for an appropriate data-driven choice of $\tau$.
Our theoretical analysis in this part is inspired by that of~\cite{hsu2012random}, who derived high probability bounds on least squares estimator for linear regression with random design.

The following rate plays a crucial rule in the analysis of this section
\begin{align*}
    \delta_n(p, K, t) = 8\parent{\frac{p}{n} + \frac{K}{n}} + 16\parent{\sqrt{\frac{p}{n}} + \sqrt{\frac{K}{n}}}\sqrt{\frac{t}{n}}
        +
        \frac{32t}{n}\enspace.
\end{align*}
Not taking into account the confidence parameter $t > 0$, $\delta_n(p, K, t) \asymp  {(p + K)}/{n}$ up to a constant multiplicative factor, which as it is shown in Theorem~\ref{thm:lower_linear} is the minimax optimal rate for the model in Eq.~\eqref{eq:model_linear} without the fairness constraint.
    \begin{theorem}[Fairness and risk upper bound]
        \label{thm:real_final}
        Define
        \begin{align*}
            &\hat\tau =
            \begin{cases}
              \alpha\parent{1 + \frac{\sigma{\delta^{\sfrac{1}{2}}_n(p, K, t)}}{{\class{U}^{\sfrac{1}{2}}(\hat f_1)} - \sigma{\delta^{\sfrac{1}{2}}_n(p, K, t)}}}^{-2} &\text{ if }\quad {\class{U}^{\sfrac{1}{2}}(\hat f_1)} > \sigma{\delta^{\sfrac{1}{2}}_n(p, K, t)}\\
              0, &\text{ otherwise}
            \end{cases}\enspace.
        \end{align*}
        Consider $p, K \in \bbN, t \geq 0$ and define $\theta(p, K, t) = \sfrac{(4\sqrt{K} + 5\sqrt{t} + 6\sqrt{p})}{(\sqrt{p}+\sqrt{t})}$.
        Assume that $\sqrt{n} \geq 2\sfrac{(\sqrt{p}+\sqrt{t})}{(\theta(p, K, t) - \sqrt{\theta^2(p, K, t) - 3})}$. Then, for any $\alpha \in [0, 1]$, with probability at least $1 - 4\exp(-t/2)$ it holds that
    \begin{align*}
        \class{U}(\hat f_{\hat\tau}) \leq \alpha\,\class{U}(f^*)
         \quad\text{and}\quad
        \risk^{\sfrac{1}{2}}(\hat f_{\hat\tau})
        \leq
        2\sigma(1 {+} \sqrt{\alpha}){\delta^{\sfrac{1}{2}}_n(p, K, t)}
        +
        (1 {-} \sqrt{\alpha})\,\class{U}^{\sfrac{1}{2}}(f^*)
        \enspace.
     \end{align*}
\end{theorem}
Theorem~\ref{thm:real_final} simultaneously provides two results: first, it shows that the estimator $\hat f_{\hat \tau}$ is $(\alpha, 4e^{-t/2})$-valid, that is, it satisfies the fairness constraint with high probability; second it provides the rate of convergence which consists of two parts.
The first part of the rate, $\sigma\delta_n^{\sfrac{1}{2}}(p, K, t)$, is the price of statistical estimation of $(\bbeta^*, \bsb^*)$, while the second part, $(1 {-} \sqrt{\alpha})\class{U}^{\sfrac{1}{2}}(f^*)$, {is the price one has to pay when introducing the $\alpha$-RI fairness constraint.}
In order to achieve the fairness validity, we need to loosen the value of $\alpha$ to reflect the base level of unfairness, that is, $\hat\tau$ is adjusted by $\class{U}(\hat f_1)$.
Let us point out that the bound of Theorem~\ref{thm:real_final} slightly differs from the conditions required by Eqs.~\eqref{eq:uniform-fairness}--\eqref{eq:uniform-risk}. In particular, it provides a joint guarantee on risk and fairness.

Let us remark that the previous result requires $n$ to be sufficiently large, similarly to the conditions in~\citep{hsu2012random, audibert2011robust}.
One can obtain a more explicit, but more restrictive bound on $n$ by finding sufficient conditions under which the assumption on $n$ is satisfied. For instance, rough computations show that it is sufficient to assume that $\sqrt{n} \geq 16\sqrt{K}$ and $\sqrt{n}\geq 12.5(\sqrt{p} + \sqrt{t})$.

At last, we emphasize that the choice of $\hat\tau$ requires the knowledge of the noise level $\sigma$, that is, this choice is not adaptive.
However, our proof can effortlessly be extended to the case when only an upper bound $\bar{\sigma}$ on the noise level $\sigma$ is known.
In this case $\sigma$ should be replaced by $\bar\sigma$ in the definition of $\hat\tau$ and in the resulting rate.
The question of adaptation to $\sigma$ without any prior knowledge should be treated separately and is out of the scope of this work.

\begin{figure}[!t]
\centering
\includegraphics[width=0.49\textwidth]{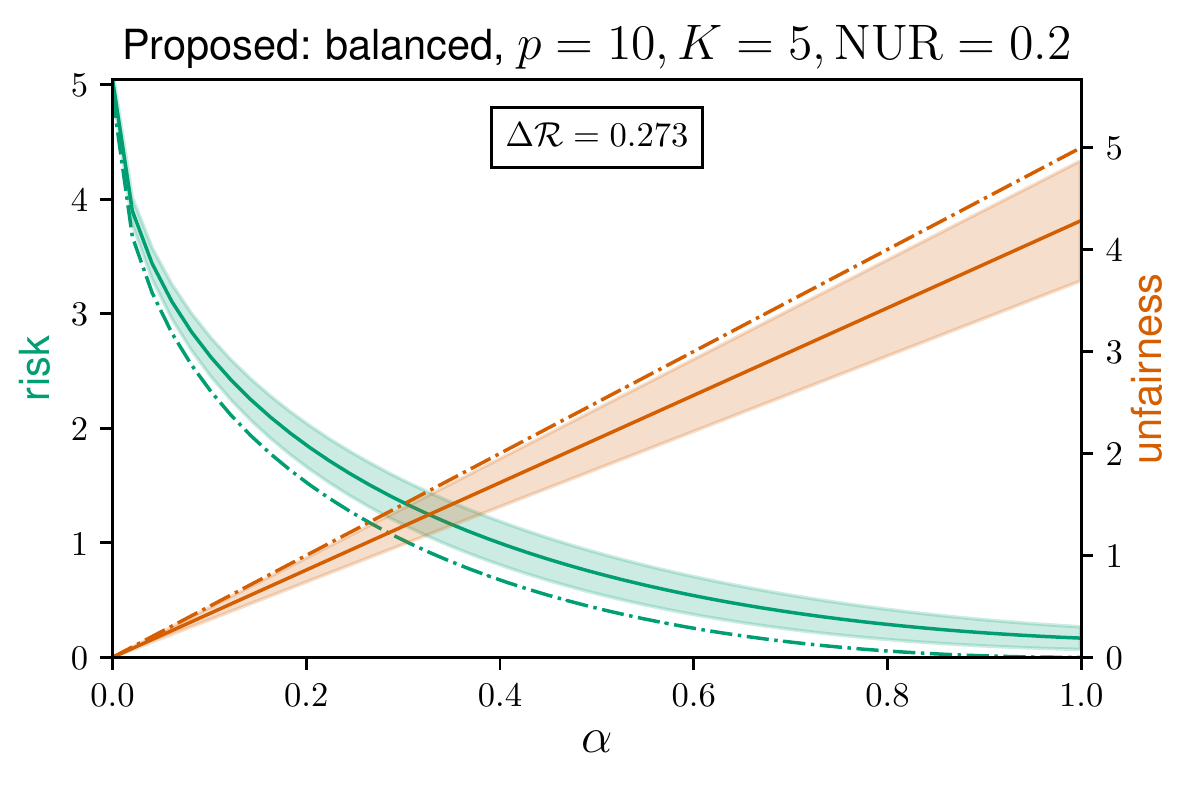}
\includegraphics[width=0.49\textwidth]{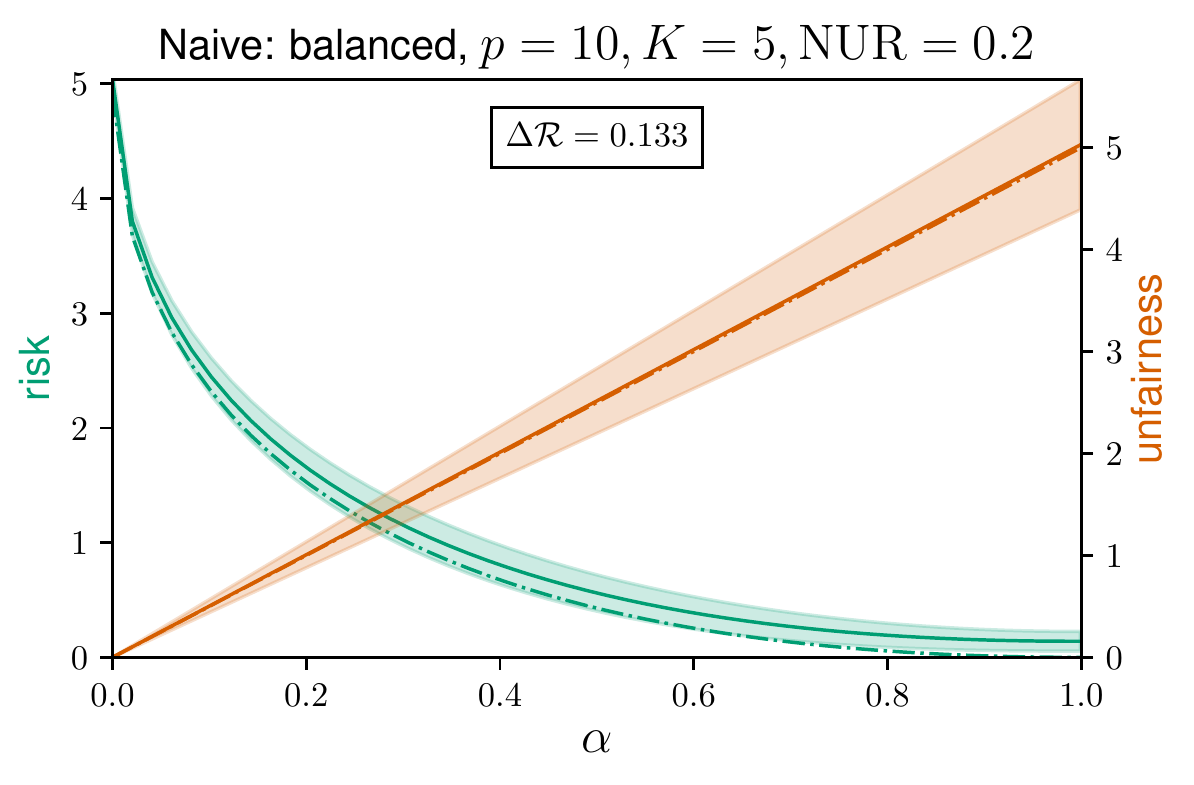}
\includegraphics[width=0.49\textwidth]{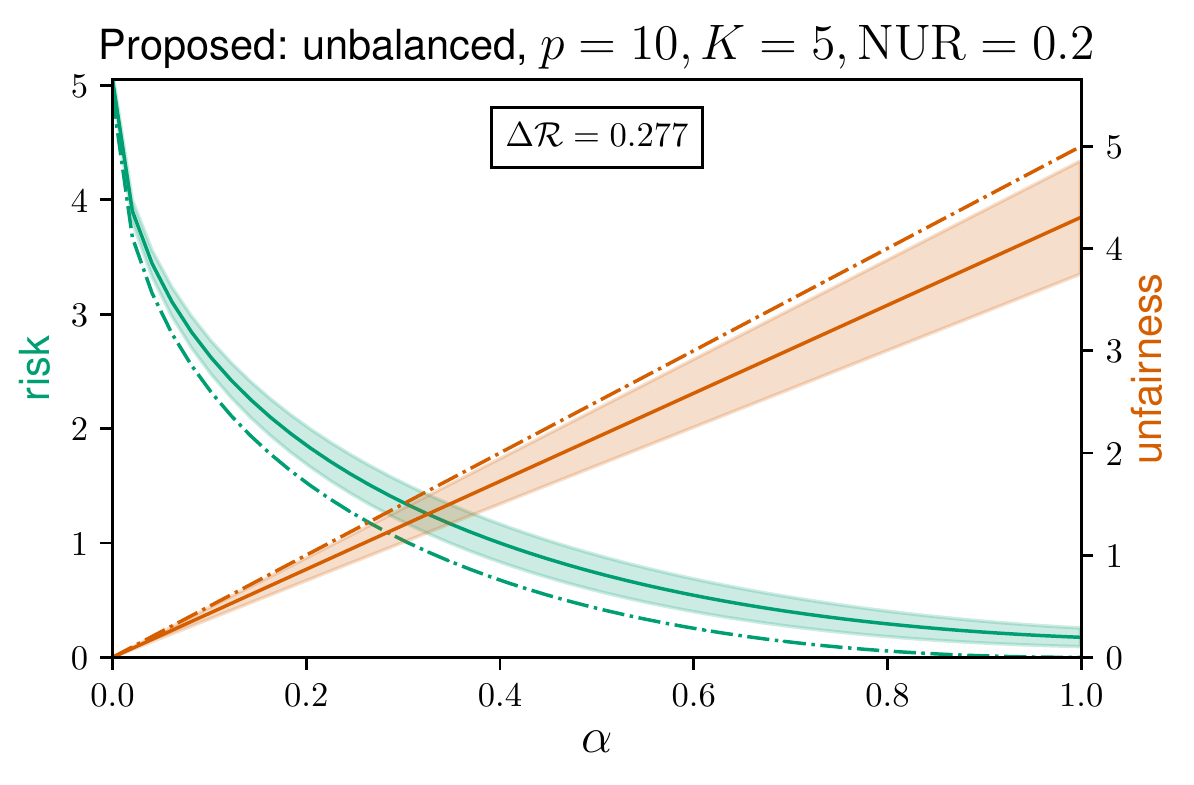}
\includegraphics[width=0.49\textwidth]{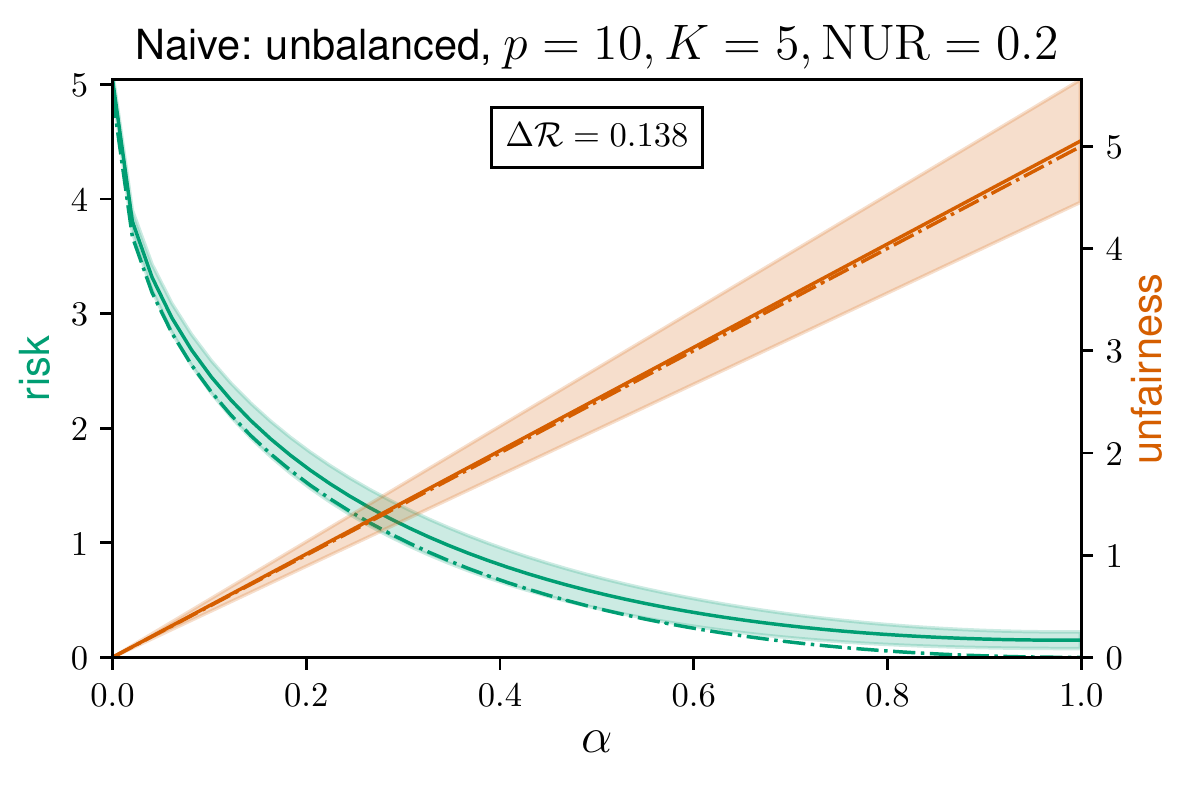}
\caption{Dashed green and brown lines correspond to the risk and unfairness of $f^*_{\alpha}$ respectively. Solid green and brown lines correspond to the average risk and unfairness of $\hat f_{\tau(\alpha)}$ and the shaded region shows three standard deviations over $50$ repetitions. On the left $\tau(\alpha) = \hat\tau$ and on the risk $\tau(\alpha) = \alpha$.}
\label{fig:exp1}
\end{figure}

\subsection{Lower bound}
The goal of this section is to provide a lower bound, demonstrating that the result of Theorem~\ref{thm:real_final} is minimax optimal up to a multiplicative constant factor.
Recall that thanks to the general lower bound derived in Theorem~\ref{THM:GENERAL_LOWER} it is sufficient to prove a lower bound on the risk without constraining the set of possible estimators.
Even though the problem of linear regression is well studied, to the best of our knowledge
there is no known lower bound for the model in Eq.~\eqref{eq:model_linear} which \emph{i)} holds for the random design \emph{ii)} is stated in probability \emph{iii)} considers explicitly the confidence parameter $t$.
Next theorem establishes such lower bound.
\begin{theorem}
    \label{thm:lower_linear}
    For all $n, p, K \in \bbN$, $t \geq 0$, $\sigma > 0$ it holds that
    \begin{align*}
       \inf_{\hat f}\sup_{(\bbeta^*, \bsb^*) \in \bbR^p \times \bbR^K, \bfSigma \succ 0} \Probf_{(\bbeta^*, \bsb^*)}\parent{\risk(\hat f) \geq \frac{\sigma^2}{3 \cdot 2^9 n} (\sqrt{p+K} + \sqrt{32t})^2} \geq \frac{1}{12}e^{-t}\enspace,
    \end{align*}
    where the infimum is taken \wrt all estimators.
\end{theorem}
The proof of Theorem~\ref{thm:lower_linear} relies on standard information theoretic results.
In particular, in order to prove optimal exponential concentration we follow similar strategy as that of~\cite{bellec2017optimal,kerkyacharian2014optimal} who derived optimal exponential concentrations in the context of density aggregation and binary classification.
Theorem~\ref{thm:lower_linear} combined with generic lower bound derived in Theorem~\ref{THM:GENERAL_LOWER} yields  the following corollary.
\begin{corollary}
    \label{cor:lower_bound_final}
    Let $\bar{\delta}_n(p, K, t) = (\sqrt{\sfrac{(p + K)}{n}} + \sqrt{\sfrac{32 t}{n}})^2 / (3 \cdot 2^9)$.
    For all $n, p, K \in \bbN$, $t \geq 0$, $\sigma > 0$, $\alpha \in [0, 1]$ it holds for all $t \geq 0$ and all $t' \leq 1 - \sfrac{e^{-t}}{12}$ that
    \begin{align*}
         \inf_{\hat f \in \widehat{\class{F}}_{\alpha, t'}}\sup_{(\bbeta^*, \bsb^*) \in \bbR^p \times \bbR^K, \bfSigma \succ 0} \Probf_{(\bbeta^*, \bsb^*)}\parent{\risk^{\sfrac{1}{2}}(\hat f) \geq {\sigma}\bar{\delta}^{\sfrac{1}{2}}_n(p, K, t) \vee (1 {-} \sqrt{\alpha})\,\class{U}^{\sfrac{1}{2}}(f^*)} \geq \frac{1}{12}e^{-t}\enspace.
    \end{align*}
\end{corollary}
Comparing the upper bound of Theorem~\ref{thm:real_final} and the lower bound of Corollary~\ref{cor:lower_bound_final} we conclude that the two obtained rates are the same up to a multiplicative constant factor.
Hence confirming the tightness of the results derived in Section~\ref{sec:general_lower}.

\subsection{Simulation study}
\label{sec:simulation}

\begin{figure}[!t]
\centering
\includegraphics[width=0.49\textwidth]{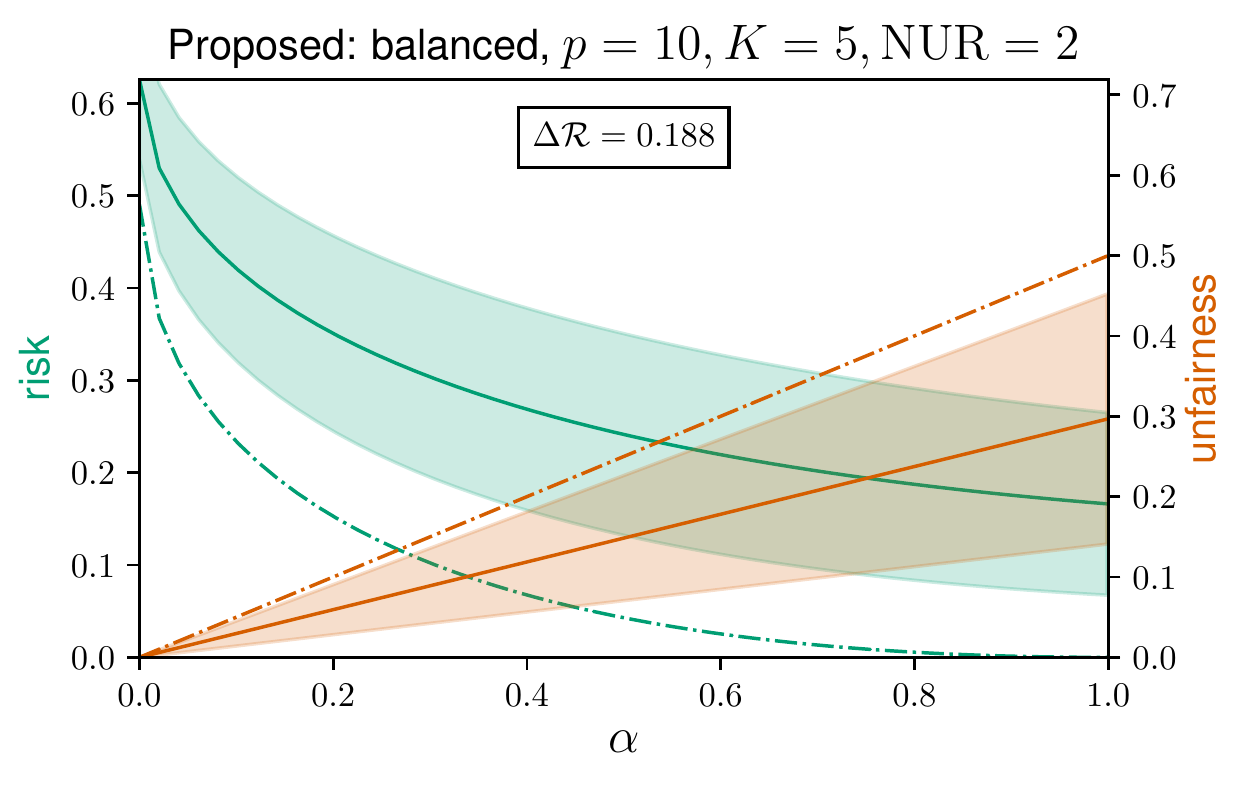}
\includegraphics[width=0.49\textwidth]{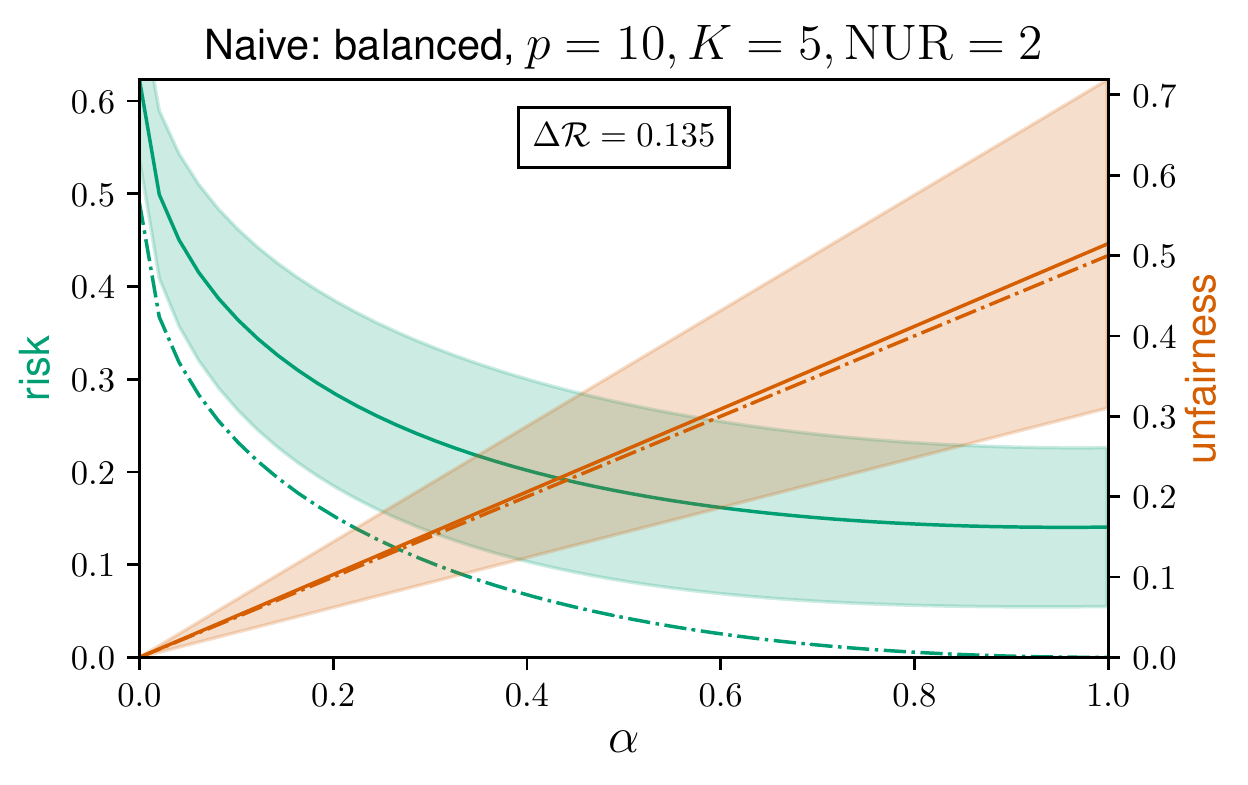}
\includegraphics[width=0.49\textwidth]{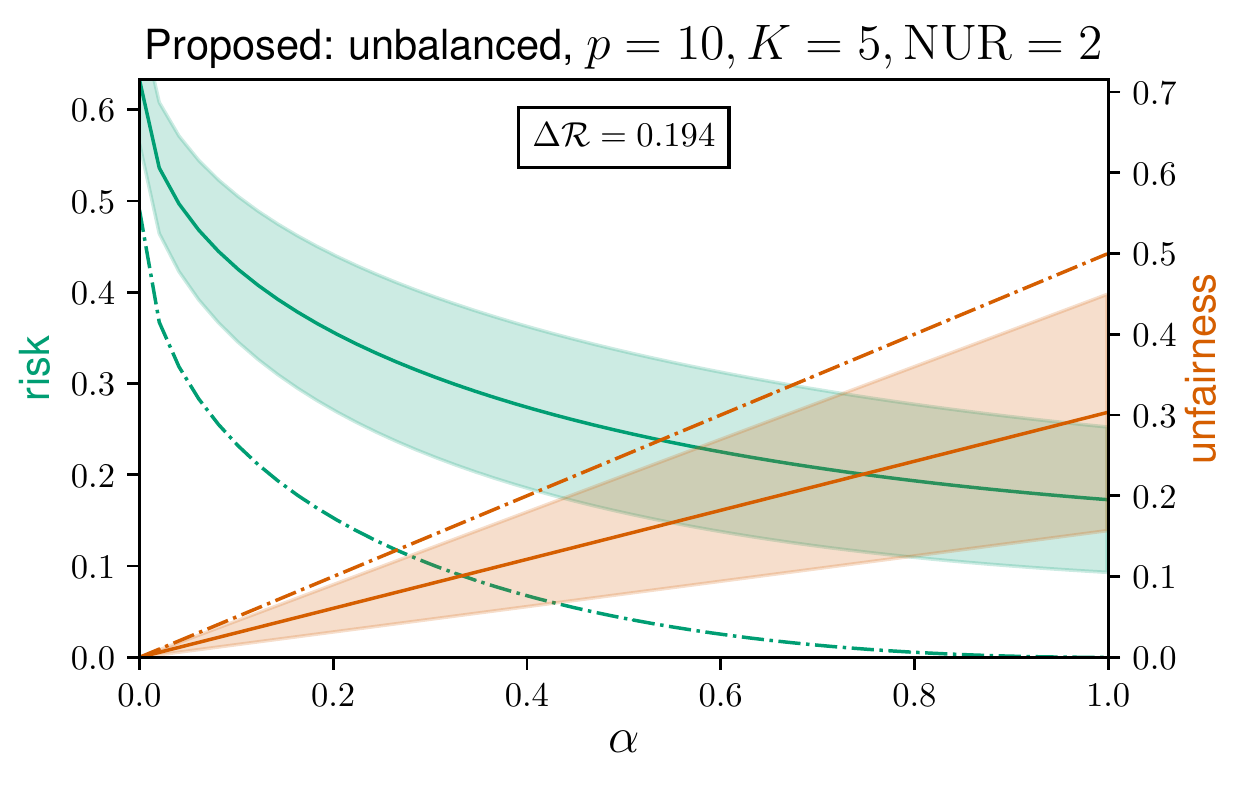}
\includegraphics[width=0.49\textwidth]{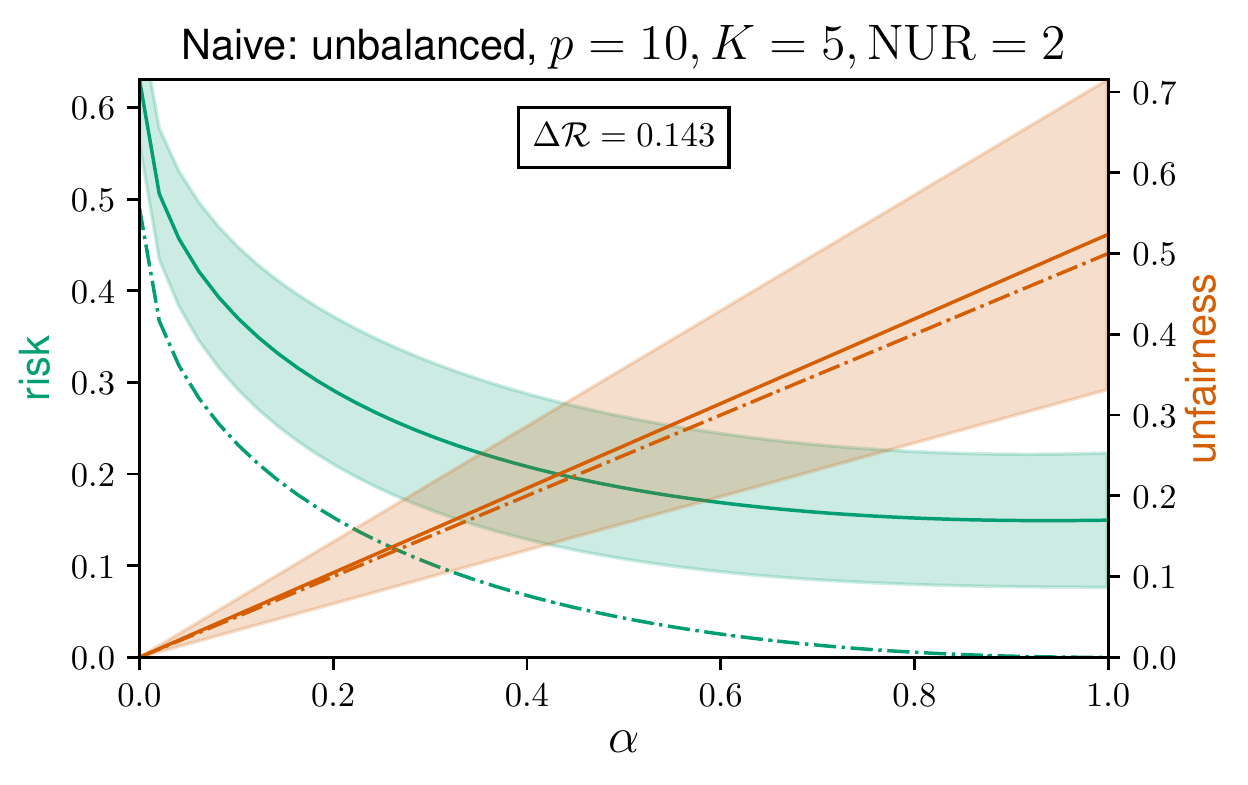}
\caption{Dashed green and brown lines correspond to the risk and unfairness of $f^*_{\alpha}$ respectively. Solid green and brown lines correspond to the average risk and unfairness of $\hat f_{\tau(\alpha)}$ and the shaded region shows three standard deviations over $50$ repetitions. On the left $\tau(\alpha) = \hat\tau$ while on the right $\tau(\alpha) = \alpha$.}
\label{fig:exp2}
\end{figure}

In this section we perform simulation study to empirically validate our theoretical analysis\footnote{For our empirical validation and illustrations we have relied on the following \texttt{python} packages: \texttt{scikit-learn}~\citep{scikit-learn}, \texttt{numpy}~\citep{numpy}, \texttt{matplotlib}~\citep{matplotlib}, \texttt{seaborn}.}.
Before continuing let us discuss the notion of signal-to-unfairness ratio.
Setting $\bbeta^* = 0$ in the model~\eqref{eq:model_linear}, if the amplitudes of $b^*_s$ is much smaller than the noise level $\sigma^2$, then the observations $\bsY_s$ are mainly composed of noise.
While for the prediction problem it is not a problem, since our rates will scale with the noise level, it becomes important for the estimation of unfairness $\class{U}(f^*)$.
Motivated by this discussion, we define the noise-to-unfairness ratio as
     \begin{align*}
        \nur^2 \eqdef {\sigma^2}\big/{\class{U}(f^*)}\enspace.
    \end{align*}
The signal-to-unfairness ratio tells as how the level of unfairness compares to the noise level.
The regime $\nur \gg 1$ means that the unfairness of the distributions is below the noise level, and it is statistically difficult to estimate it. In contrast, $\nur \ll 1$ implies that the unfairness dominates the noise.
Instead of varying $\class{U}(f^*)$ and $\sigma$ we fix $\sigma$ and perform our study for different values of $\nur$.

We follow the following protocol.
For some fixed $K, n_1, \ldots, n_K, p, \sigma, \nur$ we simulate the model in Eq.~\eqref{eq:model_linear_vector} with $\bfSigma = \mathbf{I}_p$.
In all the experiments we set $\bbeta^* = (1, \ldots, 1)^\top \in \bbR^p$.
For $\bsb^*$ we first define $\bsv = (1, -1, 1, -1, \ldots)^\top \in \bbR^K$ and set $\bsb^* = \bsv \sqrt{\sfrac{\sigma^2}{{\nur \cdot \Var_S(\bsv)}}}$, where $\Var_S(\bsv)$ is the variance of $\bsv$ with weights $w_1, \ldots, w_K$.
So that the unfairness of this model is exactly equal to $\sigma^2 / \nur^2$.
On each simulation round of the model, we compute the estimator in Eq.~\eqref{eq:estimator_general} with two choices of parameter $\tau$:
\begin{enumerate}[noitemsep,topsep=0.1ex]
    \item \textbf{Proposed}:  $\tau(\alpha) = \hat\tau$ from Theorem~\ref{thm:real_final};
    \item \textbf{Naive}:  $\tau(\alpha) = \alpha$.
\end{enumerate}
\begin{remark}
    While performing experiments we have noticed that setting $\hat\tau$ with $\delta_n(p, K, t)$ defined in Theorem~\ref{thm:real_final} results in too pessimistic estimates in terms of unfairness, for this reason in all of our experiments we set $\delta_n(p, K, t) = (\sfrac{p}{n}) + (\sfrac{K}{n})$, which is of the same order as that of Theorem~\ref{thm:real_final}.
\end{remark}
Then, for each $\hat f_{\tau(\alpha)}$ we evaluate $\risk(\hat f_{\tau(\alpha)})$ and $\class{U}(\hat f_{\tau(\alpha)})$.
This procedure is repeated $50$ times, which results in $50$ values of $\risk(\hat f_{\tau(\alpha)})$ and $\class{U}(\hat f_{\tau(\alpha)})$ for each $\alpha \in (0, 1)$.
For these $50$ values we compute mean and standard deviation.
We considered $p = 10$, $K = 5$, $\sigma = 1$, and $\nur \in \{0.2, 0.5, 2\}$.
Furthermore, for the choice of $n_1, \ldots, n_K$ we study the following two regimes
\begin{enumerate}[noitemsep,topsep=0.5ex]
    \item \textbf{Balanced}: $n_1 = \ldots = n_5 = 100$.
    \item \textbf{Unbalanced}: $n_1 = 5, n_2 = 45, n_3=100, n_4=100, n_5=250$.
\end{enumerate}
The reason we consider two regimes is to confirm the theoretical findings of Theorem~\ref{thm:real_final}, which indicate that the rate is governed by $n_1 + \ldots + n_K$ instead of the their individual values.
Finally, for a given fairness parameter function $\alpha \mapsto \tau(\alpha)$ we report cumulative risk increase over all $\alpha \in [0, 1]$ defined as
\begin{align*}
    \Delta\risk(\tau) \eqdef \int_{0}^1 \parent{\risk(\hat f_{\tau(\alpha)}) - \risk(f^*_{\alpha})}\, \d\alpha\enspace.
\end{align*}
This quantity describes the cumulative risk loss of the rule $\tau(\alpha)$ across all the levels of fairness $\alpha$ compared to the best $\alpha$-relative improvement $f^*_{\alpha}$.

On Figures~\ref{fig:exp1}--\ref{fig:exp2} we draw the evolution of the risk and of the unfairness when $\alpha$ traverses the interval $[0, 1]$. We also report $\Delta\risk(\tau)$ defined above.
Inspecting the plots we can see that that the main disadvantage of the naive choice of $\tau = \alpha$ is its poor fairness guarantee, that is, in almost half of the outcomes, the unfairness of $\hat f_{\alpha}$ exceeded the prescribed value. In contrast, the proposed choice of $\tau(\alpha) = \hat\tau$ consistently improves the unfairness of the regression function $f^*$, empirically validating our findings in Theorem~\ref{thm:real_final}. However, good fairness results come at the cost of consistently higher risk.
One can also see that the effect of unbalanced distributions is negligible for the considered model (it only affects the variance of the result). This is explained by the definition of the risk, which weights the groups proportionally to their frequencies.
Finally, observing the behavior of naive approach for $\nur = 0.2$ and $\nur = 2$ we note that in the latter case the unfairness of $\hat f_{\alpha}$ starts to deviate from the true value (with consistently positive bias). Meanwhile, since the proposed choice $\tau(\alpha) = \hat\tau$ is more conservative, the bias remains negative, that is, the unfairness of $f^*$ is still improved.


{
\section{Post-processing method without model}
\label{sec:estimation}

{The previous section was concerned with illustrating the minimax setup, introduced in Section~\ref{sec:minimax_setup}, with a concrete choice of parametric model. In practice, however, and especially for benchmark problems, parametric assumptions can hardly be verified; therefore, a more generic estimation algorithm, which does not rely on the modeling assumptions, is desirable.}

In what follows, we propose a generic post-processing algorithm which can be applied on top of \emph{any} black-box estimation procedure.
Note that, due to the appealing structure of $\alpha$-RI, we only need to estimate the two prediction functions: the Bayes rule $f^*$ and the fair optimal $f^*_0$. Then, an estimator of $f^*_{\alpha}$ can be built as a convex combination of the estimators of $f^*$ and $f^*$. The literature on the estimation of the Bayes rule $f^*$ is rather rich, thus, we only detail the estimation of $f^*_0$. Before proceeding, we introduce an additional bit of notation.




\myparagraph{Additional notation}
For any prediction function $f : \bbR^p \times [K] \to \bbR$, any $q \in [1, \infty)$, define 
\begin{align*}
  \norm{f}_q^q \eqdef \sum_{s = 1}^K w_s \Exp[|f(\bsX, S)|^q \mid S = s]
\end{align*}
and $\norm{f}_{\infty} \eqdef \sum_{s = 1}^Kw_s \inf\enscond{b \in \bbR}{\Prob(f(\bsX, S) < b \mid S = s) = 0}$.


\subsection{The algorithm}
\label{sec:adhoc}

{For each sub-population $s \in [K]$, let $(\bsX_1^s, \ldots, \bsX_{2N_s}^s)$ be $2N_s$ \iid feature vectors\footnote{For simplicity and without loss of generality, we assume that an even number of observations is available for every sensitive attribute.} sampled from the distribution $(\bsX \mid S = s)$, independently from $(\bsX, S, Y)$ and from all other observations.}
Note that we explicitly allow ourselves to sample from each sensitive group separately. This is not restrictive, since an \iid sample from $\Prob$ can be converted into the considered sampling scheme by conditioning on the number of available observations from each group.

Let $f : \bbR^p \times [K] \to \bbR$ be a fixed prediction function.
Our goal is to build a post-processing operator $f \mapsto \hat{\Pi}(f) : \bbR^p \times [K] \to \bbR$ such that the post-processing estimator $\hat{\Pi}(f)$ satisfies the Demographic Parity constraint and its closeness to $f^*_0$ is controlled by that of $f$ to $f^*$.
The rationale behind this goal is the representation of $\alpha$
-RI as a point-wise combination of $f^*$ and $f^*_0$: if the considered function $f$ is a good estimator of $f^*$ and $\hat{\Pi}(f)$ of $f^*_0$, then, as a trivial consequence of the triangle inequality, we can effectively estimate all of the $\{f^*_{\alpha}\}_{\alpha \in [0,1]}$.

More formally, we want to build a (possibly randomized) operator $\hat{\Pi}$ which satisfies, for every $f: \bbR^p \times [K] \to \bbR$, $q \in [1, +\infty)$,
\begin{align*}
    &\Law\left( \hat{\Pi}(f) (\bsX, S) \mid S = s \right) = \Law\left( \hat{\Pi}(f) (\bsX, S) \mid S = s' \right) \qquad \forall s, s' \in [K]\enspace,\\
    &\Expf\norm{\hat{\Pi}(f) - f^*_0}_q \leq \text{Error}_q(f, f^*) + \text{Remainder}(N_1, \ldots, N_K)\enspace,
\end{align*}
where $\text{Error}_q(f, f^*)$ represents the quality of the base estimator, which is supposed to be a good approximation of the Bayes rule $f^*$. Ideally, we want to have $\text{Error}_q(f, f^*) = C \cdot \norm{f - f^*}_q$ for some $C \geq 1$. We will achieve such a goal for $q = 1$, while for $q > 1$, the error term will be slightly different.

\begin{remark}
Note that the first condition is stated with respect to the joint distribution of $\hat{\Pi}(f) (\bsX, S)$, \ie it involves all the randomness present in $\hat{\Pi}(f) (\bsX, S)$.
\end{remark}

\myparagraph{Estimator construction}
Let $\zeta, (\zeta_i^{s})_{i = 1, \ldots, 2N_s; s = 1, \ldots, K}$ be \iid real valued random variables distributed uniformly on $[-\sigma, \sigma]$ for some $\sigma > 0$ to be specified.
For each $f : \bbR^p \times [K] \to \bbR$, each $s \in [K]$ and $i \in [2N_s]$, define the following random variables
\begin{align*}
    &\tilde{f}_i^s \eqdef f(\bsX_i^s, s) + \zeta_i^s\quad\text{and}\quad
    \tilde{f}(\bsx, s) := f(\bsx, s) + \zeta \qquad \forall\,\, (\bsx, s) \in \bbR^p \times [K]\enspace.
\end{align*}
Using the above quantities, we build the following estimators: for all $t \in \bbR$
\begin{align*}
    &\hat{F}_{1, \nu^{f}_s}(t) \eqdef \frac{1}{N_s {+} 1}\parent{\sum_{i = 1}^{N_s}\ind{\tilde{f}_i^s < t} + U^{s}\left(1 {+} \sum_{i = 1}^{N_s}\ind{\tilde{f}_i^s = t}\right)}\enspace,\\
    &\hat{F}_{2, \nu^{f}_s}(t) \eqdef \frac{1}{N_s}\sum_{i = N_s + 1}^{2N_s}\ind{\tilde{f}_i^s \leq t}\enspace,
\end{align*}
where $(U^{s})_{s \in [K]}$ are \iid random variables, distributed uniformly on $[0, 1]$ and independent from all the previously introduced random variables.
Consequently, for each $f: \bbR^p \times [K] \to \bbR$ we define for all $(\bsx, s) \in \bbR^p \times [K]$
\begin{align}
    \label{eq:ad_hoc_def}
    \hat{\Pi}(f)(\bsx, s) = \sum_{s' = 1}^K w_{s'}\hat{F}_{2, \nu^{f}_{s'}}^{-1} \circ \hat{F}_{1, \nu^{f}_s} \circ \tilde{f}(\bsx, s)\enspace.
\end{align}
{The form of our estimator was inspired by the explicit formula obtained for the fair optimal predictor $f_0^*$ (see Proposition~\ref{prop:optimal_alpha}). It can be seen as its empirical counterpart with additional randomization.

From now on our goal is to establish the desired guarantees on the operator $\hat{\Pi}$. A more in-depth study of this procedure is left for future works. We begin by providing an \emph{exact} demographic parity guarantee.}


\begin{theorem}[Demographic parity guarantee]\label{thm:improved_fairness}
    For any $f: \bbR^p \times [K] \to \bbR$, any joint distribution $\Prob$ of $(\bsX, S, Y)$ and any $\sigma > 0$, it holds that
    \begin{align*}
        \Law\left( \hat{\Pi}(f) (\bsX, S) \mid S = s \right) = \Law\left( \hat{\Pi}(f) (\bsX, S) \mid S = s' \right) \qquad \forall s, s' \in [K]\enspace.
    \end{align*}
\end{theorem}

{
The above theorem may appear as ``magical'' at first sight as it holds under no assumption. 
To obtain such a result, we leverage distribution-free properties on rank and order statistics presented in Lemma~\ref{lem:inverse_with_noise} in Appendix. This approach was inspired by the literature on conformal prediction~\citep{vovk2005algorithmic,lei2014distribution,lei2018distribution,foygel2021limits} in which similar tools are used. Theorem~\ref{thm:improved_fairness} and the estimator in Eq.~\eqref{eq:ad_hoc_def} improve upon the estimator of~\cite{chzhen2020fair}, for which only approximate fairness is established.

Since our fairness guarantee on its own is not necessarily informative (constant predictions trivially satisfy it), we complement it with a post-processing estimation bound to assess the predictive performance of our estimator. Unlike the previous result, which was assumption-free, we need an additional assumption stated below.
}



\begin{assumption}
    \label{as:density_bound}
    For all $s \in [K]$, the measures $\nu_s^*$ are supported on an interval in $\bbR$, admit density \wrt Lebesgue measure which is lower and upper bounded by $\underline{\lambda}_s > 0$ and $\overline{\lambda}_s > 0$ respectively.
\end{assumption}

Note that the requirement of the existence of an upper-bounded density is not particularly restrictive. The most restrictive part of this assumption is the lower bound on this density. The reason we introduce it can be intuitively understood from asymptotic normality results on order statistics~\cite[see e.g.,][Section~21.2 and Corollary~21.5]{van2000asymptotic} {which show that the limiting variance depends on the inverse density evaluated at the quantile that one wishes to estimate. In particular, it  explodes when this density approaches zero.} In our case, we actually need to estimate the whole quantile function (since $f^*_0$ is based on it, see Prop.~\ref{prop:optimal_alpha}), hence we require that the density is uniformly lower bounded.
It appears that the most important implication of the lower bounded density that we use is the $\underline{\lambda}_s^{-1}$-Lipschitz continuity of the quantile function of $\nu^*_s$.
Due to this reason, it is certainly possible to relax the lower boundedness assumption by H\"older smoothness condition on the quantile function of $\nu^*_s$. We leave this investigation for future works.

\begin{theorem}[Estimation guarantee]\label{thm:risk_bound_ad_hoc}
   Let Assumptions~\ref{as:atomless} and~\ref{as:density_bound} be satisfied, then for any base prediction rule $f : \bbR^p \times [K] \to \bbR$, any $\sigma \in (0, 1)$ and any $q \in [1, \infty)$, the proposed post-processing procedure satisfies
    \begin{align*}
        \Expf\normin{\hat{\Pi}(f) - f^*_0}_q \leq C_{ \overline{\underline{\boldsymbol\lambda}}}^q\bigg( &\norm{f - f^*}_q {+}  \min\left\{\norm{f - f^*}_{q-1}^{1 - \sfrac{1}{q}} {+} \sigma^{1 - \sfrac 1 q},\, \norm{f - f^*}_{\infty} {+} \sigma \right\}\ind{q {>} 1} \\
        & + \left\{\sum_{s = 1}^Kw_sN_{s}^{-\sfrac{1}{2}}\right\} +  \left\{\sum_{s = 1}^Kw_sN_s^{-\sfrac q 2}\right\}^{\sfrac{1}{q}} + \sigma\bigg)\enspace,
    \end{align*}
    where $C_{\overline{\underline{\boldsymbol\lambda}}}^q$ depends only on $(\underline{\lambda}_s)_s, (\overline{\lambda}_s)_s$ from Assumption~\ref{as:density_bound} and $q \in [1, \infty)$.
\end{theorem}
Note that for $q = 1$ we get a ``truly'' post-processing guarantee---the quality of the post-processed prediction function in $\ell_1$-norm is controlled by the quality of the initial function in $\ell_1$-norm. For the case $q > 1$, the bound slightly deteriorates and comprises two parts. One that involves a $(q{-}1)$-norm to the power $1 - \tfrac{1}{q}$ and the other one involves a $\infty$-norm. Note that, following~\cite{Stone77}, it is always possible to build an estimator $f$ of $f^*$ using an independent set of labeled data which is consistent in $(q{-}1)$-norm. However, due to the presence of $1 - \tfrac{1}{q}$ in the exponent, the potential rate of convergence might be far from optimal. To this end, we included the term with $\infty$-norm. Indeed, under smoothness assumptions on $f^*$ and extra conditions on $\Prob_{\bsX \mid S}$ it is possible to build estimators $f$ of $f^*$ in $\infty$-norm only loosing an extra logarithmic factor~\citep[see e.g.,][]{Tsybakov09} compared to the estimation in $q < \infty$ norm.
Furthermore we note that the constants in the above bound do not depend on the number of groups $K$. The dependency on $K$ is implicit in the quality of estimation of $f^*$ by $f$ and is explicit in the parametric part of the rate. As an example consider $w_s = \Prob(S = s)$, $N_s = N \cdot\Prob(S=s)$, and $q=1$. In that case the remainder term can be bounded by $\sqrt{\sfrac{K}{N}}$---the standard parametric rate.

\section{Empirical study on real data}
\label{sec:real_exps}
In this section we perform empirical study on the \texttt{Communities and Crime} dataset from UCI Machine Learning Repository\footnote{Data source: \href{https://archive.ics.uci.edu/ml/datasets/communities+and+crime}{https://archive.ics.uci.edu/ml/datasets/communities+and+crime}}.
We study the post-processing procedure described in Eq.~\eqref{eq:ad_hoc_def} as well as the estimator from Eq.~\eqref{eq:estimator_general} that we designed in the context of linear regression with systematic bias in Section~\ref{SEC:LINEAR} (we refer to the latter estimator as \texttt{BLS}).
The post-processing algorithm in Eq.~\eqref{eq:ad_hoc_def} relies on some base estimator $f$ and we set $\hat{f}_{\alpha} \eqdef \sqrt{\alpha} f + (1 - \sqrt{\alpha}) \hat{\Pi}(f)$. That is, for $\alpha=1$ we recover $\hat{f}_{\alpha}\equiv f$---the base estimator itself. {Thus, the $\texttt{BLS}$ and each base estimator + post-processing induce a family of estimators parametrized by $\alpha \in [0, 1]$. In what follows we study these families of estimators.}


\subsection{Statistics about predictions}
We split the dataset $\data$ into three disjoint parts: $\data_{\train}$, $\data_{\unlab}$, and $\data_{\test}$ (of sizes $50\%, 30\%, 20\%$ respectively). The first, $\data_{\train}$, is used to fit an initial estimator $f$ of $f^*$; the second, $\data_{\unlab}$, is used to perform the post-processing $\hat{\Pi}(f)$ described in Eq.~\eqref{eq:ad_hoc_def} (we set $\sigma = 10^{-6}$); the last, $\data_{\test}$, is used to compute various statistics related to the performance and fairness of estimators.
For the estimator from Section~\ref{SEC:LINEAR} we use both $\data_{\train}$ for the train, since it is a one-shot estimator, which does not require data splitting.

Let $\bsw \in \Delta^{K-1}$ be a weight vector. For any predictor $f:\mathbb{R}^p \times [K]  \rightarrow \mathbb{R}$, we measure its performance with weighted mean-squared error on the test data
\begin{align*}
    \widehat{\MSE}(f) = \sum_{s = 1}^Kw_s\hat{\Exp}\left[(Y - f(\bsX, S))^2 \mid S = s\right]\enspace,
\end{align*}
where $\hat{\Exp}$ is evaluated on $\data_{\test}$.
The unfairness estimator (see Appendix~\ref{sec:q_losses} for details on this estimator) of $f:\mathbb{R}^p \times [K]  \rightarrow \mathbb{R}$ is defined as
\begin{align*}
    \widehat{\class{U}}(f) =  \int_0^1 \min_{y \in \bbR}\sum_{s=1}^K w_s\lvert \hat{F}_{\nu_s^f}^{-1}(t) - y \rvert \d t\enspace,
\end{align*}
where $\hat{F}_{\nu_s^f}^{-1}(\cdot)$ is the generalized inverse of $\hat{F}_{\nu_s^f}(t) \eqdef \tfrac{1}{|\data_{\test}^s|}\sum_{(\bsX, S, Y) \in \data_{\test}^s}\ind{f(\bsX, S) \leq t}$ with $\data_{\test}^s = \{(\bsX, S, Y) \in \data_{\test} \,:\, S = s\}$.
Note that as long as the image measure of each $\{\Prob_{\bsX \mid S = s}\}_{s \in [K]}$ under $f(\cdot, s)$ is supported on an interval and its density is positive, $\widehat{\class{U}}(f)$ is a consistent estimator (conditionally on $\data_{\train}, \data_{\unlab}, f$) of $\class{U}(f)$ as a consequence of~\cite[Theorem 5.2]{bobkov2019one} and Lemma~\ref{lem:barycenter_p} (in Appendix~\ref{sec:q_losses}).
Unless stated otherwise, we fix $\bsw = (\sfrac{1}{K}, \ldots, \sfrac{1}{K})^\top$.

\begin{wrapfigure}{r}{0.4\textwidth}
    \centering
    \includegraphics[width=0.37\textwidth]{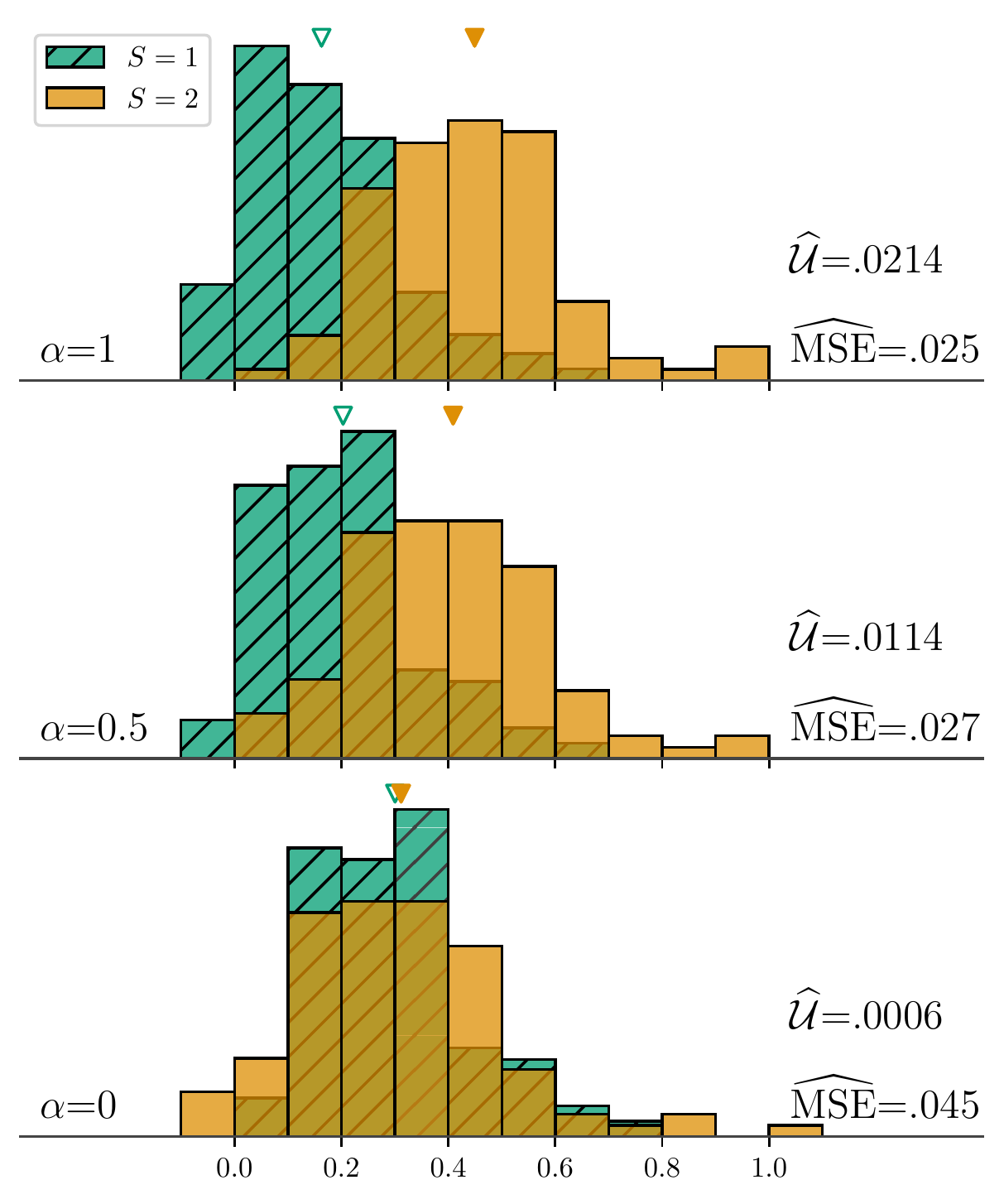}%
    \caption{Group-wise distribution of \texttt{BLS} from~Eq.~\eqref{eq:estimator_general} for three values of $\alpha$.
    }
    \label{fig:LRWSB}
\end{wrapfigure}
\subsection{\texttt{BLS} estimator from Section~\ref{SEC:LINEAR}} The \texttt{BLS} estimator from Section~\ref{SEC:LINEAR} strongly relies on the linear model with systematic bias.
Nevertheless, it appears that after properly pre-processing the data, this estimator can actually yield reasonable performance both in terms of risk and unfairness.
Given the linear model in Eq.~\eqref{eq:model_linear}, the data pre-processing is rather straightforward: we need to (at least) make sure that the group-wise means of the feature vectors are (approximately) zero. This is achieved by estimating these means on $\data_{\train}$ and subtracting these estimated means from all the features on $\data_{\test}$.
After fitting the \texttt{BLS} estimator from Eq.~\eqref{eq:estimator_general}, we display the histograms of group-wise distributions of the predicted values on Figure~\ref{fig:LRWSB}. As expected, the \texttt{BLS} estimator does not modify the shape of the group-wise distributions---it only changes their means, which are displayed by the triangles pointing downwards. It is interesting to note that such a simple method, with a proper data pre-processing step, can deliver a reasonable performance both in terms of risk and fairness. 
Nevertheless, if the equalization of the distributions (and not only means) is mandatory for the given application, one should consider instead the ad-hoc procedure that we study in the next section.

\begin{figure}[t!]
\begin{subfigure}{.475\textwidth}
  \centering
  \includegraphics[width=\textwidth]{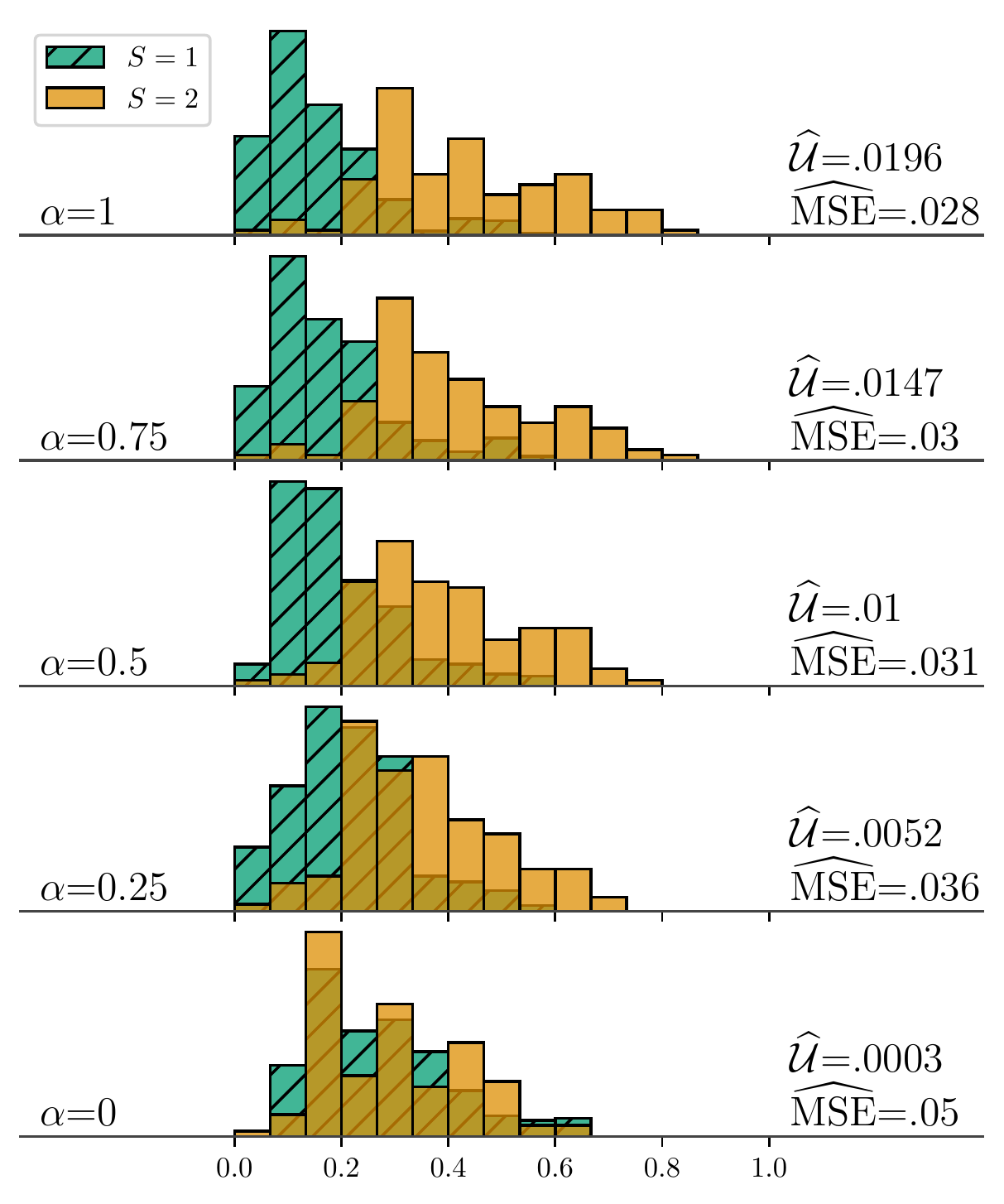}
  \caption{Evaluated on $\data_{\test}$.}
  \label{fig:sfig1}
\end{subfigure}%
\begin{subfigure}{.475\textwidth}
  \centering
  \includegraphics[width=\textwidth]{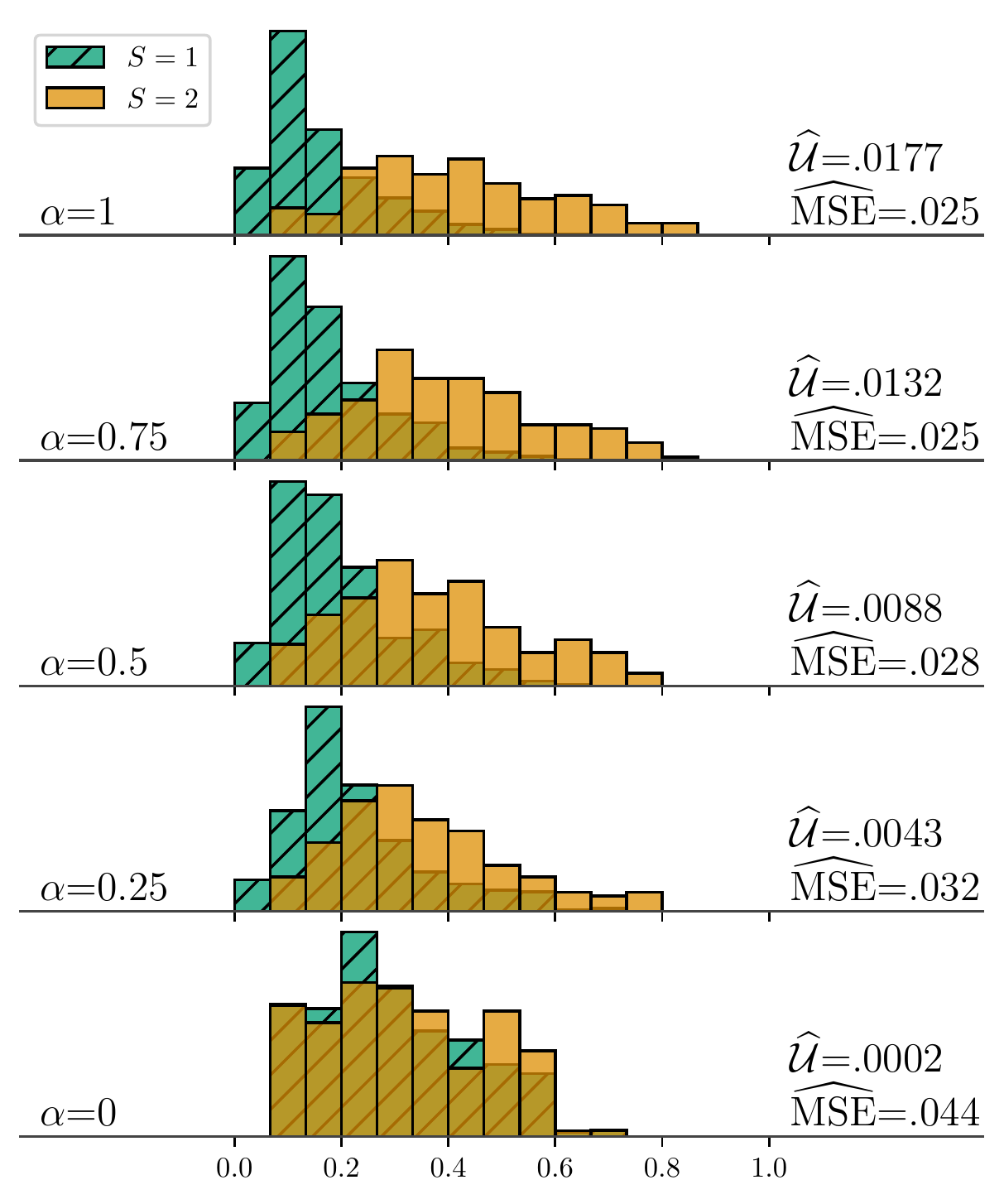}
  \caption{Evaluated on $\data_{\unlab}$.}
  \label{fig:sfig2}
\end{subfigure}
\caption{Evolution of the empirical distributions of the post-processing method for different choices of $\alpha \in [0, 1]$ coupled with the \texttt{RF} as the base estimator.}
\label{fig:dist_evolution}
\end{figure}

\subsection{Ad-hoc estimator from Section~\ref{sec:adhoc}: distribution equalization}

We tested our post-processing procedure with two different base methods: random forest (\texttt{RF}) and $k$-Nearest Neighbours (\texttt{kNN})\footnote{As before, we rely on \texttt{scikit-learn} package \citep{scikit-learn} to implement those base methods.}. We fitted each method on the training set $\data_{\train}$ using 3-fold CV procedure to optimize hyper-parameters. 

First, on Figure~\ref{fig:dist_evolution}, we display the evolution of the group-wise distributions computed on $\data_{\test}$ (Fig.~\ref{fig:sfig1}) and on $\data_{\unlab}$ (Fig.~\ref{fig:sfig2}) when $\alpha$ varies from $1$ (no fairness adjustment) to $0$ (estimation of fair optimal). Due to the form of the post-processing procedure, it is expected that the post-processing estimator achieves near perfect distribution matching when evaluated on the unlabeled data. This phenomenon is displayed on Fig.~\ref{fig:sfig2}. Furthermore, on $\data_{\test}$, which was never used during the training stage, the post-processing algorithm also displays a visually (and quantitatively, as indicated by risk and unfairness measures) superior performance.

\subsection{The performance frontier}

We repeat the splitting of the whole dataset $\data$ into the three sets for $100$ times and, for each $\alpha \in [0, 1]$, average the resulting unfairness and MSE. Then, for each method and each $\alpha \in [0, 1]$ we obtain a point in the coordinates $(\widehat{\MSE},\, \widehat{\class{U}})$ and the resulting curves (parametrized by $\alpha$) are displayed in Figure~\ref{fig:Pareto_empirical}.
We note that the estimator \texttt{BLS} from Section~\ref{SEC:LINEAR} performs reasonably well and even dominates the post-processing based on the \texttt{kNN} in the regime of moderately low values $\alpha$. At the same time, the post-processing coupled with the \texttt{RF} as the base method uniformly dominates both \texttt{kNN}+post-processing and \texttt{BLS}---model-based method of Section~\ref{SEC:LINEAR}.
Furthermore, on Figure~\ref{fig:Pareto_empirical2} one can clearly see the main drawback of the \texttt{BLS} estimator---it fails to achieve a very low level of unfairness {(as a point of reference compare the largest unfairness of $\texttt{RF}$ and the smallest of $\texttt{BLS}$)}.
Such a behaviour is not surprising since the Gaussian features assumption is probably violated and thus, the best we can hope for in general situation is the group-wise mean equalization.

\begin{figure}[t!]
\begin{subfigure}{.45\textwidth}
  \centering
  \includegraphics[width=\textwidth]{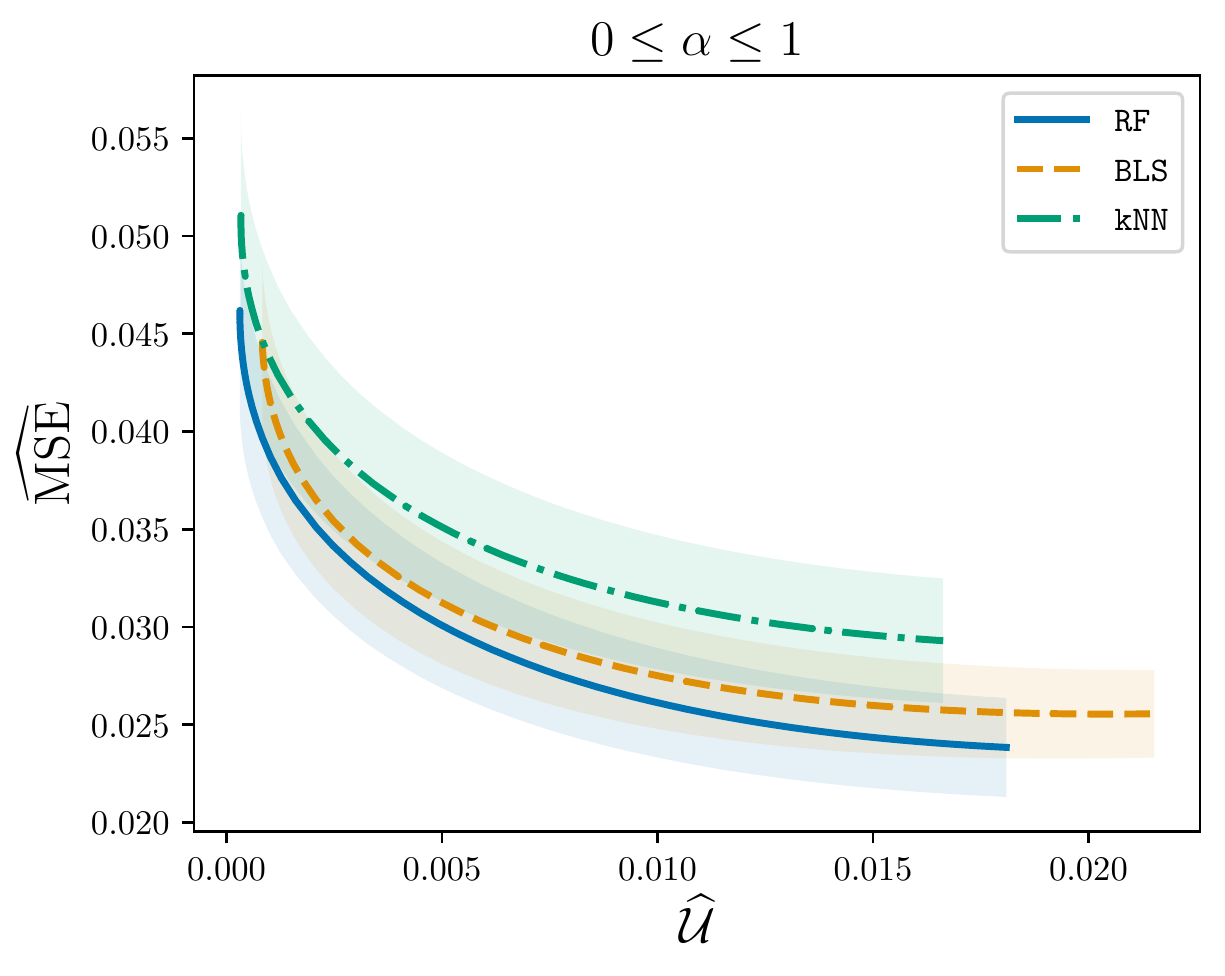}
  \caption{Performance frontier of three methods: \texttt{RF}+post-process; \texttt{kNN}+post-process; \texttt{BLS} from~Eq.~\eqref{eq:estimator_general} for $\alpha \in [0, 1]$.}
  \label{fig:Pareto_empirical1}
\end{subfigure}%
\hfill
\begin{subfigure}{.45\textwidth}
  \centering
  \includegraphics[width=\textwidth]{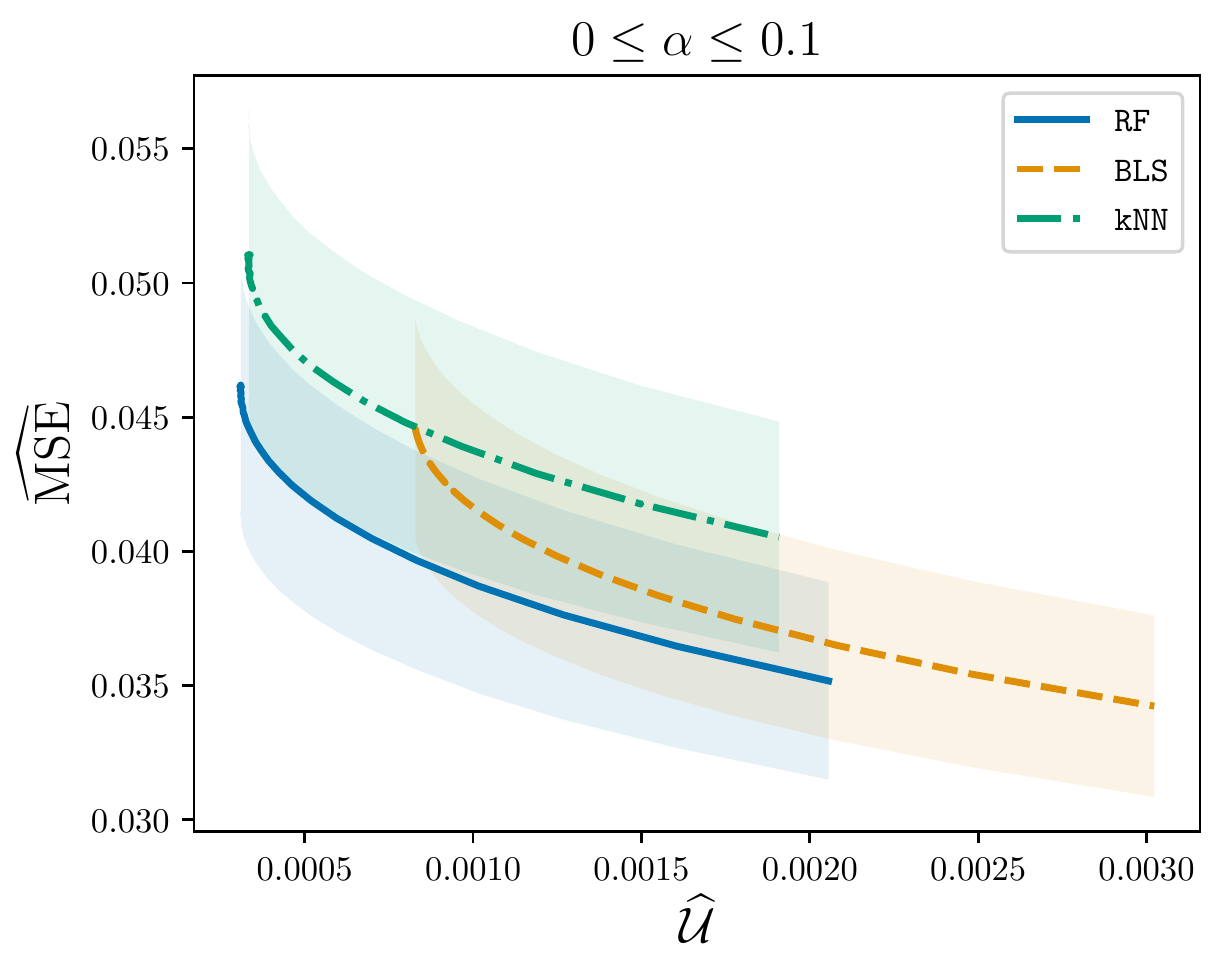}
  \caption{Performance frontier of three methods: \texttt{RF}+post-process; \texttt{kNN}+post-process; \texttt{BLS} from~Eq.~\eqref{eq:estimator_general} for $\alpha \in [0, 0.1]$.}
  \label{fig:Pareto_empirical2}
\end{subfigure}
\caption{Performance frontiers of different families of estimators}
\label{fig:Pareto_empirical}
\end{figure}

\subsection{Ad-hoc estimator from Section~\ref{sec:adhoc}: verifying properties}

Even though the ad-hoc procedure in Section~\ref{sec:adhoc} directly mimics the expression for the $\alpha$-RI, it is actually a randomized prediction. Hence, as suggested by one of the reviewers, it is interesting to empirically verify the order-preservation and the mean stability properties announced in Section~\ref{SUBSEC:GENERAL}.

\myparagraph{Order preservation} Note that, while the order preservation property was stated for the $\alpha$-RI relevant to the Bayes optimal $f^*$, it should be understood relative to the base estimator for the post-processing procedure .
To this end, we fit the post-processing operator $\hat{\Pi}$ on $\data_{\unlab}$ and, for three values of $\alpha$, display
\begin{align*}
    f(\bsx, s) \qquad\text{vs}\qquad \hat{f}_{\alpha}(\bsx, s) = \sqrt{\alpha}f(\bsx, s) + \hat{\Pi}(f)(\bsx, s)\qquad\qquad (\bsx, s) \in \bbR^p \times [K]\enspace,
\end{align*}
evaluated on $(\bsx, s) \in \data_{\test}$, where $f$ is a base-estimator trained on $\data_{\train}$. If the order preservation property holds, we expect to see a monotone curve for all $\alpha \in [0, 1)$. {Note that the curve for the reference value $\alpha = 1$ corresponds to that of the identity mapping.}
\begin{figure}[t!]
  \centering
  \includegraphics[width=\textwidth]{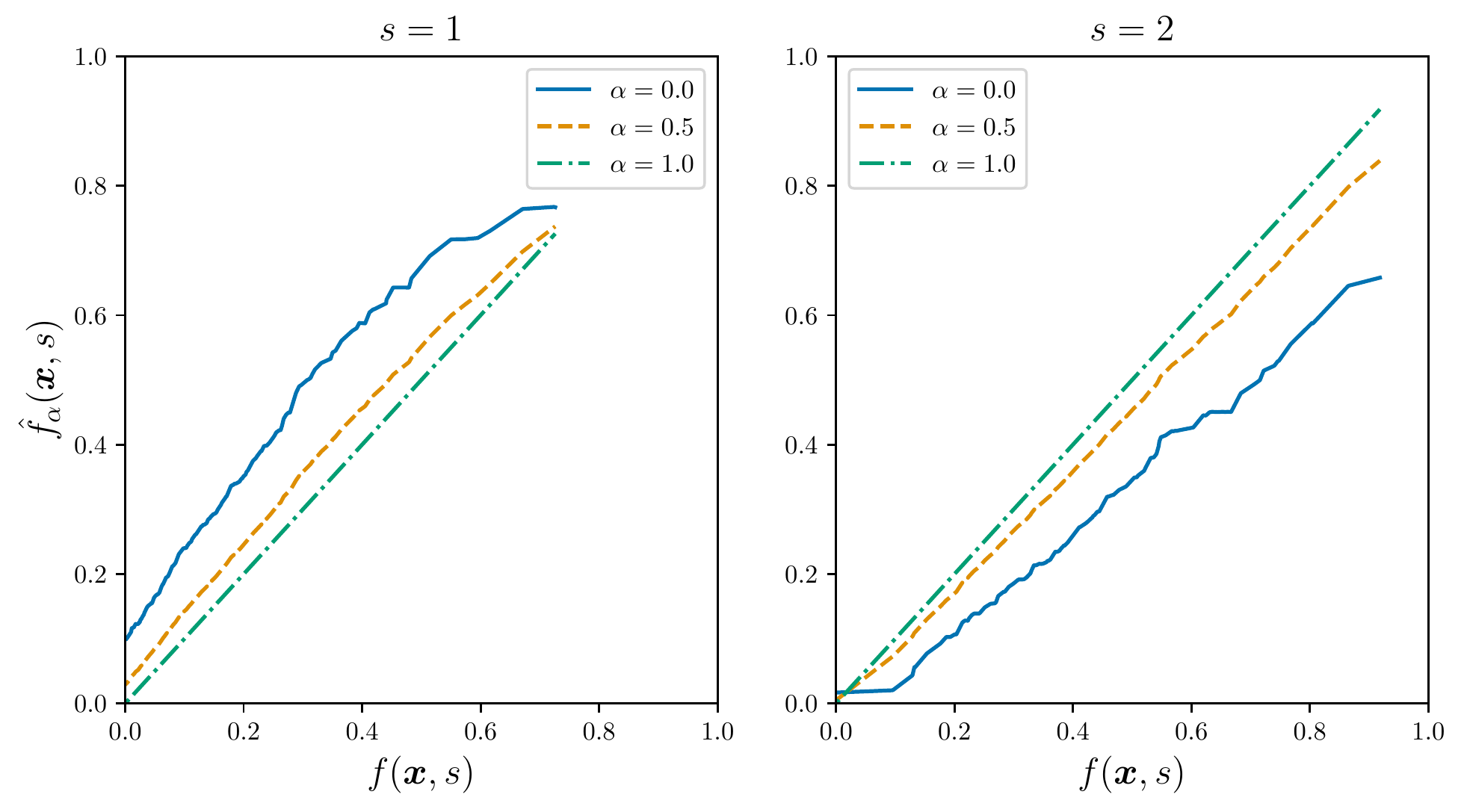}
\caption{Order preservation for the ad-hoc method}
\label{fig:order}
\end{figure}
\begin{figure}[t!]
  \centering
  \includegraphics[width=\textwidth]{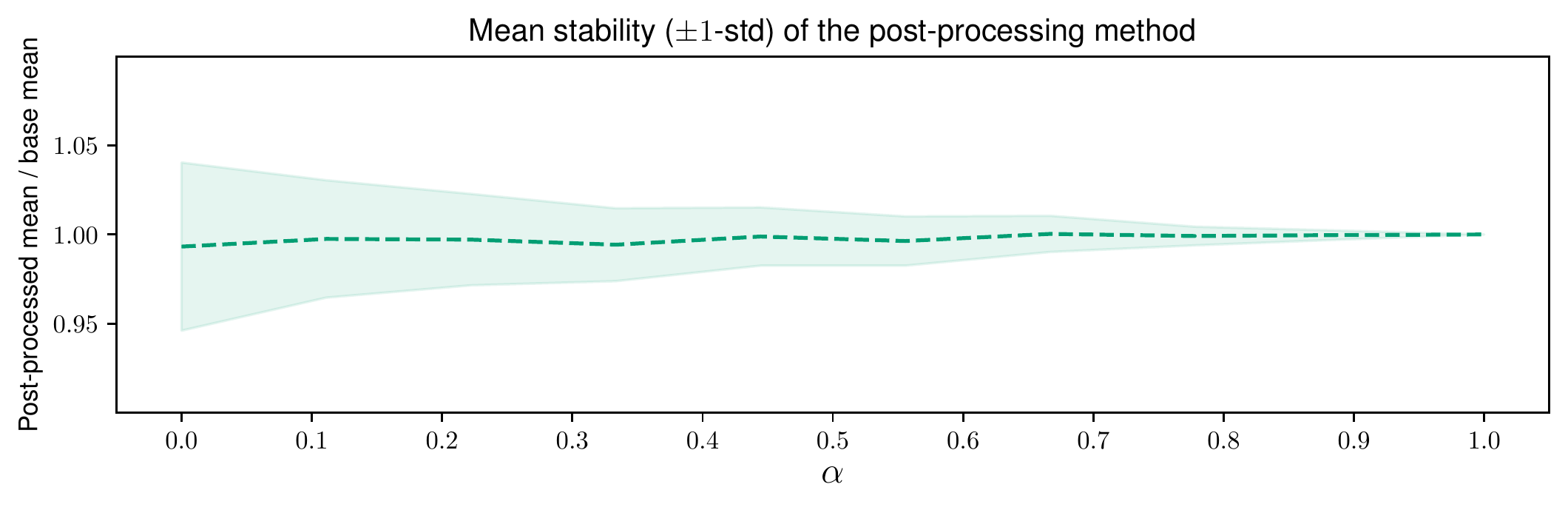}
\caption{Mean stability of the ad-hoc method}
\label{fig:mean}
\end{figure}
Figure~\ref{fig:order} displays the obtained results. We observe, as expected, that the order-preservation property is approximately satisfied.

\myparagraph{Mean stability} Concerning the mean stability, we display the following quantity
\begin{align*}
    \parent{\frac{1}{K}\sum_{s = 1}^K \hat\Exp[\hat{f}_{\alpha}(\bsX, S) \mid S = s]} \bigg / \parent{\frac{1}{K}\sum_{s = 1}^K \hat\Exp[f(\bsX, S) \mid S = s]}\qquad\qquad \alpha \in [0, 1]\enspace,
\end{align*}
where $\hat{\Exp}$ is evaluated on $\data_{\test}$. We evaluate this quantity $100$ times to account for randomness that is present within $\hat{f}_{\alpha}$. Ideally, the above quantity should be equal to one for all $\alpha \in [0, 1]$, which signifies that the mean stays stable along all the values of $\alpha \in [0, 1]$. We display these results on Figure~\ref{fig:mean} and observe that the mean stability property is indeed present with the standard deviation being at most $5\%$.


\section{Conclusion}
\label{sec:concl}
In this work, we proposed a theoretical framework for rigorous analysis of regression problems under fairness requirements. Our framework allows to interpolate between regression under the Demographic Parity constraint and unconstrained regression, using a univariate parameter between zero and one. Within this framework we precisely quantified the risk-fairness trade-off and derived general plug-n-play lower bound. To demonstrate the generality of our results we provided minimax analysis of the linear model with systematic group-dependent bias. 
{\color{red}Finally, we have proposed a post-processing algorithm which enjoys strong theoretical guarantees and performed empirical validations both on simulated and benchmark data.}
For future work, it would be interesting to extend our analysis to other statistical models, providing estimators with high confidence fairness improvement.

\bibliography{biblio}
\bibliographystyle{imsart-nameyear}

}





\newpage
\clearpage
\appendix
\setcounter{page}{1}
\begin{center}
    {\Large\bf Supplementary material for ``A minimax framework for quantifying risk-fairness trade-off in regression''}\\
    \vspace{1cm}
    {\large CONTENTS}
\end{center}

\startcontents[appendices]
\printcontents[appendices]{l}{1}{\setcounter{tocdepth}{2}}

\section{Reminder}
\label{sec:reminder}

\subsection{The Wasserstein-2 distance}\label{sec:Wassersteinreminder} 
\myparagraph{Additional notation} For any $s \in [K]$, we denote by $\mu_{\bsX | s}$ the conditional distribution of the feature vector $\bsX$ knowing the attribute $s$.
For a probability measure $\mu$ on $\bbR^p$ and a measurable function $g : \bbR^p \to \bbR$, we denote by $g \# \mu$ the push-forward (image) measure.
That is for all measurable set $\class{C} \subset \bbR$ it holds that $(g \# \mu) (\class{C}) \eqdef \mu\{\bsx \in \bbR^p\,:\, g(\bsx) \in \class{C}\}$.

We recall basic results on the Wasserstein-2 distance on the real line.
We recall that the \emph{Wasserstein-2} distance between probability distributions $\mu$ and $\nu$ in $\mathcal{P}_2(\mathbb{R}^d)$, the space of measures on $\mathbb{R}^d$ with finite second moment, is defined as
\begin{align}\label{eq:Wasserstein2}
    \sW_2^2(\mu, \nu) \coloneqq \inf_{\pi \in \Gamma(\mu, \nu)} \left\{ \int_{\mathbb{R}^d \times \mathbb{R}^d} \lVert \bsx - \bsy \rVert_2^2 \d\pi(\bsx, \bsy)\right\}\enspace,
\end{align}
where $\Gamma(\mu, \nu)$ denotes the collection of measures on $\mathbb{R}^d \times \mathbb{R}^d$ with marginals $\mu$ and $\nu$.
See \citep{santambrogio2015optimal,villani2003topics} for more details about Wasserstein distances and optimal transport.

The following lemma gives a closed form expression for the Wasserstein-2 distance between two univariate Gaussian distributions.
\begin{lemma}[\cite{frechet1957distance}]
    \label{lem:wass_between_gaus}
    For any $m_0, m_1 \in \bbR, \sigma_0, \sigma_1 \geq 0$ it holds that
    \begin{align*}
        \sW_2^2\Big({\class{N}(m_0, \sigma_0^2),\, \class{N}(m_1, \sigma_1^2)\Big)} = (m_0 - m_1)^2 + (\sigma_0 - \sigma_1)^2\enspace.
    \end{align*}
\end{lemma}

The next lemma gives a closed form expression for the barycenter of $K$ univariate Gaussian distributions. It shows in particular that such barycenter is also a univariate Gaussian distribution.
\begin{lemma}[\cite{agueh2011barycenters}]
    \label{lem:baryc_gaus}
    Let $\bsw \in \bbR^K$ be a probability vector, then the solution of
    \begin{align*}
        \min_{\nu \in \class{P}_2(\bbR)} \sum_{s = 1}^K w_s\sW_2^2\parent{\class{N}(m_s, \sigma_s^2), \nu}\enspace,
    \end{align*}
    is given by $\class{N}(\bar{m}, \bar{\sigma}^2)$ with
    \begin{align*}
        \bar{m} = \sum_{s = 1}^K w_sm_s\quad\text{and}\quad\bar{\sigma} = \sum_{s = 1}^Kw_s\sigma_s\enspace.
    \end{align*}
\end{lemma}

Finally we state a lemma giving an explicit form for the transport map to the barycenter of probability distributions supported on the real line and the corresponding constant speed geodesics. See \cite[Section 6.1]{agueh2011barycenters}.

\begin{lemma}
    \label{lem:A3}
    Let $a_1, \dots, a_K$ be non-atomic probability measures on the real line that have finite second moments, and let $w_1, \dots, w_K$ be positive reals that sum to 1. Denote by $\bar{a}$ a barycenter of those measures (w.r.t. to the Wasserstein-2 distance). For any $s \in [K]$, the transport map from $a_s$ to the barycenter $\bar{a}$ is given by
\begin{align*}
    T_{a_s \rightarrow \bar{a}} = \parent{\sum_{s'=1}^K w_{s'} F_{s'}^{-1} \circ F_{s}} \enspace,
\end{align*}
where $F_s$ is the cumulative distribution function of $a_s$ and $F_s^{-1}$ denotes the generalized inverse of $F_s$ defined as
\begin{align*}
    F_s^{-1}(t) = \inf\{x : F_s(x) \geq t\}\enspace.
\end{align*}
In particular, the constant speed geodesic $\gamma_s(\cdot)$ from $a_s$ to $\bar{a}$ is given by
\begin{align*}
    \gamma_s(t) = ((1-t)\Id + t T_{a_s \rightarrow \bar{a}})\# a_s , \quad t \in [0, 1]\enspace.
\end{align*}
\end{lemma}

\subsection{Tail inequalities}
The next result can be found in~\cite[Lemma 1]{laurent2000adaptive}.
\begin{lemma}
    \label{lem:conc_chi_squared}
    Let $\zeta_1, \ldots, \zeta_p$ be \iid standard Gaussian random variables. Let $\bsa = (a_1, \ldots, a_p)^\top$ be component-wise non-negative, then
    \begin{align*}
        \Probf\parent{\sum_{j = 1}^pa_j(\zeta_j^2 - 1) \geq 2\norm{\bsa}_2\sqrt{t} + 2\norm{\bsa}_{\infty}t} \leq \exp(-t),\quad \forall t \geq 0\enspace.
    \end{align*}
\end{lemma}
In particular, setting $\bszeta = (\zeta_1, \ldots, \zeta_p)^\top$ and applying the previous result with $a_1 = \ldots = a_p = 1$ we get
\begin{align*}
    \Probf\parent{\|\bszeta\|_2^2 \geq p + 2\sqrt{pt} + 2t} \leq \exp(-t),\quad \forall t \geq 0
\end{align*}

We need one result from random matrix theory to control the smallest and largest singular values of a Gaussian matrix, see \cite[Corollary 5.35]{vershynin2010introduction}.
\begin{lemma}
    \label{lem:gauss_matrix}
    Let $\bfA$ be an $N \times m$ matrix whose entries are independent standard normal random variables. Then,
    \begin{align*}
        \Probf\parent{\sigma_{\min}(\bfA) \leq \sqrt{N}
        - \sqrt{m} - t} \vee \Probf\parent{\sigma_{\max}(\bfA) \geq \sqrt{N}
        + \sqrt{m} + t} \leq \exp(-t^2/2), \quad \forall t \geq 0\enspace.\\
    \end{align*}
\end{lemma}

\section{Proofs for Section~\ref{SUBSEC:GENERAL}}\label{sec:general_proof}

\subsection{Auxiliary results}
The next result is taken from~\citep[Theorem 3]{gouic2020price}.
\begin{lemma}\label{lem:squared_loss_wasserstein}
    Let $f: \bbR^p \times [K] \to \bbR$ be any measurable function.  Let Assumption~\ref{as:atomless} be satisfied, then
    \begin{align*}
        \risk(f) \geq \sum_{s = 1}^Kw_s \sW_2^2\bigg(f(\cdot, s)\# \mu_{\bsX|s}, f^*(\cdot, s)\# \mu_{\bsX|s}\bigg)\enspace.
    \end{align*}
\end{lemma}

\begin{lemma}[Minkowski's inequality]
    \label{lem:fancy_triangle}
    Let $(\class{X}, d)$ be a metric space. Fix integers $K \geq 2$, $q \in [1, +\infty)$, a weight vector $\bsw \in \Delta^{K - 1}$ and  define the mapping $d_{\bsw, q} : \class{X}^K \times \class{X}^K \to \bbR$ as
    \begin{align*}
        d_{\bsw, q}(\bsa, \bsb) = \left\{\sum_{s = 1}^K w_s d^q(a_s, b_s)\right\}^{\sfrac{1}{q}},\qquad \text{for any  $\bsa, \bsb \in \class{X}^K$} \enspace.
    \end{align*}
    Then, $d_{\bsw, q}$ is a pseudo-metric on the product space $\class{X}^K$.
\end{lemma}
\begin{proof}
    The mapping $d_{\bsw, q}$ is clearly symmetric and non-negative.
    We only have to check the triangle inequality. Fix arbitrary $\bsa, \bsb, \bsc \in \class{X}^K$. Then, by triangular inequalities on the distance $d$ and H\"older's inequality,
    \begin{align*}
        \sum_{s = 1}^K w_s d^q(a_s, b_s)
        &\leq
        \sum_{s = 1}^K w_s d(a_s, c_s)d^{q-1}(a_s, b_s) + \sum_{s = 1}^K w_s d(c_s, b_s)d^{q-1}(a_s, b_s)\\
        &\leq
        \left(\left\{\sum_{s = 1}^K w_s d^q(a_s, c_s)\right\}^{\sfrac 1 q} + \left\{\sum_{s = 1}^K w_s d^q(c_s, b_s)\right\}^{\sfrac{1}{q}}\right) \left\{\sum_{s = 1}^K w_sd^q(a_s, b_s)\right\}^{1 - \frac{1}{q}}\enspace.
    \end{align*}
    That is, after rearranging we obtain
    \begin{align*}
        d_{\bsw, q}(\bsa, \bsb) = \left\{\sum_{s = 1}^K w_s d^q(a_s, b_s)\right\}^{\sfrac 1 q}
        &\leq \left\{\sum_{s = 1}^K w_sd^q(a_s, c_s)\right\}^{\sfrac{1}{q}} + \left\{\sum_{s = 1}^K w_sd^q(c_s, b_s)\right\}^{\sfrac{1}{q}}\\
        &= d_{\bsw, q}(\bsa, \bsc) + d_{\bsw, q}(\bsc, \bsb)\enspace.
    \end{align*}
\end{proof}

\begin{lemma}
    \label{lem:barycenters_coincide}
    Fix some $q \in [1, +\infty)$.
    Let $\bsa = (a_1, \ldots, a_K) \in \class{X}^K$, $\bsw = (w_1, \ldots, w_K)^\top \in \Delta^{K - 1}$.
    Assume that $\bsb = (b_1, \cdots, b_K) \in \mathcal{X}^K$ satisfies~\eqref{eq:prop_1}--\eqref{eq:prop_2}, then
    \begin{align*}
        \left\{\sum_{s = 1}^K w_s d^q(b_s, C_{\bsb})\right\}^{\sfrac 1 q} =  \left\{\sum_{s = 1}^K w_s d^q(b_s, C_{\bsa})\right\}^{\sfrac 1 q}\enspace.
    \end{align*}
\end{lemma}
\begin{proof}
    Let $C_{\bsb}$ be a barycenter of $(b_s)_{s\in[K]}$ with weights $(w_s)_{s \in [K]}$, then by Lemma~\ref{lem:fancy_triangle} it holds that
    \begin{align}
        \left\{\sum_{s = 1}^K w_s d^q(a_s, C_{\bsb})\right\}^{\sfrac{1}{q}}
        &\leq\left\{\sum_{s = 1}^K w_s d^q(a_s, b_s)\right\}^{\sfrac{1}{q}}
        +
        \left\{\sum_{s = 1}^K w_s d^q(b_s, C_{\bsb})\right\}^{\sfrac{1}{q}}\enspace.\label{eq:fancy_triangle}
    \end{align}
    The following chain of inequalities holds thanks to Eq.~\eqref{eq:fancy_triangle} and properties~\eqref{eq:prop_1}--\eqref{eq:prop_2}
    \begin{align*}
        \left\{\sum_{s = 1}^K w_s d^q(b_s, C_{\bsb})\right\}^{\sfrac{1}{q}}
        &\geq
        \left\{\sum_{s = 1}^K w_s d^q(a_s, C_{\bsb})\right\}^{\sfrac{1}{q}} - \left\{\sum_{s = 1}^K w_s d^q(a_s, b_s)\right\}^{\sfrac{1}{q}}\\
        &\geq
        \left\{\sum_{s = 1}^K w_s d^q(a_s, C_{\bsa})\right\}^{\sfrac{1}{q}} - \left\{\sum_{s = 1}^K w_s d^q(a_s, b_s)\right\}^{\sfrac{1}{q}}\\
        &=
        \frac{1}{{\alpha}^{\sfrac 1 q}}\left\{\sum_{s = 1}^K w_s d^q(b_s, C_{\bsa})\right\}^{\sfrac{1}{q}} - \frac{1 {-} {\alpha}^{\sfrac 1 q}}{{\alpha}^{\sfrac{1}{q}}}\left\{\sum_{s = 1}^K w_s d^q(b_s, C_{\bsa})\right\}^{\sfrac{1}{q}}\\
        &=
        \left\{\sum_{s = 1}^K w_s d^q(b_s, C_{\bsa})\right\}^{\sfrac{1}{q}}\enspace.
    \end{align*}
    The converse inequality follows from the definition of $C_{\bsb}$, which concludes the proof.
\end{proof}

\subsection{Proof of Proposition~\ref{prop:optimal_alpha}}

    Let $\alpha \in [0, 1]$. For any $s \in [K]$, define
    \begin{align}
    a_s  &= f^*(\cdot, s)\# \mu_{\bsX|s} = \nu_s^*\enspace,
    \end{align}
    Let $\gamma_s$ be the (constant-speed) geodesic between $a_s$ and $C_{\bsa}$ \ie $\gamma_s(0) = a_s$, $\gamma_s(1)=C_{\bsa}$ and $\sW_2(\gamma_s(t_1), \gamma_s(t_2)) = |t_2 - t_1| \sW_2(a_s, C_{\bsa})$ for any $t_1, t_2 \in [0, 1]$. Note that the uniqueness of the geodesic come from the particular structure of the Wasserstein-2 space on the real line, see \eg \cite[Section 2.2]{kloeckner2010geometric}.
     We define $b_s \eqdef \gamma_s(1 {-} \sqrt{\alpha})$ for $s \in[K]$.
    Let us show that $\bsb= (b_s)_{s \in [K]}$ satisfies the properties~\eqref{eq:prop_1}--\eqref{eq:prop_2} of the Geometric Lemma~\ref{lem:geometric_general} when considering $\bsa = (a_s)_{s \in [K]}$ with the weights $(w_s)_{s \in [K]}$ and $d \equiv \sW_2$.
    By construction of $b_s = \gamma_s(1 {-} \sqrt{\alpha})$, we have
    \begin{align}
        &\sW_2(b_s, C_{\bsa}) = \sqrt{\alpha}\sW_2(a_s, C_{\bsa})\label{eq:optimal_form_geodesic_properties}\enspace,\\
        &\sW_2(b_s, a_s) = (1 {-} \sqrt{\alpha}) \sW_2(a_s, C_{\bsa})\label{eq:optimal_form_geodesic_properties1}\enspace.
    \end{align}
    This shows that $\bsb = (b_s)_{s \in [K]}$ satisfies \eqref{eq:prop_1} and \eqref{eq:prop_2}.
    Therefore, using Lemma~\ref{lem:geometric_general} we get
    \begin{align}
        \label{eq:using_lemma_geom}
         \sum_{s = 1}^Kw_s\sW_2^2(b_s, a_s)
         =
         \inf_{\bsb \in \mathcal{P}^K_2(\bbR)}\enscond{\sum_{s = 1}^K w_s \sW_2^2(b_s, a_s)}{\sum_{s = 1}^K w_s W^2_2(b_s, C_{\bsb}) {\leq} {\alpha} \sum_{s = 1}^K w_s d^2(a_s, C_{\bsa})}\enspace.
    \end{align}


    Finally, thanks to the Assumption~\ref{as:atomless} which says that that $a_s = \nu_s^*$ is atomless the constant speed geodesic $\gamma_s$ between $a_s$ and $C_{\bsa}$ can be written as
    \begin{align*}
        \gamma_s(t)
        &=
        \left((1-t) \Id + t\left(\sum_{s' = 1}^K w_{s'} F_{a_{s'}}^{-1}\circ F_{a_s}\right) \right)\# a_s\\
        &= \left\{\left((1-t) \Id + t\left(\sum_{s' = 1}^K w_{s'} F_{a_{s'}}^{-1}\circ F_{a_s}\right)\right) \circ f^*(\cdot, s)\right\} \# \mu_{\bsX|s}
        , \quad t \in [0, 1]\enspace.
    \end{align*}
    See Appendix~\ref{sec:Wassersteinreminder} for details about the first equality.
    Substituting $t = 1 {-} \sqrt{\alpha}$ to $\gamma_s$, the expression for $b_s$ is
    \begin{align}
        \label{eq:bs_as_push_forward}
        b_s
        =
        \left\{\left(\sqrt{\alpha} \Id + \big(1 {-} \sqrt{\alpha}\big)\left(\sum_{s' = 1}^K w_{s'} F_{a_{s'}}^{-1}\circ F_{a_s}\right)\right) \circ f^*(\cdot, s)\right\} \# \mu_{\bsX|s}\enspace.
    \end{align}
    We define $f_{\alpha}^*$ for all $(\bsx, s) \in \bbR^p \times [K]$ as
    \begin{align}
        f_{\alpha}^*(\bsx, s) = \sqrt{\alpha} f^*(\bsx, s) + (1 {-} \sqrt{\alpha}) \sum_{s' = 1}^K w_{s'} F_{a_{s'}}^{-1} \left( F_{a_s}(f^*(\bsx, s)\right))\enspace,
    \end{align}
    then after Eq.~\eqref{eq:bs_as_push_forward} it holds that $b_s = f^*_{\alpha}(\cdot, s) \# \mu_{\bsX|s}$ and
    \begin{align}
        \sW_2^2(b_s, a_s) = \Exp\left[(f^*(\bsX, S) -  f^*_{\alpha}(\bsX, S))^2 \mid  S = s\right]\enspace.
    \end{align}
    with $\class{U}(f^*_{\alpha}) = \alpha\,\class{U}(f^*)$.
    Moreover, Lemma~\ref{lem:squared_loss_wasserstein} implies that for any $f$ such that $\class{U}(f) \leq \alpha\,\class{U}(f^*)$ we have
    \begin{align*}
        \Exp(f^*(X, S) - f(X, S))^2
        \geq
        \sum_{s = 1}^K w_s\sW_2^2(b_s, a_s)
        &=
        \sum_{s = 1}^K w_s\Exp\left[(f^*(\bsX, S) -  f^*_{\alpha}(\bsX, S))^2 \mid  S = s\right]\\
        &=
        \risk(f^*_{\alpha})\enspace.
    \end{align*}
    Thus, $f^*_{\alpha}$ is the optimal fair prediction with $\alpha$ relative improvement. The proof is concluded.

\subsection{Proof of Proposition~\ref{prop:regular_pareto}}
\begin{proof}
For $\lambda = 0$, the statement trivially holds. Fix some $\lambda \in (0, +\infty)$. 
Following the Pareto frontier interpretation, the minimum of the functional $\risk(f) + \lambda \cdot \class{U}(f)$ equals to some $F^* \geq 0$ if and only if, in the coordinates $(\risk, \class{U})$, the line $\enscond{(\risk, \class{U}) \in \bbR^2}{\risk + \lambda \cdot \class{U} = F^*}$
is tangent to the Pareto frontier curve $\big\{((1 {-} \sqrt{\alpha})^2,\, \alpha ) \cdot \class{U}(f^*)\big\}_{\alpha \in [0, 1]}$.
Thus, writing this condition explicitly, it should hold that
\begin{align*}
    -\lambda = 1 - \sqrt{\frac{\class{U}(f^*)}{\class{U}(f^{*, \lambda})}}\enspace.
\end{align*}
Hence, at the minimum $f^{*, \lambda}$ of $\risk(f) + \lambda \cdot \class{U}(f)$ it holds that $\class{U}(f^{*, \lambda}) = \tfrac{1}{(1+\lambda)^2}\cdot\class{U}(f^*)$.
\end{proof}

\section{Proof of Theorem~\ref{THM:GENERAL_LOWER}}
\label{sec:general_lower_proof}
    To ease the notation we write $\delta_n$ instead of $\delta_n(\class{F}, \Theta, t)$.
    We also define
    \begin{align*}
    \Psi(\hat f, (f^*, \bstheta) ) \eqdef \Probf_{(f^*, \bstheta)}\parent{\risk^{\sfrac{1}{2}}(\hat f) \geq \delta_n^{\sfrac{1}{2}} \vee (1 {-} \sqrt{\alpha}){\class{U}^{\sfrac{1}{2}}({f^*})}}\enspace.
    \end{align*}
    We split the proof according to two complementary cases.

    {\bf Case 1:} there exists $(f^*, \bstheta) \in {\class{F}} \times \Theta$ such that $\delta_n \leq (1 {-} \sqrt{\alpha})^2\class{U}({f^*})$.
    In this case, for such couple $(f^*, \bstheta) \in {\class{F}} \times \Theta$ and for any estimator $\hat{f} \in \widehat{\class{F}}_{(\alpha, t')}$  we have
    \begin{align*}
         \Psi(\hat f, (f^*, \bstheta) )
         &\geq
         \Probf_{(f^*, \bstheta)}\parent{\risk^{\sfrac{1}{2}}(\hat f) \geq \delta_n^{\sfrac{1}{2}} \vee (1 {-} \sqrt{\alpha}){\class{U}^{\sfrac{1}{2}}({f^*})},\, \class{U}(\hat f) \leq \alpha\,\class{U}(f^*)}\\
         &\leftstackrel{\text{def. of $f^*_{\alpha}$}}{\geq}
          \Probf_{(f^*, \bstheta)}\parent{\risk^{\sfrac{1}{2}}(f^*_{\alpha}) \geq \delta_n^{\sfrac{1}{2}} \vee (1 {-} \sqrt{\alpha}){\class{U}^{\sfrac{1}{2}}({f^*})},\, \class{U}(\hat f) \leq \alpha\,\class{U}(f^*)}\\
         &\leftstackrel{\text{Lemma~\ref{lem:distance_fair_and_almost}}}{=}
         \Probf_{(f^*, \bstheta)}\parent{\class{U}(\hat f) \leq \alpha\,\class{U}({f^*})}\ind{\delta_n \leq (1 {-} \sqrt{\alpha})^2\class{U}({f^*})} \enspace.
    \end{align*}
    Note that by definition of $\widehat{\class{F}}_{(\alpha, t')}$ it holds that
    \begin{align*}
        \forall \hat{f} \in \widehat{\class{F}}_{(\alpha, t')},\forall (f^*, \bstheta) \in \class{F} \times \Theta,\quad \Probf_{(f^*, \bstheta)}\parent{\class{U}(\hat f) \leq \alpha\,\class{U}({f^*})} \geq 1 - t'\enspace.
    \end{align*}
    Since in the considered case there exists a couple $(f^*, \bstheta) \in \widehat{\class{F}}_{(\alpha, t')} \times \Theta$ such that $\delta_n \leq (1 {-} \sqrt{\alpha})^2\class{U}({f^*})$, by definition of $\widehat{\class{F}}_{(\alpha, t')}$ we have
    \begin{align}
    \label{eq:lower1}
        \inf_{\hat f \in \widehat{\class{F}}_{(\alpha, t')}}\sup_{(f^*, \bstheta) \in \class{F}\times \Theta}\Psi(\hat f, (f^*, \bstheta) ) \geq 1 - t'\enspace.
    \end{align}

    {\bf Case 2:} for any couple $(f^*, \bstheta) \in \class{F}\times \Theta$ it holds that $\delta_n > (1 {-} \sqrt{\alpha})^2\class{U}({f^*})$.
    In this case, for any couple $(f^*, \bstheta) \in {\class{F}} \times \Theta$ and for any estimator $\hat{f} \in \widehat{\class{F}}_{(\alpha, t')}$,
    \begin{align*}
        \Psi(\hat f, (f^*, \bstheta) ) = \Probf_{(f^*, \bstheta)}\parent{\risk(\hat f) \geq \delta_n}\enspace.
    \end{align*}
    By definition of $\delta_n$ it holds in this case that
    \begin{align}
        \label{eq:lower2}
        \inf_{\hat f \in \widehat{\class{F}}_{(\alpha, t')}}\sup_{(f^*, \bstheta) \in \class{F}\times \Theta}\Psi(\hat f, (f^*, \bstheta) )
        &\geq
        \inf_{\hat f}\sup_{(f^*, \bstheta) \in \class{F}\times \Theta}\Psi(\hat f, (f^*, \bstheta) ) \nonumber\\
        &=
        \inf_{\hat f}\sup_{(f^*, \bstheta) \in \class{F}\times \Theta}\Probf_{(f^*, \bstheta)}\parent{\risk(\hat f)
         \geq \delta_n}
        \geq t\enspace.
    \end{align}

    Putting two cases together, and in particular using Eqs.~\eqref{eq:lower1} and~\eqref{eq:lower2} we obtain
    \begin{align*}
        \inf_{\hat f \in \widehat{\class{F}}_{(\alpha, t')}}\sup_{(f^*, \bstheta) \in \class{F}\times \Theta}\Psi(\hat f, (f^*, \bstheta) )
        \geq
        \begin{cases}
            1- t' &\text{if}\quad \exists(f^*, \bstheta) \in \class{F}\times \Theta \text{ s.t. } \delta_n \leq (1 {-} \sqrt{\alpha})^2\class{U}({f^*})\\
            t &\text{otherwise}
        \end{cases}\enspace.
    \end{align*}
    We conclude the proof observing that the \rhs of the last inequality is lower bounded by $t \wedge (1 - t')$.


\section{Proofs for Section~\ref{SEC:LINEAR}}
\label{sec:proofs_linear}
\myparagraph{Additional notation}
We denote by $\mathbb{S}^{p - 1}$ the unit sphere in $\bbR^p$.
For any matrix $\mathbf{A}$ we denote by $\|\mathbf{A}\|_{\op}$, the operator norm of $\mathbf{A}$.
We denote by $\chi^2(p)$ the standard chi-square distribution with $p$ degrees of freedom and by $\mathcal{N}(\boldsymbol{\mu}, \bfSigma)$ the multivariate Gaussian with mean $\boldsymbol{\mu}$ and covariance $\bfSigma$.
 We denote by $\mathbf{I}_p$  the identity matrix of size $p \times p$.

\subsection{Proof of Lemma~\ref{lem:unfairness_estimator1}}
Throughout the proof we implicitly condition on the observations.
Let $\tau \in [0, 1]$. For each $s \in [K]$ we set $\hat m_s = \sqrt{\tau}\hb_s + (1 - \sqrt{\tau})\sum_{s = 1}^Kw_s\hb_s$.
Note that for all $s \in [K]$, $(\hat f_{\tau}(\bsX, S) | S = s) \sim \mathcal{N}(\hat{m}_s, \scalarin{\hbeta}{\bfSigma \hbeta})$.
Therefore, by the definition of the unfairness and Lemma~\ref{lem:baryc_gaus}
\begin{align*}
    \class{U}(\hat f_{\tau})
    &= \min_{\nu} \sum_{s = 1}^Kw_s\sW_2^2\parent{\mathcal{N}(\hat m_s, \scalarin{\hbeta}{\bfSigma \hbeta}),\, \nu}\\
    &= \sum_{s = 1}^Kw_s\sW_2^2\parent{\mathcal{N}(\hat m_s, \scalarin{\hbeta}{\bfSigma \hbeta}),\, \mathcal{N}(\bar{m}, \scalarin{\hbeta}{\bfSigma \hbeta})}\enspace,
\end{align*}
where $\bar{m} = \sum_{s = 1}^Kw_s \hat m_s$.
We conclude the proof by noticing that thanks to Lemma~\ref{lem:wass_between_gaus} it holds that
\begin{align*}
    \sW_2^2\parent{\mathcal{N}(\hat m_s, \scalarin{\hbeta}{\bfSigma \hbeta}), \mathcal{N}(\bar{m}, \scalarin{\hbeta}{\bfSigma \hbeta})}
    &= (\hat m_s - \bar{m})^2\\
    &= \left\{\sqrt{\tau}\hb_s + (1 - \sqrt{\tau})\sum_{s = 1}^Kw_s\hb_s - \sum_{s = 1}^Kw_s\hb_s\right\}^2\enspace.
\end{align*}
The proof is concluded.

\subsection{Auxiliary results for Theorem~\ref{thm:real_final}}

\begin{lemma}[Fixed design analysis]
    \label{lem:fixed_design}
    Define the following matrix of size $(p + K) \times (p + K)$
   \begin{align*}
   \widehat\bfPsi =
       \left[
        \begin{array}{c|c}
          \frac{1}{2}\sum_{s = 1}^K w_s\bfX_s^\top\bfX_s / n_s & \bfO \\
          \hline
          \bfO^\top & \frac{1}{2}\bfW
        \end{array}
     \right]\enspace,
   \end{align*}
   where $\bfO = [w_1\bar{\bsX}_1, \ldots, w_K\bar{\bsX}_K] \in \bbR^{p \times K}$ and $\bfW = \diag(w_1, \ldots, w_K)$.
    For all $t\geq 0$ it holds that
    \begin{align*}
        \Probf\parent{\|\widehat{\bfPsi}^{1/2}\hbsdelta\|_2^2 \geq \sigma^2\left\{\parent{\frac{p}{n} + \frac{K}{n}} + 2\parent{\sqrt{\frac{p}{n}} + \sqrt{\frac{K}{n}}}\sqrt{\frac{t}{n}}
        +
        4\frac{t}{n} \right\} \,\bigg|\, \bfX_{1 : K}} \leq 2\exp(-t)\enspace,
    \end{align*}
    where $\hbsdelta = (\hbeta - \bbeta^*, \hbsb - \bsb^*) \in \mathbb{R}^{p} \times \bbR^K$ and $\bfX_{1 : K} = (\bfX_1, \ldots, \bfX_K)$.
\end{lemma}
\begin{proof}
   By optimality of $(\hbeta, \hbsb)$ and the linear model assumption in Eq.~\eqref{eq:model_linear_vector} it holds that
   \begin{align*}
       \sum_{s = 1}^K w_s \norm{\bsY_s - \bfX_s \hbeta - \hb_s\1_{n_s}}_{n_s}^2 \leq \sum_{s = 1}^K w_s \norm{\bsxi_s}_{n_s}^2\enspace.
   \end{align*}
   After simplification, the above yields
   \begin{align*}
    \sum_{s = 1}^K w_s \norm{\bfX_s(\bbeta^* - \hbeta) + (b^*_s - \hb_s)\1_{n_s}}_{n_s}^2
       &\leq
       2\sum_{s = 1}^K w_s \scalar{\bfX_s(\hbeta - \bbeta^*) + (\hb_s - b^*_s)\1_{n_s}}{\bsxi_s / n_s}\\
       &=
        2\scalar{\hbeta - \bbeta^*}{\sum_{s = 1}^K\bfX_s^\top\bsxi_s / n}
        +
        2\sum_{s = 1}^K w_s(\hb_s - b^*_s)\bar{\bsxi}_s\enspace,
   \end{align*}
   where $\bar{\bsxi}_s = (\sfrac{1}{n_s})\sum_{i = 1}^{n_s} (\bsxi_s)_i$.
  Using Young's inequality, we can write
   \begin{align*}
       2\scalar{\hbeta - \bbeta^*}{\sum_{s = 1}^K\bfX_s^\top\bsxi_s / n}
       &\leq
       \frac{1}{2}\sum_{s = 1}^Kw_s\|\bfX_s(\bbeta^* - \hbeta)\|_{n_s}^2 +
       2\parent{\frac{\scalar{\hbeta - \bbeta^*}{\sum_{s = 1}^K\bfX_s^\top\bsxi_s / n}}{\sqrt{\sum_{s = 1}^Kw_s\|\bfX_s(\bbeta^* - \hbeta)\|_{n_s}^2}}}^2\\
       &\leq
        \frac{1}{2}\sum_{s = 1}^Kw_s\|\bfX_s(\bbeta^* - \hbeta)\|_{n_s}^2 +
       2\sup_{\bsdelta \in \bbR^p}\parent{\frac{\scalar{\bsdelta}{\sum_{s = 1}^K\bfX_s^\top\bsxi_s / n}}{\sqrt{\sum_{s = 1}^Kw_s\|\bfX_s\bsdelta\|_{n_s}^2}}}^2\enspace.
   \end{align*}
   We also observe that again thanks to Young's inequality
   \begin{align*}
       2\sum_{s = 1}^K w_s(\hb_s - b^*_s)\bar{\bsxi}_s \leq \frac{1}{2}\sum_{s = 1}^Kw_s(\hb_s - b^*_s)^2 + 2\sum_{s = 1}^Kw_s\bar{\bsxi}_s^2\enspace.
   \end{align*}
   Putting everything together, we have shown that
   \begin{align}
        \label{eq:fixed_1}
       \|\widehat{\bfPsi}^{1/2}\hbsdelta\|_2^2 \leq 2 \sup_{\bsdelta \in \bbR^p}\parent{\frac{\scalar{\bsdelta}{\sum_{s = 1}^K\bfX_s^\top\bsxi_s / n}}{\sqrt{\sum_{s = 1}^Kw_s\|\bfX_s\bsdelta\|_{n_s}^2}}}^2 + 2 \sum_{s = 1}^Kw_s\bar{\bsxi}_s^2\enspace.
   \end{align}
    Notice that since $\bsxi_s \sim \mathcal{N}(\boldsymbol{0}, \sigma^2 \mathbf{I}_{n_s})$, then conditionally on $\bfX_1, \ldots, \bfX_K$,
   \begin{align*}
       \sum_{s = 1}^K\bfX_s^\top\bsxi_s / n \leftstackrel{d}{=} \frac{\sigma}{n}\parent{\sum_{s = 1}^K\bfX_s^\top\bfX_s }^{1/2}\bszeta\enspace,
   \end{align*}
   where $\bszeta \sim \mathcal{N}(\boldsymbol{0}, \mathbf{I}_p)$.
   Besides, since $w_s = \sfrac{n_s}{n}$, it holds for all $\bsdelta \in \bbR^p$ that
   \begin{align*}
       \sum_{s = 1}^Kw_s\|\bfX_s\bsdelta\|_{n_s}^2 = \bsdelta^\top \parent{\frac{1}{n}\sum_{s = 1}^K\bfX_s^\top \bfX_s}\bsdelta = \norm{\parent{\frac{1}{n}\sum_{s = 1}^K\bfX_s^\top \bfX_s}^{1/2}\bsdelta}_2^2\enspace.
   \end{align*}
   The above implies that conditionally on $\bfX_1, \ldots, \bfX_K$,
   \begin{align}
        \label{eq:fixed_2}
        \sqrt{U} \eqdef \sup_{\bsdelta \in \bbR^p}\frac{\scalar{\bsdelta}{\sum_{s = 1}^K\bfX_s^\top\bsxi_s / n}}{\sqrt{\sum_{s = 1}^Kw_s\|\bfX_s\bsdelta\|_{n_s}^2}} &\leftstackrel{d}{=} \frac{\sigma}{\sqrt{n}} \sup_{\bsdelta \in \bbR^p} \frac{\scalar{\parent{\sum_{s = 1}^K\bfX_s^\top \bfX_s}^{1/2} \bsdelta}{\bszeta}}{\norm{\parent{\sum_{s = 1}^K\bfX_s^\top \bfX_s}^{1/2}\bsdelta}_2}\enspace.
   \end{align}
   Note that for any random variable $\bszeta$ taking values in $\mathbb{R}^p$,
   \begin{align}
    \sup_{\bsdelta \in \bbR^p} \frac{\scalar{\parent{\sum_{s = 1}^K\bfX_s^\top \bfX_s}^{1/2} \bsdelta}{\bszeta}}{\norm{\parent{\sum_{s = 1}^K\bfX_s^\top \bfX_s}^{1/2}\bsdelta}_2} \leq \lVert \bszeta \rVert_2\enspace \text{almost surely}.
   \end{align}
   Furthermore, recalling that $\bar{\bsxi}_s \sim \mathcal{N}(\boldsymbol{0}, \sfrac{1}{n_s})$ we get
   \begin{align}
    \label{eq:fixed_3}
       V \eqdef \sum_{s = 1}^Kw_s\bar{\bsxi}_s^2 \sim \frac{\sigma^2}{n}\chi^2(K)\enspace.
   \end{align}
   For any $u, v \in \bbR$ it holds that
   \begin{align*}
       \Probf\parent{\|\widehat{\bfPsi}^{1/2}\hbsdelta\|_2^2 \geq 2(u + v) \,\big |\, \bfX_{1 : K}} &\leftstackrel{\text{\eqref{eq:fixed_1}}}{\leq} \Probf\parent{ 2(U + V) \geq 2(u + v) \,\big |\, \bfX_{1 : K}} \\
       &\leftstackrel{(a)}{\leq}
       \Probf\parent{\frac{\sigma^2}{n}\chi^2(p) \geq u \,\big |\, \bfX_{1 : K}}
       +
       \Probf\parent{\frac{\sigma^2}{n}\chi^2(K) \geq v \,\big |\, \bfX_{1 : K}}\enspace,
   \end{align*}
   where inequality $(a)$ uses Eqs.~\eqref{eq:fixed_2} and~\eqref{eq:fixed_3} and the fact that $\Probf(U + V \geq u + v) \leq \Probf(U \geq u) + \Probf(V \geq v)$ for all random variables $U, V$ and all $u, v \in \bbR$.
   Finally, setting $u = u_n(\sigma, p, t), v = v_n(\sigma, p, t)$ with
   \begin{align*}
       &u_{n}( \sigma, p, t)
       = \frac{\sigma^2p}{n} +
       2\sigma^2\sqrt{\frac{p}{n}}\sqrt{\frac{t}{n}}
       +
       2\frac{\sigma^2t}{n},\quad
       v_{n} (\sigma, K, t) =  \frac{\sigma^2K}{n} +
2\sigma^2\sqrt{\frac{K}{n}}\sqrt{\frac{t}{n}} + 2\frac{\sigma^2t}{n}\enspace,
   \end{align*}
   we obtain the stated result after application of Lemma~\ref{lem:conc_chi_squared} in appendix
   \begin{align*}
       \Probf\parent{\|\widehat{\bfPsi}^{1/2}\hbsdelta\|_2^2 \geq 2(u_n(\sigma, p, t) + v_n(\sigma, p, t)) \,\big |\, \bfX_{1 : K}} \leq 2\exp(-t)\enspace.
   \end{align*}
\end{proof}

\begin{theorem}[From fixed to random design]
    \label{thm:random_design}
    Define,
    \begin{align*}
        \delta_n(p, K, t) = 8\parent{\frac{p}{n} + \frac{K}{n}} + 16\parent{\sqrt{\frac{p}{n}} + \sqrt{\frac{K}{n}}}\sqrt{\frac{t}{n}}
        +
        \frac{32t}{n}\enspace.
    \end{align*}
    Consider $p, K \in \bbN, t \geq 0$ and define $\theta(p, K, t) = \sfrac{(4\sqrt{K} + 5\sqrt{t} + 6\sqrt{p})}{(\sqrt{p}+\sqrt{t})}$.
    Assume that $\sqrt{n} \geq 2\sfrac{(\sqrt{p}+\sqrt{t})}{(\theta(p, K, t) - \sqrt{\theta^2(p, K, t) - 3})}$, then with probability at least $1 - 4\exp(-t/2)$
    \begin{align*}
        \|\bfSigma^{1/2}(\bbeta^* - \hbeta)\|_2^2 + \sum_{s = 1}^K w_s(b^*_s - \hb_s)^2 \leq \sigma^2\delta_n(p, K, t)\enspace.
    \end{align*}
\end{theorem}
\begin{proof}
 Define the $(p + K) \times (p + K)$ matrix
    \begin{align}
        \bfPsi =
       \frac{1}{2}\left[
        \begin{array}{c|c}
          \bfSigma & \mathbf{0} \\
          \hline
          \mathbf{0} & \bfW
        \end{array}
     \right]\enspace,
    \end{align}
    then under notation of Lemma~\ref{lem:fixed_design} we can write
    \begin{align}
       \|\widehat{\bfPsi}^{1/2}\hbsdelta\|_2^2
       &=
       \hbsdelta^\top {\bfPsi}^{1/2}{\bfPsi}^{-1/2}\widehat{\bfPsi}{\bfPsi}^{-1/2}{\bfPsi}^{1/2} \hbsdelta\nonumber\\
       &=
       \hbsdelta^\top {\bfPsi}^{1/2}{\bfPsi}^{-1/2}\left(\widehat{\bfPsi} - \bfPsi \right){\bfPsi}^{-1/2}{\bfPsi}^{1/2} \hbsdelta + \hbsdelta^\top{\bfPsi}\hbsdelta\nonumber\\
       &\geq
       {
       \parent{1 + \lambda_{\min}\parent{{\bfPsi}^{-1/2}\left(\widehat{\bfPsi} - \bfPsi \right){\bfPsi}^{-1/2}}}\|{\bfPsi}^{1/2}\hbsdelta\|_2^2} \enspace.\label{eq:from_fixed_to_random}
    \end{align}
    If we set $\widehat{\bfSigma} = \sum_{s = 1}^Kw_s\bfX_s^\top \bfX_s / n_s$, then
    \begin{align*}
        {\bfPsi}^{-1/2}\left(\widehat{\bfPsi} - \bfPsi \right){\bfPsi}^{-1/2} =
        \left[
        \begin{array}{c|c}
          \bfSigma^{-1/2}\parent{\widehat{\bfSigma} -\bfSigma}\bfSigma^{-1/2} & 2\bfSigma^{-1/2}\bfO\bfW^{-1/2} \\
          \hline
          2\bfW^{-1/2}\bfO^\top\bfSigma^{-1/2} & \mathbf{0}
        \end{array}
     \right]\enspace.
    \end{align*}
    Furthermore, by Courant-Fisher theorem it holds that
    \begin{align}
        \label{eq:random1}
        {
        \lambda_{\min}\parent{{\bfPsi}^{-1/2}\left(\widehat{\bfPsi} - \bfPsi \right){\bfPsi}^{-1/2}}
        \geq
        \lambda_{\min}\parent{\bfSigma^{-1/2}\parent{\widehat{\bfSigma} -\bfSigma}\bfSigma^{-1/2}} - 4\|\bfSigma^{-1/2}\bfO\bfW^{-1/2}\|_{\op}\enspace.
        }
    \end{align}
    Using the definition of $\bfO$ we can write
    \begin{align*}
        \bfSigma^{-1/2}\bfO\bfW^{-1/2} = [w_1^{1/2}\bfSigma^{-1/2}\bar{\bsX}_1, \ldots, w_K^{1/2}\bfSigma^{-1/2}\bar{\bsX}_K]\enspace.
    \end{align*}
    Note that the random variable on right hand side of Eq.~\eqref{eq:random1} is independent from $\bsxi_1, \ldots, \bsxi_K$.
    Recall that since $w_s =\sfrac{n_s}{n}$ and $\bar{\bsX}_s \sim \class{N}(\boldsymbol{0}, \bfSigma / n)$, then for all $s = 1,\ldots, K$ it holds that
    \begin{align*}
        w_s^{1/2}\bfSigma^{-1/2}\bar{\bsX}_s \sim \class{N}(\boldsymbol{0}, \mathbf{I}_p / n)\enspace,
    \end{align*}
    and these vectors are independent.
    Hence, the matrix $\bfSigma^{-1/2}\bfO\bfW^{-1/2} \in \bbR^{p \times K}$ has \iid Gaussian entries with variance $\sfrac{1}{n}$.
    Therefore, by Lemma~\ref{lem:gauss_matrix} we get
    \begin{align}
        \label{eq:off_diag_conc}
        \Probf\parent{\|\bfSigma^{-1/2}\bfO\bfW^{-1/2}\|_{\op} \geq \sqrt\frac{p}{n} + \sqrt\frac{K}{n} + \sqrt\frac{t}{n}} \leq \exp(-t/2)\enspace.
    \end{align}

    Furthermore, we observe that
    \begin{align*}
        \bfSigma^{-1/2}\widehat{\bfSigma}\bfSigma^{-1/2} \stackrel{d}{=} \frac{1}{n}\sum_{i = 1}^n \bszeta_i\bszeta_i^\top\enspace,
    \end{align*}
    where $\bszeta_i \simiid \mathcal{N}(\boldsymbol{0}, \mathbf{I}_p)$.
    It implies that
    \begin{align*}
        \bfSigma^{-1/2}\parent{\widehat{\bfSigma} -\bfSigma}\bfSigma^{-1/2} \stackrel{d}{=} \frac{1}{n}\sum_{i = 1}^n \bszeta_i\bszeta_i^\top - \mathbf{I}_p  = \frac{1}{n}\parent{\mathbf{Z}^\top \mathbf{Z} - n\mathbf{I}_p}\enspace,
    \end{align*}
    where $\mathbf{Z}$ is a matrix of size $n \times p$ with $i^{\text{th}}$-row being equal to $\bszeta_i^\top$.
    Note that the spectral theorem and the relation between eigenvalues of $\mathbf{Z}^\top \mathbf{Z}$ and the singular values of $\mathbf{Z}$ imply that
    \begin{align*}
        {
        n\lambda_{\min}\parent{\bfSigma^{-1/2}\parent{\widehat{\bfSigma} -\bfSigma}\bfSigma^{-1/2}} \stackrel{d}{=}
        \lambda_{\min}\parent{\mathbf{Z}^\top \mathbf{Z}- n\mathbf{I}_p}
        =
        \sigma_{\min}^2(\mathbf{Z}) - n      \enspace.
        }
    \end{align*}
    where $\sigma_{\min}(\mathbf{Z})$ is the maximal singular value of $\mathbf{Z}$.
    Applying Lemma~\ref{lem:gauss_matrix} from appendix we get for all $t \geq \sqrt{n} - \sqrt{p}$ that $\Probf\parent{\tfrac{1}{n} \sigma^2_{\min}(Z) \leq \tfrac{1}{n}(\sqrt{n} - \sqrt{p} - t)^2}$ equals to
    \begin{align*}
        \Probf\parent{\frac{1}{n} (\sigma^2_{\min}(Z) - n) \leq \frac{p}{n}  +2\sqrt\frac{p}{n}\frac{t}{\sqrt{n}} + \frac{t^2}{n} - 2 \sqrt{\frac{p}{n}} - 2 \frac{t}{\sqrt{n}}}
        \leq \exp(-t^2/2)\enspace.
    \end{align*}
    Changing variables $t^2 \mapsto t$ we get
    \begin{align}
          \label{eq:diag_conc}
          \Probf\parent{\lambda_{\min}\parent{\bfSigma^{-1/2}\parent{\widehat{\bfSigma} -\bfSigma}\bfSigma^{-1/2}} \leq \frac{p}{n}  + 2\sqrt\frac{p}{n}\sqrt\frac{t}{n} + \frac{t}{n} - 2\sqrt\frac{p}{n} -\sqrt\frac{t}{n}}
        &\leq \exp(-t/2)\enspace.
    \end{align}
    Combining Eqs.~\eqref{eq:random1},\eqref{eq:off_diag_conc}, and~\eqref{eq:diag_conc} we deduce that
    \begin{align*}
        \Probf\parent{\lambda_{\min}\parent{{\bfPsi}^{-1/2}\left(\widehat{\bfPsi} - \bfPsi \right){\bfPsi}^{-1/2}} \leq \psi_n(p, K. t)} \leq 2\exp(-t/2)\enspace,
    \end{align*}
    where $\psi_n(p, K, t) = \tfrac{p}{n} - 6\sqrt{\tfrac{p}{n}} + 2\sqrt{\tfrac{p}{n}}\sqrt{\tfrac{t}{n}} - 4\sqrt{\tfrac{K}{n}} + \tfrac{t}{n} - 5\sqrt{\tfrac{t}{n}}$.
    Applying Lemma~\ref{lem:condition_nkpt} we deduce that under the assumption on $n$ that $\psi_n(p, K, t) \geq -0.75$.
    Thus,
    \begin{align*}
        \Probf\parent{\lambda_{\min}\parent{{\bfPsi}^{-1/2}\left(\widehat{\bfPsi} - \bfPsi \right){\bfPsi}^{-1/2}} \leq -0.75} \leq 2\exp(-t/2)\enspace.
    \end{align*}
    Combining the above fact with Eq.~\eqref{eq:from_fixed_to_random} and Lemma~\ref{lem:fixed_design} we conclude that with probability at least $1 - 2\exp(-t) - 2\exp(- t / 2)$
    \begin{align*}
        \|{\bfPsi}^{1/2}\hbsdelta\|_2^2 \leq \sigma^2\left\{4\parent{\frac{p}{n} + \frac{K}{n}} + 8\parent{\sqrt{\frac{p}{n}} + \sqrt{\frac{K}{n}}}\sqrt{\frac{t}{n}}
        +
        16\frac{t}{n}\right\} = \sigma^2\frac{\delta_n(p, K, t)}{2}\enspace.
    \end{align*}
    The statement of the lemma follows from the fact that
    \begin{align*}
        \|{\bfPsi}^{1/2}\hbsdelta\|_2^2 = \frac{1}{2}\parent{\|\bfSigma^{1/2}(\bbeta^* - \hbeta)\|_2^2 + \sum_{s = 1}^K w_s(b^*_s - \hb_s)^2}\enspace.
    \end{align*}

\end{proof}

\begin{lemma}
    \label{lem:condition_nkpt}
    Consider $p, K \in \bbN, t \geq 0$ and define
    \begin{align*}
        \theta(p, K, t) = \frac{4\sqrt{K} + 5\sqrt{t} + 6\sqrt{p}}{\sqrt{p}+\sqrt{t}}\enspace.
    \end{align*}
    For all $n, K, p \in \bbN, t \geq 0$, the following two conditions are equivalent
    \begin{itemize}
    \item $n \geq \left(\frac{2(\sqrt{p}+\sqrt{t})}{\theta(p, K, t) - \sqrt{\theta^2(p, K, t) - 3}}\right)^2$;
    \item $\frac{p}{n} - 6\sqrt\frac{p}{n} + 2\sqrt\frac{p}{n}\sqrt\frac{t}{n} - 4\sqrt{\frac{K}{n}} + \frac{t}{n} - 5\sqrt\frac{t}{n} \geq -0.75$.
    \end{itemize}
\end{lemma}
\begin{proof}
To simplify the notation and to save space we write $\theta$ instead of $\theta(p, K, t)$.
Let $x = n^{-1/2}$, we want to solve
\begin{align*}
    x^2(\sqrt{p}+\sqrt{t})^2 - x(6\sqrt{p} + 4 \sqrt{K} + 5\sqrt{t}) \geq -0.75
\end{align*}
Set $y=x(\sqrt{p}+\sqrt{t})$,
then thanks to the definition of $\theta$, the previous inequality amounts to
\begin{align*}
    y^2 - \theta y + 0.75 \geq 0\enspace.
\end{align*}
The roots of the polynomial above are
\begin{align*}
    x_-, x_+ = \frac{\theta \pm \sqrt{\theta^2 - 3}}{2}\enspace,
\end{align*}
which are both positive.
The polynomial is non-negative outside the interval $(x_-, x_+) \subset \bbR_+$. Hence, a sufficient condition is to have
\begin{align*}
    y \leq \frac{\theta - \sqrt{\theta^2 - 3}}{2}\enspace.
\end{align*}
Substituting $x = n^{-1/2}$ and the expression for $\theta$ we conclude.

\end{proof}

\begin{lemma}[General unfairness control]
 \label{lem:unfairness_estimator2}
 Under notation of Lemma~\ref{lem:unfairness_estimator1} it holds that, for any $\alpha \in [0, 1]$,
 \begin{align}
     \label{eq:unfairness_bound1}
     \class{U}(\hat f_{\alpha}) \leq \alpha\,\class{U}(f^*)\left\{1 +  \nur\sqrt{\frac{\sum_{s = 1}^Kw_s(\hb_s - b^*_s)^2}{\sigma^2}}\right\}^2, \quad \text{almost surely.}
 \end{align}
Moreover,
 \begin{align}
    \label{eq:unfairness_bound2}
     \abs{\class{U}^{\sfrac{1}{2}}(\hat f_{1}) - \class{U}^{\sfrac{1}{2}}(f^*)} \leq \left\{{\sum_{s = 1}^Kw_s(\hb_s - b^*_s)^2} \right\}^{\sfrac{1}{2}}, \quad \text{almost surely.}
 \end{align}
\end{lemma}
\begin{proof}
Let $U$ and $V$ be discrete random variables such that $\Probf(U = \hat b_s, V = b^*_{s'}) = w_s \delta_{s,s'}$, for any $s,s' \in [K]$. Note that, in particular, $\Probf(U = \hat b_s) = w_s$ and $\Probf(V = b^*_s) = w_s$.
Then, according to Lemma~\ref{lem:unfairness_estimator1} and the definition of $\hat f_\alpha$ it holds that
\begin{align*}
    \class{U}(\hat f_{\alpha}) = \alpha\Var(U) \quad \text{ and } \quad \class{U}(\hat f_{\alpha}) = \alpha\,\class{U}(f^*) =  \alpha\Var(V)\enspace.
\end{align*}
Therefore, with our notations we have
\begin{align}
    \label{eq:unf1}
    \class{U}(\hat f_{\alpha}) - \alpha\,\class{U}(f^*) = \alpha\parent{\Var(U) - \Var(V)}\enspace.
\end{align}
Furthermore, for all $\varepsilon \in (0, 1)$ we have that $\Var(U)$ equals to
\begin{align}
    \Var(U - V + V)
    &= \Var(U - V) + 2\Expf[(U - V - \Expf[U] + \Expf[V])(V - \Expf[V])] + \Var(V)\nonumber\\
    &\leq
    \Var(U-V) + 2\sqrt{\Var(U - V)\Var(V)} + \Var(V)\nonumber\\
    &\leq
    \sum_{s = 1}^Kw_s(\hb_s - b^*_s)^2 + 2\sqrt{\class{U}(f^*)}\sqrt{\sum_{s = 1}^Kw_s(\hb_s - b^*_s)^2} + \Var(V) \label{eq:unf2}
\end{align}
Finally, combining Eqs.~\eqref{eq:unf1} and~\eqref{eq:unf2} we deduce
\begin{align*}
    \class{U}(\hat f_{\alpha}) \leq \alpha\parent{\sum_{s = 1}^Kw_s(\hb_s - b^*_s)^2 + 2\sqrt{\class{U}(f^*)}\sqrt{\sum_{s = 1}^Kw_s(\hb_s - b^*_s)^2} + \class{U}(f^*)} \enspace.
\end{align*}
The proof of Eq.~\eqref{eq:unfairness_bound1} is concluded after factorizing the square of the \emph{r.h.s.} of the above bound.
To prove Eq.~\eqref{eq:unfairness_bound2}, we set $\alpha = 1$ in Eq.~\eqref{eq:unfairness_bound1} to get
\begin{align*}
     \class{U}^{\sfrac{1}{2}}(\hat f_{1}) \leq \class{U}^{\sfrac{1}{2}}(f^*) + \left\{{\sum_{s = 1}^Kw_s(\hb_s - b^*_s)^2} \right\}^{\sfrac{1}{2}}
     \enspace.
\end{align*}
The converse bound derived in a similar fashion using $\Var(V) \leq \Var(U - V) + 2 \sqrt{\Var(U - V)\Var(U)} + \Var(U)$.
\end{proof}

\begin{lemma}[General risk control]
    \label{lem:general_risk_control}
     Under notation of Lemma~\ref{lem:unfairness_estimator1} it holds that
     \begin{align*}
         \risk(\hat f_{\alpha})
         &\leq
         \sum_{s = 1}^Kw_s\Exp(\scalarin{\bsX}{\bbeta^* - \hbeta} + (b^*_s - \hb_s))^2\\
        &\phantom{=}+
        2(1 {-} \sqrt{\alpha})\sqrt{\sum_{s = 1}^Kw_s(b^*_s - \hb_s)^2}\sqrt{\sum_{s = 1}^K w_s\parent{\hb_s - \sum_{s' = 1}^Kw_{s'}\hb_{s'}}^2}\\
        &\phantom{=}+
        (1 {-} \sqrt{\alpha})^2\sum_{s = 1}^Kw_s\parent{\hb_s - \sum_{s' = 1}^Kw_{s'}\hb_{s'}}^2\enspace.
     \end{align*}
\end{lemma}
\begin{proof}
    Recall the expression for $\hat f_{\alpha}$
    \begin{align*}
        \hat f_{\alpha}(\bsx, s) = \scalarin{\bsx}{\hbeta} + \sqrt{\alpha}\hb_s + (1 {-} \sqrt{\alpha})\sum_{s = 1}^Kw_s\hb_s\enspace.
    \end{align*}
    Using this expression, we can write for the risk of $\hat f_{\alpha}$
    \begin{align*}
        \risk(\hat f_{\alpha})
        &=
        \sum_{s = 1}^Kw_s\Exp\parent{\scalarin{\bsX}{\bbeta^* - \hbeta} + (b^*_s - \hb_s) + (1 {-} \sqrt{\alpha})\parent{\hb_s - \sum_{s' = 1}^Kw_{s'}\hb_{s'}}}^2\\
        &\stackrel{(a)}{=}
        \sum_{s = 1}^Kw_s\Exp(\scalarin{\bsX}{\bbeta^* - \hbeta} + (b^*_s - \hb_s))^2
        +
        2(1 {-} \sqrt{\alpha})\sum_{s = 1}^Kw_s(b^*_s - \hb_s)\parent{\hb_s - \sum_{s' = 1}^Kw_{s'}\hb_{s'}}\\
        &\phantom{=}+
        (1 {-} \sqrt{\alpha})^2\sum_{s = 1}^Kw_s\parent{\hb_s - \sum_{s' = 1}^Kw_{s'}\hb_{s'}}^2\\
        &\stackrel{(b)}{\leq}
        \sum_{s = 1}^Kw_s\Exp(\scalarin{\bsX}{\bbeta^* - \hbeta} + (b^*_s - \hb_s))^2 + (1 {-} \sqrt{\alpha})^2\sum_{s = 1}^Kw_s\parent{\hb_s - \sum_{s' = 1}^Kw_{s'}\hb_{s'}}^2\\
        &\phantom{=}+
        2(1 {-} \sqrt{\alpha})\sqrt{\sum_{s = 1}^Kw_s(b^*_s - \hb_s)^2}\sqrt{\sum_{s = 1}^K w_s\parent{\hb_s - \sum_{s' = 1}^Kw_{s'}\hb_{s'}}^2}
    \end{align*}
    where $(a)$ follows from the fact that $\bsX$ is centered and $(b)$ is due to the Cauchy-Schwarz inequality.
\end{proof}

\begin{theorem}[Risk-unfairness bound for any $\tau$]
    \label{thm:final}
     Recall the definition of $\delta_n(p, K, t)$
     \begin{align*}
         \delta_n(p, K, t) = 8\parent{\frac{p}{n} + \frac{K}{n}} + 16\parent{\sqrt{\frac{p}{n}} + \sqrt{\frac{K}{n}}}\sqrt{\frac{t}{n}}
        +
        \frac{32t}{n}\enspace.
     \end{align*}
     On the event
     \begin{align*}
        \class{A} &= \left\{ \|\bfSigma^{1/2}(\bbeta^* - \hbeta)\|_2^2 + \sum_{s = 1}^K w_s(b^*_s - \hb_s)^2 \leq \sigma^2\delta_n(p, K, t)\right\}\enspace,
    \end{align*}
     it holds that
     \begin{align*}
         &\risk(\hat f_{\tau}) \leq \parent{{\sigma\delta^{\sfrac{1}{2}}_n(p, K, t)} + (1 - \sqrt{\tau}){\class{U}^{\sfrac{1}{2}}(\hat f_{1})}}^2\enspace.\\
         &\class{U}(\hat f_{\tau}) \leq \tau\class{U}(f^*)\parent{1 +  \nur{\delta^{\sfrac{1}{2}}_n(p, K, t)}}^2\enspace.
     \end{align*}
\end{theorem}
\begin{proof}
    We recall that Lemma~\ref{lem:unfairness_estimator1} gives, for any $\tau \in [0, 1]$,
    \begin{align}
        \class{U}(\hat f_{\tau}) = \tau\sum_{s = 1}^Kw_s\parent{\hb_s - \sum_{s' = 1}^Kw_{s'}\hb_{s'}}^2\enspace.
    \end{align}

    Let us start by proving the first part of the statement.
    Using Lemma~\ref{lem:general_risk_control} to upper bound the risk $\risk(\hat f_{\tau})$ and the definition of $\class{A}$ to control this upper bound, we obtain
    \begin{align*}
         \risk(\hat f_{\tau})
         &\leq
         \sigma^2\delta_n(p, K, t) + 2\sigma\sqrt{\delta_n(p, K, t)}\parent{(1 - \sqrt{\tau})\sqrt{\class{U}(\hat f_{1})}} + (1 - \sqrt{\tau})^2{\class{U}(\hat f_{1})}\\
         &=
         \parent{\sigma\sqrt{\delta_n(p, K, t)} + (1 - \sqrt{\tau})\sqrt{\class{U}(\hat f_{1})} }^2\enspace.
    \end{align*}
    The second part of the statement follow by applying Lemma~\ref{lem:unfairness_estimator2} and Theorem~\ref{thm:random_design} to get
    \begin{align}
        &\class{U}(\hat f_{\tau}) \leq \tau\class{U}(f^*)\parent{1 +  \nur\sqrt{\delta_n(p, K, t)}}^2\enspace.\label{eq:final_pre2}
    \end{align}

\end{proof}
\subsection{Proof of Theorem~\ref{thm:real_final}}
    We set $\hat f \eqdef \hat f_1$ and $\delta_n = \delta_n(p, K, t)$.
    The proof relies on Eq.~\eqref{eq:unfairness_bound2} of Lemma~\ref{lem:unfairness_estimator1}.
    Using notations of Theorem~\ref{thm:final} we also define the event
    \begin{align}
        \class{A} &= \left\{ \|\bfSigma^{1/2}(\bbeta^* - \hbeta)\|_2^2 + \sum_{s = 1}^K w_s(b^*_s - \hb_s)^2 \leq \sigma^2\delta_n(p, K, t)\right\}
    \end{align}
    which holds with probability at least $1-4\exp(-t)$.

    {\bf Case 1.} Assume that ${\class{U}^{\sfrac{1}{2}}(\hat f)} > \sigma{\delta_n^{\sfrac{1}{2}}(p, K, t)}$.
    Note that thanks to Theorem~\ref{thm:final}, and the definition of $\hat\tau$ we derive on the event $\class{A}$ that

    \begin{align*}
        \class{U}(\hat f_{\hat\tau})
        \leq
        \hat\tau\class{U}(f^*)\parent{1 + \sigma\sqrt\frac{\delta_n}{\class{U}(f^*)}}^2
        &=
        \alpha\,\class{U}(f^*)\parent{1 + \sigma\sqrt\frac{\delta_n}{\class{U}(f^*)}}^2\parent{1 + \frac{\sigma{\delta^{\sfrac{1}{2}}_n}}{{\class{U}^{\sfrac{1}{2}}(\hat f)} - \sigma{\delta^{\sfrac{1}{2}}_n}}}^{-2}\\
        &\stackrel{(a)}{\leq}
        \alpha\,\class{U}(f^*)\parent{1 + \sigma\sqrt\frac{\delta_n}{\class{U}(f^*)}}^2\parent{1 + \frac{\sigma{\delta^{\sfrac{1}{2}}_n}}{{\class{U}^{\sfrac{1}{2}}(f^*)}}}^{-2}\\
        &=
         \alpha\,\class{U}(f^*)\enspace.
    \end{align*}
    In the last equation, inequality $(a)$ follows from Eq.~\eqref{eq:unfairness_bound2} of Lemma~\ref{lem:unfairness_estimator1} and thanks to the fact that on the event $\class{A}$ it holds that $\class{U}^{\sfrac{1}{2}}(\hat f) \leq \class{U}^{\sfrac{1}{2}}(f^*) + \left\{{\sum_{s = 1}^Kw_s(\hb_s - b^*_s)^2} \right\}^{\sfrac{1}{2}} \leq \class{U}^{\sfrac{1}{2}}(f^*) + \sigma\delta_n^{\sfrac{1}{2}}$.
    For the risk we have thanks to Theorems~\ref{thm:final} that
    \begin{align}
        \label{eq:real_final0}
         \risk(\hat f_{\hat\tau}) \leq \parent{\sigma\sqrt{\delta_n} + (1 - \sqrt{\hat\tau})\sqrt{\class{U}(\hat f)} }^2\enspace.
    \end{align}
    Furthermore, we note that
    \begin{align}
        \sqrt{\hat\tau\class{U}(\hat f)}
        =
        \sqrt{\alpha} \frac{\sqrt{\class{U}(\hat f)}}{1 + \frac{\sigma\sqrt{\delta_n}}{\sqrt{\class{U}(\hat f)} - \sigma\sqrt{\delta_n}}}
        &=
        \sqrt{\alpha}\parent{\sqrt{\class{U}(\hat f)} - \sigma\sqrt{\delta_n}}\nonumber\\
        &\stackrel{(b)}{\geq} \sqrt{\alpha}\parent{\sqrt{\class{U}(f^*)} - 2\sigma\sqrt{\delta_n}}\label{eq:real_final1}\enspace,
    \end{align}
    where inequality $(b)$ again follows from Eq.~\eqref{eq:unfairness_bound2} of Lemma~\ref{lem:unfairness_estimator1} and thanks to the fact that on the event $\class{A}$ it holds that $\class{U}^{\sfrac{1}{2}}(\hat f) \geq \class{U}^{\sfrac{1}{2}}(f^*) - \left\{{\sum_{s = 1}^Kw_s(\hb_s - b^*_s)^2} \right\}^{\sfrac{1}{2}} \geq \class{U}^{\sfrac{1}{2}}(f^*) - \sigma\delta_n^{\sfrac{1}{2}}$.
    Recall, that we have already shown that on the event $\class{A}$ we have
    \begin{align}
    \label{eq:real_final2}
        {\class{U}^{\sfrac{1}{2}}(\hat f)} \leq {\class{U}^{\sfrac{1}{2}}( f^*)} + \sigma{\delta^{\sfrac{1}{2}}_n}\enspace.
    \end{align}
    Combining Eqs.~\eqref{eq:real_final1} and~\eqref{eq:real_final2} we obtain
    \begin{align*}
        (1 - \sqrt{\hat\tau})\sqrt{\class{U}(\hat f)}
        &\leq
        \sqrt{\class{U}( f^*)} + \sigma\sqrt{\delta_n} - \sqrt{\alpha}\parent{\sqrt{\class{U}(f^*)} - 2\sigma\sqrt{\delta_n}}\\
        &=
        (1 {-} \sqrt{\alpha})\sqrt{\class{U}(f^*)} + (1 + 2\sqrt{\alpha})\sigma\sqrt{\delta_n}
    \end{align*}
    Thus since the function $(\sigma{\delta^{\sfrac{1}{2}}_n} + \cdot)^2$ is increasing on $[-\sigma\sqrt{\delta_n}, \infty)$ we get from Eq.~\eqref{eq:real_final0} that
    \begin{align*}
        \risk(\hat f_{\hat\tau}) \leq \parent{2(1 + \sqrt{\alpha})\sigma\sqrt{\delta_n} + (1 {-} \sqrt{\alpha})\sqrt{\class{U}(f^*)}}^2\enspace,
    \end{align*}
    which concludes the proof of the first case.

    {\bf Case 2.} if ${\class{U}^{\sfrac{1}{2}}(\hat f)} \leq \sigma{\delta_n^{\sfrac{1}{2}}(p, K, t)}$, then
    \begin{align*}
        \hat f_0(\bsx, s) = \scalarin{\bsx}{\hbeta} + \sum_{s = 1}^Kw_s\hb_s\enspace.
    \end{align*}
    Furthermore, on the event $\class{A}$ thanks to Theorem~\ref{thm:final} it holds that $0 = \class{U}(\hat f_0) \leq \alpha\,\class{U}(f^*)$ and
    \begin{align*}
        \risk(\hat f_0)
        \leq
        \parent{\sigma{\delta^{\sfrac{1}{2}}_n} + \class{U}^{\sfrac{1}{2}}(\hat f)}^2
        &=
        \parent{\sigma{\delta^{\sfrac{1}{2}}_n} + \sqrt{\alpha}\class{U}^{\sfrac{1}{2}}(\hat f) + (1 {-} \sqrt{\alpha})\class{U}^{\sfrac{1}{2}}(\hat f)}^2\\
        &\leq
        \parent{(1 + \sqrt{\alpha})\sigma{\delta^{\sfrac{1}{2}}_n} + (1 {-} \sqrt{\alpha})\class{U}^{\sfrac{1}{2}}(\hat f)}^2\\
        &\leq
        \parent{(1 + \sqrt{\alpha})\sigma{\delta^{\sfrac{1}{2}}_n} + (1 {-} \sqrt{\alpha})\parent{\class{U}^{\sfrac{1}{2}}(f^*) + \sigma{\delta_n^{\sfrac{1}{2}}}}}^2\\
        &=
         \parent{2\sigma{\delta^{\sfrac{1}{2}}_n} + (1 {-} \sqrt{\alpha})\class{U}^{\sfrac{1}{2}}(f^*)}^2\enspace.
    \end{align*}
    The proof is concluded by application of Theorem~\ref{thm:random_design} to control the probability of event $\class{A}$.

\subsection{Auxiliary results for Theorem~\ref{thm:lower_linear}}
Let us first present auxiliary results used for the proof of Theorem~\ref{thm:lower_linear}.
The next lemma is known as Varshamov-Gilbert Lemma \citep{Varshamov57,Gilbert52}, its statement is taken from \cite[Lemma 4.12]{rigollet2015high}, see also~\cite[Lemma 2.9]{Tsybakov09}.

\begin{lemma}\label{lem:varshamov}
    Let $d \geq 1$ be an integer.
    There exist binary vectors $\bsomega_1, \dots, \bsomega_M \in \{0,1\}^d$ such that
    \begin{enumerate}
    \setlength{\itemsep}{2pt}%
        \item $\rho(\bsomega_j, \bsomega_{j'}) \geq \sfrac{d}{4}$ for all $j \neq j'$,
        \item $M = \lfloor e^{d / 16} \rfloor \geq e^{d / 32}$,
    \end{enumerate}
    where $\rho(\cdot, \cdot)$ is the Hamming's distance on binary vectors.
\end{lemma}

The next lemmas can be found in \cite[Lemma 5.1]{bellec2017optimal}, see also \cite[Lemma 3]{kerkyacharian2014optimal}.

\begin{lemma}
    \label{lem:Bellec}
    Let $(\Omega, \mathcal{A})$ be a measurable space and $M \geq 1$. Let $A_0, \dots, A_M$ be disjoint measurable events. Assume that $\Qprobf_0, \dots, \Qprobf_M$ are probability measures on $(\Omega, \mathcal{A})$ such that
    \begin{align*}
        \frac{1}{M}\sum_{j=1}^M \KL(\Qprobf_j, \Qprobf_0) \leq \kappa < \infty\enspace.
    \end{align*}
Then,
\begin{align*}
    \max_{j=0,\ldots,M} \Qprobf_j(A_j^c) \geq \frac{1}{12}\min(1, M\exp(-3\kappa))\enspace.
\end{align*}
\end{lemma}

Define the diagonal matrix $\bfW = \diag(w_1, \ldots, w_K)$.

\begin{lemma}\label{lem:expLB}
    Let $n \geq 1$ be an integer and $s > 0$ be a positive number. Let $M\geq 1$ and $(\bbeta_j, \bsb_j) \in \bbR^p \times \bbR^K$, $j = 0, \ldots, M$, such that $\normin{\bfSigma^{1/2}(\bbeta_j - \bbeta_k)}_2^2 + \normin{\bfW^{1/2}(\bsb_j - \bsb_k)}_2^2 \geq 4s$ for $j \neq k$. Assume that
    \begin{align*}
        \frac{1}{M}\sum_{j=1}^M \KL(\Probf_{(\bbeta_j, \bsb_j)}, \Probf_{(\bbeta_0, \bsb_0)}) \leq \kappa < \infty.
    \end{align*}
    Then, for any estimator $\hat f$,
    \begin{align*}
        \max_{j = 0, \ldots, M} \Probf_{(\bbeta_j, \bsb_j)}\left(\risk(\hat f) \geq s \right) \geq \frac{1}{12}\min \left(1, M\exp(-3\kappa)  \right)\enspace.
    \end{align*}
\end{lemma}
\begin{proof}
Denote by $A_j$ the event $\risk_j(\hat{f}) < s$ for $j=1, \dots, M$. Note that the events $A_0, \dots, A_M$ are pair-wise disjoint. Indeed, if they were not there would exist indices $j$ and $j'$, with $j \neq j'$, such that, on the non-empty event $A_j \cap A_{j'}$,
\begin{align}
     \normin{\bfSigma^{1/2}(\bbeta_j - \bbeta_{j'})}_2^2 + \normin{\bfW^{1/2}(\bsb_j - \bsb_{j'})}_2^2 \leq 2 \risk_j(\hat{f}) + 2 \risk_{j'}(\hat{f}) < 4s
\end{align}
contradicting our assumption on the $(\bbeta_j, \bsb_j)$ and $(\bbeta_{j'}, \bsb_{j'})$. We conclude applying Lemma~\ref{lem:Bellec}.
\end{proof}

\subsection{Proof of Theorem~\ref{thm:lower_linear}}

Define the $(p + K) \times (p + K)$ matrix
    \begin{align}
        \bfPsi =
       \left[
        \begin{array}{c|c}
          \bfSigma & \mathbf{0} \\
          \hline
          \mathbf{0} & \bfW
        \end{array}
     \right]\enspace,
    \end{align}

Apply Lemma~\ref{lem:varshamov} to obtain $\bsomega_0, \dots, \bsomega_M$ with $M+1 \geq e^{(p+K)/32}$ and such that $\rho(\bsomega_j, \bsomega_k) \geq (p+K)/4$. Let $\bsB_0 = (\bbeta_0, \bsb_0), \ldots, \bsB_M = (\bbeta_M, \bsb_M)$ be such that
\begin{align}
    \bsB_j = \varphi\sqrt{\frac{ \sigma^2}{{n}}} \parent{1+\sqrt{\sfrac{t}{(p+K)}}} \bfPsi^{-1/2} \bsomega_j,
\end{align}
with $p+K \leq 32\log(M)$ and $\varphi > 0$ to be determined later.

On the one hand we have
\begin{align*}
    \normin{\bfSigma^{1/2}(\bbeta_j - \bbeta_k)}_2^2 + \normin{\bfW^{1/2}(\bsb_j - \bsb_k)}_2^2 &= \frac{\varphi^2 \sigma^2}{n}(1+\sqrt{t/(p+K)})^2 \rho(\bsomega_j, \bsomega_{j'})\\
    &\geq \frac{\varphi^2 \sigma^2}{n}(1+\sqrt{t/(p+K)})^2 (p+K)/4\\
    &=  \frac{\varphi^2 \sigma^2}{4n}(\sqrt{p+K}+\sqrt{t})^2\enspace.
\end{align*}
On the other hand, recall that $\Prob_{Y | \bsX, S = s} = \class{N}(\scalarin{\bsX}{\bbeta} + b_s, \sigma^2)$ and $\Prob_{\bsX} = \class{N}(\boldsymbol{0}, \bfSigma)$, then, for a given $(\bbeta, \bsb) \in \bbR^p \times \bbR^K$ the joint distribution of observations is
\begin{align*}
    \Probf_{(\bbeta, \bsb)} = \bigotimes_{s = 1}^K \parent{\class{N}(\scalarin{\bsX}{\bbeta} + b_s, \sigma^2) \otimes \class{N}(\boldsymbol{0}, \bfSigma)}^{\otimes n_s} \enspace.
\end{align*}
Given $\bsB=(\bbeta, \bsb), \bsB'=(\bbeta', \bsb')$ in $\bbR^p \times \bbR^K$ we can write
\begin{align*}
    \KL\parent{ \Probf_{(\bbeta, \bsb)},  \Probf_{(\bbeta', \bsb')}}
    &=
    \sum_{s = 1}^K n_s\Exp_{\bsX \sim \class{N}(\boldsymbol{0}, \bfSigma)}\left[\widetilde\KL \parent{\class{N}(\scalarin{\bsX}{\bbeta} + b_s, \sigma^2), \class{N}(\scalarin{\bsX}{\bbeta'} + b_s', \sigma^2)}\right]\\
    &=
    \sum_{s = 1}^K n_s\Exp_{\bsX \sim \class{N}(\boldsymbol{0}, \bfSigma)}\parent{\frac{\parent{\scalarin{\bsX}{\bbeta - \bbeta'} + b_s - b_s'}^2}{2\sigma^2}}\\
    &=
    \sum_{s = 1}^K n_s\parent{\frac{\normin{\bfSigma^{1/2}(\bbeta - \bbeta')}_2^2}{2\sigma^2} + \frac{(b_s - b_s')^2}{2\sigma^2}}\\
    &=
    n\parent{\frac{\normin{\bfSigma^{1/2}(\bbeta - \bbeta')}_2^2}{2\sigma^2} + \frac{\normin{\bfW^{1/2}(\bsb - \bsb')}_2^2}{2\sigma^2}}\\
    &= \frac{n}{2\sigma^2}\normin{\bfPsi^{1/2}(\bsB-\bsB')}_2^2\\
    &\leq \frac{\varphi^2}{2} (\sqrt{p+K} + \sqrt{t})^2
    \leq
    \varphi^2(p + K) + \varphi^2 t \leq  32\varphi^2\log(M) + \varphi^2 t
    \enspace.
\end{align*}

Let $\hat{f}$ be any estimator and define the risks
\begin{align*}
    \risk_j(\hat{f}) = \sum _{s = 1}^K w_s \Exp \left[(\hat{f}(\bsX, S) - \langle \bsX, \bbeta_j \rangle - (b_{j})_S)^2 \mid  S = s\right], \quad j=1,\ldots,M\enspace.
\end{align*}

Set $u_n(p, K, t, \varphi, \sigma) = \frac{\varphi^2 \sigma^2}{16n}(\sqrt{p+K}+\sqrt{t})^2$.
Applying Lemma~\ref{lem:expLB} after reducing the supremum to a finite number of hypothesis, we get for all estimators $\hat f$ that
\begin{align*}
    \sup_{(\bbeta^*, \bsb^*)\in \mathbb{R}^{p} \times \bbR^K} \Probf_{(\bbeta^*, \bsb^*)}\left( \risk(\hat{f}) \geq u_n(p, K, t, \varphi, \sigma) \right)
    &\geq
    \max_{j=0, \dots, M} \Probf_{(\bbeta_j, \bsb_j)}\left( \risk_j(\hat{f}) \geq u_n(p, K, t, \varphi, \sigma) \right)\\
    &\geq
    \frac{1}{12}\min\parent{1, M\exp\parent{- 96\varphi^2\log(M) - 3\varphi^2 t}}\\
\end{align*}
Setting  $\varphi = 1/\sqrt{96}$, we obtain
\begin{align*}
    \sup_{(\bbeta^*, \bsb^*)\in \mathbb{R}^{p} \times \bbR^K} \Probf_{(\bbeta^*, \bsb^*)}\left( \risk(\hat{f}) \geq \frac{\sigma^2}{1536 n}\parent{\sqrt{p+K}+\sqrt{t}}^2 \right) \geq \frac{1}{12}\exp\parent{ - \frac{t}{32}}\enspace.
\end{align*}
The proof is concluded.

{
\section{Proofs for ad-hoc procedure}

\subsection{Demographic parity guarantee}
The proof of Theorem~\ref{thm:improved_fairness} relies on~\cite[Lemma 8.7]{vovk2005algorithmic} which is recalled below. We provide an alternative elementary proof of this result.
\begin{lemma}
    \label{lem:inverse_with_noise}
    Let $n \geq 1$.
  Let $V_1, \ldots, V_n, V_{n + 1}$ be exchangeable real-valued random variables and $U$ distributed uniformly on $[0, 1]$ be independent from $V_1, \ldots, V_n, V_{n + 1}$, then the statistics
  \begin{align*}
        T(V_1, \ldots, V_n, V_{n + 1}, U) =
        \frac{1}{n+1}\parent{\sum_{i = 1}^n\ind{V_i < V_{n + 1}} + U \cdot \parent{1 + \sum_{i = 1}^n\ind{V_{i} = V_{n+1}}}}\enspace,
    \end{align*}
    is distributed uniformly on $[0, 1]$.
\end{lemma}

\begin{proof}[Proof of Lemma~\ref{lem:inverse_with_noise}]
        Set $V \sim \frac{1}{n + 1}\sum_{i = 1}^{n + 1}\delta_{V_i}$ be independent from $U$ and denote by $F_{V | \bsV}$ the cumulative distribution function of $V$ conditionally on $\bsV \eqdef (V_1, \ldots, V_n, V_{n + 1})$. Define the following statistics
        \begin{align*}
            \class{T}(v, u) = \Probf(V < v \mid \bsV) + u\Probf(V = v \mid \bsV)\enspace,\qquad \forall\,\, (u,v) \in \bbR^2\enspace.
        \end{align*}
        Note that $\class{T}(V_{n + 1}, U) = T(V_1, \ldots, V_n, V_{n + 1}, U)$. In what follows we show that $\class{T}(V_{n + 1}, U)$ is distributed uniformly on $[0, 1]$. 
        Fix some $t \in [0, 1]$. Define $v(t) = F_{V | \bsV}^{-1}(t+)$ and
        \begin{align*}
            u(t) = 
            \begin{cases}
                 1, &\text{if }\Probf(V = v(t) \mid \bsV) = 0\\
                 \frac{t - \Probf(V < v(t) \mid \bsV)}{\Probf(V = v(t) \mid \bsV)} &\text{otherwise}
            \end{cases}\enspace.
        \end{align*}
        One can verify that by construction of $(u(t), v(t)) \in \bbR^2$ the event
        \begin{align*}
            \{\class{T}(V, U) \leq t\} \Longleftrightarrow \left\{V < v(t)\right\} \sqcup \big(\{V = v(t)\} \cap \{U < u(t))\} \big)\enspace,
        \end{align*}
        where $\sqcup$ stands for the disjoint union of two sets.
        Thus, since the events $\{V < v(t)\}$ and $\{V = v(t)\} \cap \{U < u(t))\}$ are disjoint and $U$ is distributed uniformly on $[0, 1]$ and is independent from $V, \bsV$ it holds that
        \begin{align*}
            \Probf(\class{T}(V, U) \leq t \mid \bsV)
            =
            \Probf(V < v(t) \mid \bsV) + \Probf(U < u(t) \mid \bsV)\Probf(V = v(t) \mid \bsV)
            = t\enspace.
        \end{align*}
        Integrating the above equality we get $\Probf(\class{T}(V, U) \leq t) = t$.
        To conclude, we first notice that thanks to the definition of $V$ it holds that
        \begin{align*}
            \Probf(\class{T}(V, U) \leq t)
            &=
            \Expf[\Probf(\class{T}(V, U) \leq t \mid \bsV)]\\
            &=
            \frac{1}{n + 1}\Expf\left[\sum_{i = 1}^{n + 1}\Probf(\class{T}(V_i, U) \leq t \mid \bsV)\right]\\
            &=
            \frac{1}{n + 1}\sum_{i = 1}^{n + 1}\Probf(\class{T}(V_i, U) \leq t)\enspace.
        \end{align*}
        Since, the random variables $V_1,\ldots, V_n, V_{n + 1}$ are exchangeable and the Bernoulli random variables $\ind{\class{T}(V_i, U) \leq t}$ for $i \in [n]$ are also exchangeable and are distributed identically, then
        \begin{align*}
            t = \Probf(\class{T}(V, U) \leq t) = \frac{1}{n + 1}\sum_{i = 1}^{n + 1}\Probf(\class{T}(V_i, U) \leq t) = \Probf(\class{T}(V_{n + 1}, U) \leq t)\enspace.
        \end{align*}
        The proof is concluded. 
\end{proof}

\begin{proof}[Proof of Theorem~\ref{thm:improved_fairness}]
    To prove the claimed guarantee, we will show that the Kolmogorov-Smirnov distance between $\Law(\hat{\Pi}(f)(\bsX, S) \mid S = s)$ and $\Law(\hat{\Pi}(f)(\bsX, S) \mid S = s')$ equals to zero for any $s \neq s'\in [K]$. For conciseness we define for all $s \neq s' \in [K]$
    \begin{align*}
        \KS_{s, s'} \coloneqq \KS\parent{\Law(\hat{\Pi}(f)(\bsX, S) \mid S = s),\, \Law(\hat{\Pi}(f)(\bsX, S)  \mid S = s')}\enspace.
    \end{align*}

    Note that, according to the definition of $\hat{\Pi}$ in Eq.~\eqref{eq:ad_hoc_def}, we have for any $(\bsx, s ) \in \bbR^p \times [K]$
    \begin{align*}
    \hat{\Pi}(f)(\bsx, s) = \hat{Q} \circ \hat{F}_{1, \nu^{f}_s} \circ f(\bsx, s)\enspace,
    \end{align*}
    with $\hat{Q}(\cdot) = \sum_{s' = 1}^K w_{s'}\hat{F}_{2, \nu^{f}_{s'}}^{-1}(\cdot)$. Note that, thanks to the splitting of the data into two parts, $\hat{Q}$ is independent from $({\hat{F}_{1, \nu^{f}_s} \circ f(\bsX, S) \mid S = s})$ for each $s \in [K]$.

    Fix some $s \in [K]$ and, for all $i = 1, \ldots, N_s$, set $V_i = \tilde{f}_i^s$ with $V_{N_s + 1} \stackrel{d}{=} (\tilde{f}(\bsX, S) \mid S = s)$ independent from $(V_i)_{i = 1, \ldots, N_s}$. Since the random variables $V_1, \ldots, V_{N_s}, V_{N_s + 1}$ are exchangeable (they are even independent), Lemma~\ref{lem:inverse_with_noise} implies that for all $s \in [K]$
    \begin{align*}
        \parent{\hat{F}_{1, \nu^{f}_s} \circ \tilde{f}(\bsX, S) \mid S = s} \text{ is distributed uniformly on } [0, 1]\enspace.
    \end{align*}
    The above arguments yield for all $s, s' \in [K]$ that
    \begin{align*}
        \KS_{s, s'}
        &=
        \sup_{t \in \bbR}\abs{\Probf\parent{\hat{\Pi}(f)(\bsX, S) \leq t \mid S = s} - \Probf\parent{\hat{\Pi}(f)(\bsX, S) \leq t \mid S = s'}}\\
        &=
        \sup_{t \in \bbR}\abs{\Probf\parent{\hat{F}_{1, \nu^{f}_s} \circ \tilde{f}(\bsX, S) \leq \hat{Q}^{-1}(t) \mid S = s} - \Probf\parent{\hat{F}_{1, \nu^{f}_{s'}} \circ \tilde{f}(\bsX, S) \leq \hat{Q}^{-1}(t) \mid S = s'}}\\
        &=
        \sup_{t \in \bbR}\abs{\Expf[\hat{Q}^{-1}(t) \mid S = s] - \Expf[\hat{Q}^{-1}(t) \mid S = s']} = 0\enspace.
    \end{align*}
    The first equality uses the definition of $\hat{\Pi}$; the second uses the fact that $\hat{Q}$ is monotone by construction and~\cite[Lemma~21.1(i)]{van2000asymptotic}; and the third one invokes the independence of $\hat{Q}$ from $\hat{F}_{1, \nu^{f}_{s}} \circ f(\bsX, S)$ conditionally on $S = s$ for any $s \in [K]$, Lemma~\ref{lem:inverse_with_noise}, and the independence of $\hat{Q}$ from $S$. The proof is concluded.
\end{proof}

\subsection{Risk guarantee}
\begin{proof}[Proof of Theorem~\ref{thm:risk_bound_ad_hoc}]
The proof of this result follows similar lines as that of~\cite[Theorem 4.4]{chzhen2020fair}. However, since our estimator is different, then several adaptations need to be introduced. Furthermore, our guarantees are stated for general $\ell_q$-norms.

    By the triangle inequality we can write
    \begin{align}
        \label{eq:ad_hoc0}
        \Expf\norm{\hat{\Pi}(f) - f^*_0}_q
        \leq \sum_{s' = 1}^Kw_{s'} \Expf \underbrace{\norm{\hat{F}_{2, \nu^{f}_{s'}}^{-1} \circ \hat{F}_{1, \nu^{f}_S} \circ \tilde{f} - F_{\nu^*_{s'}}^{-1} \circ F_{\nu^*_S} \circ f^*}_q}_{=: \texttt{A}_{s'}}\enspace,
    \end{align}
    where we introduced the short-hand notation $\texttt{A}_{s'}$.
    Furthermore, by definition of $\norm{\cdot}_q$ for each $s' \in [K]$, we can write
    \begin{align}
    \label{eq:ad_hoc1}
        \texttt{A}^q_{s'} =
        \sum_{s = 1}^Kw_s\texttt{A}^q_{(s', s)}\enspace,
    \end{align}
    where for all $s, s' \in [K]$, we defined
    \begin{align}
        \label{eq:ad_hoc2}
        \texttt{A}^q_{(s', s)} = \Exp\left[\abs{\hat{F}_{2, \nu^{f}_{s'}}^{-1} \circ \hat{F}_{1, \nu^{f}_s} \circ \tilde{f}(\bsX, S) - F_{\nu^*_{s'}}^{-1} \circ F_{\nu^*_s} \circ f^*(\bsX, S)}^q \mid S = s\right]\enspace.
    \end{align}
    From now on, our goal is to provide a suitable bound on $\texttt{A}^q_{(s', s)}$. To do so, using the fact that $(u + v + w)^q \leq 3^{q-1}u^q + 3^{q-1}v^q + 3^{q-1}w^q$, we further bound $\texttt{A}^q_{(s', s)}$ by $a_{(s', s)} + b_{(s', s)} + c_{(s', s)}$, defined as
    \begin{align*}
        a_{(s', s)} &= 3^{q-1}\Exp\left[\abs{\hat{F}_{2, \nu^{f}_{s'}}^{-1} \circ \hat{F}_{1, \nu^{f}_s} \circ \tilde{f}(\bsX, S) - \hat{F}_{\nu^*_{s'}}^{-1} \circ \hat{F}_{1, \nu^{f}_s} \circ \tilde{f}(\bsX, S)}^q \mid S = s\right]\enspace,\\
        b_{(s', s)} &= 3^{q-1}\Exp\left[\abs{\hat{F}_{\nu^*_{s'}}^{-1} \circ \hat{F}_{1, \nu^{f}_s} \circ \tilde{f}(\bsX, S) - F_{\nu^*_{s'}}^{-1} \circ \hat{F}_{1, \nu^{f}_s} \circ \tilde{f}(\bsX, S)}^q \mid S = s\right]\enspace,\\
        c_{(s', s)} &= 3^{q-1}\Exp\left[\abs{F_{\nu^*_{s'}}^{-1} \circ \hat{F}_{1, \nu^{f}_s} \circ \tilde{f}(\bsX, S) - F_{\nu^*_{s'}}^{-1} \circ F_{\nu^*_s} \circ f^*(\bsX, S)}^q \mid S = s\right]\enspace.
    \end{align*}
    In the above decomposition, we introduced the empirical quantile function $\hat{F}_{\nu^*_{s'}}^{-1}$ of $\nu^*_{s'}$, defined as the generalized inverse of
    \begin{align*}
        t \mapsto \frac{1}{N_{s'}}\sum_{i = N_{s'} + 1}^{2N_{s'}}\ind{f^*(\bsX_i^{s'}, s') \leq t}\enspace.
    \end{align*}
    Thus, for now, we have managed to show that for all $s' \in [K]$
    \begin{equation}
     \label{eq:ad_hoc3}
    \begin{aligned}
        \Expf[\texttt{A}_{s'}]
        &\leq
        \Expf\left\{\sum_{s = 1}^Kw_s\left(a_{(s', s)} + b_{(s', s)} + c_{(s', s)}\right) \right\}^{\sfrac{1}{q}}\\
        &\leq
        \left\{\sum_{s = 1}^Kw_s\Expf[a_{(s', s)}] \right\}^{\sfrac{1}{q}}
        +
        \left\{\sum_{s = 1}^Kw_s\Expf[b_{(s', s)}] \right\}^{\sfrac{1}{q}}
        +
        \Expf\left\{\sum_{s = 1}^Kw_sc_{(s', s)} \right\}^{\sfrac{1}{q}}\enspace.
    \end{aligned}
    \end{equation}
    We proceed by bounding each of these terms separately.
    \myparagraph{Bounding $\Expf[a_{(s', s)}]$}
    In order to bound $\Expf[a_{(s', s)}]$ we observe that conditionally on $(S = s, \bsX_{N_{s'}+1}^{s'}, \ldots, \bsX_{2N_{s'}}^{s'}, \zeta_{N_{s'} + 1}^{s'}, \ldots, \zeta_{2N_{s'}}^{s'}, f)$, by Lemma~\ref{lem:inverse_with_noise}, the random variable $(\hat{F}_{1, \nu^{f}_s} \circ \tilde{f}(\bsX, S) \mid S = s)$ is distributed uniformly on $[0, 1]$ (note that this is achieved by splitting the unlabeled sample into two parts and will be re-used for bounding $b_{(s', s)}$).
    Hence, the following equality holds for $\Expf[a_{(s', s)}]$:
    \begin{align*}
        \Expf[a_{(s', s)}] = 3^{q-1}\Expf\left[\int_{0}^1 \abs{\hat{F}_{2, \nu^{f}_{s'}}^{-1}(t) - \hat{F}_{\nu^*_{s'}}^{-1}(t)}^q \d t\right]\enspace.
    \end{align*}
    Applying~\cite[Theorem 2.10]{bobkov2019one} and the representation of the Wasserstein-$q$ distance as the minimum over all couplings of the $q$-th moment of the difference, we deduce that
    \begin{align*}
        \Expf\left[\int_{0}^1 \abs{\hat{F}_{2, \nu^{f}_{s'}}^{-1}(t) - \hat{F}_{\nu^*_{s'}}^{-1}(t)}^q \d t\right]
        &\leq
        \Expf\left[\frac{1}{N_{s'}}\sum_{i = N_{s'}+1}^{2N_{s'}}\abs{f(\bsX_i^{s'}, s') + \zeta_{i}^{s'} - f^*(\bsX_{i}^{s'}, s')}^q\right]\\
        &=
        2^{q-1}\Exp\left[ \abs{f(\bsX, S) - f^*(\bsX, S)}^q \mid S = s'\right] + 2^{q-1} \sigma^q\enspace,
    \end{align*}
    where the last equality uses the fact that $(\bsX_{i}^{s'})_{i = N_{s'} + 1}^{2N_{s'}}$ are \iid from $\Prob_{\bsX \mid S = s'}$ and that $\zeta_{i}^{s'}$ are bounded in $[-\sigma, \sigma]$.
    Thus, we have shown that
    \begin{align}
    \label{eq:bound_ass}
        \Expf[a_{(s', s)}] \leq 6^{q-1}\Exp[\abs{f(\bsX, S) - f^*(\bsX, S)}^q \mid S = s'] + 6^{q-1}\sigma^q\enspace,
    \end{align}
    which bounds the first term in Eq.~\eqref{eq:ad_hoc3}.
    
    \myparagraph{Bounding $\Expf[b_{(s', s)}]$} Identical argument with the use of Lemma~\ref{lem:inverse_with_noise} and~\cite[Theorem 2.10]{bobkov2019one} allows us to deduce that
    \begin{align*}
        \Expf[b_{(s', s)}] = 3^{q-1}\Expf[\sW_q^q(\hat{\nu}_{s'}, \nu_{s'}^*)]\enspace,
    \end{align*}
    with $\hat{\nu}_{s'}$ being the measure defined by its cumulative distribution $\hat{F}_{\nu^*_{s'}}$. To conclude, we deploy~\cite[Theorem 5.3]{bobkov2019one}, which states that
    \begin{align*}
        \Expf[\sW_q^q(\hat{\nu}_{s'}, \nu_{s'}^*)] \leq \left(\frac{5q}{\sqrt{N_{s'} + 2}}\right)^qJ_q(\nu_{s'}^*)\enspace,
    \end{align*}
    where
    \begin{align*}
        J_q(\nu_{s'}^*) \eqdef \int_{0}^1\parent{\frac{\sqrt{t(1-t)}}{F_{\nu_{s'}^*}'\circ F_{\nu_{s'}^*}^{-1} (t)}}^q \d t\enspace,
    \end{align*}
    $F'_{\nu_{s'}^*}$ is the derivative of $F_{\nu_{s'}^*}$. Note that Assumption~\ref{as:density_bound} guarantees that the integral $J_q(\nu_{s'}^*)$ is finite since it states that the measure  $\nu_{s'}^*$ is supported on an interval and that the associated density is lower bounded by $\underline{\lambda}_s$.
    Thus, we have shown that
    \begin{align*}
        \Expf[b_{(s', s)}] \leq \frac{1}{3}\left(\frac{15q}{\sqrt{N_{s'} + 2}}\right)^qJ_q(\nu_{s'}^*)\enspace,
    \end{align*}
    implying the following bound on the second term in Eq.~\eqref{eq:ad_hoc3}:
    \begin{align}
    \label{eq:bound_bss}
        \left\{\sum_{s = 1}^Kw_s\Expf[b_{(s', s)}] \right\}^{\sfrac{1}{q}} \leq 3^{-\sfrac{1}{q}}\frac{15q}{\sqrt{N_{s'} + 2}} J_q^{\sfrac{1}{q}}(\nu_{s'}^*) \enspace.
    \end{align}
    
    \myparagraph{Bounding $\Expf[c_{(s', s)}]$} By assumption, for each $s' \in [K]$ the measures $\nu^*_{s'}$ are supported on an interval and admit a density which is lower bounded by $\underline{\lambda}_{s'}$. Thus, the quantile function $F^{-1}_{\nu_{s'}^*}$ is $\underline{\lambda}_{s'}^{-1}$-Lipschitz.
    Introduce the following function 
    \begin{align*}
        t \mapsto F_{\nu^f_{s}}(t) = \Probf\left(f(\bsX, S) + \zeta \leq t \mid S = s\right)\enspace.
    \end{align*}
    By triangle inequality and $\underline{\lambda}_{s'}^{-1}$-Lipschitz property of $F^{-1}_{\nu_{s'}^*}$,  we deduce that
    \begin{align*}
        c_{(s', s)}
        \leq
        &6^{q-1}\underline{\lambda}_{s'}^{-q}\Exp\left[\abs{\hat{F}_{1, \nu^{f}_s} \circ \tilde{f}(\bsX, S) - F_{\nu^f_{s}} \circ \tilde{f}(\bsX, S)}^q \mid S = s\right]\\
        &+ 6^{q-1}\underline{\lambda}_{s'}^{-q}\Exp\left[\abs{ F_{\nu^f_{s}} \circ \tilde{f}(\bsX, S) -  F_{\nu^*_s} \circ f^*(\bsX, S)}^q \mid S = s\right]\enspace.
    \end{align*}
    The first term in the above decomposition is controlled by Hoeffding inequality conditionally on $(S = s, \bsX, \zeta)$.
    To do so, we observe that for any $t \in \bbR$ it holds that
    \begin{align*}
        \Expf\abs{\hat{F}_{1, \nu^{f}_s} (t) - F_{\nu^f_{s}} (t)}^q
        &\leq
        2^{q-1}\underbrace{\Expf\abs{\frac{1}{N_s}\sum_{i = 1}^{N_s}\parent{\ind{f(\bsX_{i}^s, s) + \zeta_i^s \leq t} - \Probf(f(\bsX, S) + \zeta \leq t \mid S = s)}}^q}_{\leq \frac{2}{(2N_s)^{\sfrac q 2}} \int_{0}^{\infty} \exp(-t^{\sfrac 2 q}) \d t \text{ by integrating Hoeffding's inequality}}\\
        &\phantom{\leq}+ \frac{2^{q-1}}{(N_{s} + 1)^q}
        \leq \frac{2^{q-1}}{N_s^{\sfrac q 2}} \left(\int_{0}^{\infty} \exp(-t^{\sfrac 2 q}) \d t + \frac{1}{2}\left(\frac{2}{N_s}\right)^{\sfrac q 2}\right) \enspace.
    \end{align*}
    Since the above holds for any $t$ and due to the enforced splitting, we deduce that
    \begin{align*}
        \Expf\left[\abs{\hat{F}_{1, \nu^{f}_s} \circ \tilde{f}(\bsX, S) - F_{\nu^f_{s}} \circ \tilde{f}(\bsX, S)}^q \mid S = s\right] \leq \frac{2^{q-1}}{N_s^{\sfrac q 2}} \left(\int_{0}^{\infty} \exp(-t^{\sfrac 2 q}) \d t + \frac{1}{2}\left(\frac{2}{N_s}\right)^{\sfrac q 2}\right)\enspace.
    \end{align*}
    It remains to bound
    \begin{align*}
        \Exp\left[\abs{ F_{\nu^f_{s}} \circ \tilde{f}(\bsX, S) - F_{\nu^*_s} \circ f^*(\bsX, S)}^q \mid S = s\right]\enspace,
    \end{align*}
    which is independent from the data.
    We apply Lemma~\ref{lem:for_p_moments} with $A = \tilde{f}(\bsX, S), B = f^*(\bsX, S)$, $M = \norm{f - f^*}_{\infty} + \sigma$ to deduce that $\Exp\left[\abs{ F_{\nu^f_{s}} \circ \tilde{f}(\bsX, S) -  F_{\nu^*_s} \circ f^*(\bsX, S)}^q \mid S = s\right]$ is bounded by
    \begin{align*}
         C(\overline{\lambda}_s, q)
        \begin{cases}
            \norm{f(\cdot, s) - f^*(\cdot, s)}_{1, \mu_s} + \sigma & q = 1\\
            \min\left\{\norm{f(\cdot, s) - f^*(\cdot, s)}_{q - 1, \mu_s}^{q-1} + \sigma^{q-1},\, \norm{f(\cdot, s) - f^*(\cdot, s)}_{\infty, \mu_s}^q + \sigma^q \right\} &q \in (1, \infty)
        \end{cases}\enspace,
    \end{align*}
    where $\mu_s = \Prob_{\bsX \mid S = s}$, for all $g: \bbR^p \to \bbR$ and all probability measures $\mu$ on $\bbR^p$
    \begin{align*}
        \norm{g}_{q, \mu} \eqdef \parent{\int g(\bsx) \d\mu(\bsx)}^{\sfrac{1}{q}}\qquad\text{and}\qquad \norm{g}_{\infty, \mu} \eqdef \inf\enscond{b \in \bbR}{\mu(g(\bsx) < b) = 0}\enspace.
    \end{align*}
    
    \myparagraph{Combining the deduced bounds}
    Combining Eq.~\eqref{eq:bound_ass},~\eqref{eq:bound_bss}, the bound above derived for $c_{(s', s)}$ with Eqs.~\eqref{eq:ad_hoc0}--\eqref{eq:ad_hoc3} we deduce that for constants $C_q(\underline{\boldsymbol\lambda}, \overline{\boldsymbol\lambda}), C'_q({\underline{\boldsymbol \lambda}}), C''_q({\underline{\boldsymbol \lambda}}) > 0$ that depend only on $\underline{\boldsymbol\lambda} \eqdef (\underline{\lambda}_s)_s, \overline{\boldsymbol\lambda} \eqdef (\overline{\lambda}_s)_s, q$ 
    \begin{align*}
        \sum_{s' = 1}^Kw_{s'}\texttt{A}_{s'} &\leq  C_q(\underline{\boldsymbol\lambda}, \overline{\boldsymbol\lambda})\left( \norm{f - f^*}_q +  \left\{\big(\norm{f - f^*}_{q-1}^{1 - \sfrac{1}{q}} + \sigma^{1 - \sfrac 1 q}\big)\wedge \left(\norm{f - f^*}_{\infty} + \sigma \right)\right\}\ind{q > 1} + \sigma \right)\\
        &\phantom{\leq} + C'_q({\underline{\boldsymbol \lambda}})\left\{\sum_{s = 1}^Kw_sN_{s}^{-\sfrac{1}{2}}\right\} + C''_q({\underline{\boldsymbol \lambda}}) \left\{\sum_{s = 1}^Kw_sN_s^{-\sfrac q 2}\right\}^{\sfrac{1}{q}}\enspace.
    \end{align*}
    Substituting the above into Eq.~\eqref{eq:ad_hoc0} we conclude.
\end{proof}

\begin{lemma}\label{lem:for_p_moments}
        Let $(A, B)$ be two real valued random variables (with an arbitrary coupling) with cumulative distribution function $F_A$ and $F_B$ respectively. Assume that $B$ admits a density \wrt the Lebesgue measure which is upper bounded by $C_B$. Then, for all $q \in [1, +\infty)$ it holds that
        \begin{align*}
            \Expf[\abs{F_A(A) - F_B(B)}^q] \leq
            \begin{cases}
                2C_B\Exp|A - B| &q=1\\
                C_qC_B^{q-1}\Exp\absin{A - B}^{q-1} &q \in (1, +\infty)
            \end{cases}\enspace,
        \end{align*}
        where $C_q$ depends only on $q$. Furthermore, if $(A, B)$ is such that $|A - B| \leq M$ almost surely, then
        \begin{align*}
            \Expf[\abs{F_A(A) - F_B(B)}^q] \leq
            \begin{cases}
                2C_B\Exp|A - B| &q=1\\
                \min \left\{C_qC_B^{q-1}\Exp\absin{A - B}^{q-1},\, 4^qC_B^qM^q\right\} &q \in (1, +\infty)
            \end{cases}\enspace,
        \end{align*}
    \end{lemma}
    \begin{proof}
        \textbf{Case: $q=1$.} Let $(\tilde{A}, \tilde{B})$ be an independent copy of $(A, B)$, then, since $\ind{a \leq a'} - \ind{b \leq b'} \leq \ind{|b-b'| \leq |a - b| + |a' - b'|}$ it holds that
        \begin{align*}
            \Expf[\abs{F_A(A) - F_B(B)}^q]
            &=
            \Expf\left[\abs{\Probf(\tilde{A} \leq A \mid A) - \Probf(\tilde{B} \leq B \mid B)}\right]\\
            &\leq
            \Expf\left[\Probf\left(\absin{B - \tilde{B}} \leq \absin{A - B} + \absin{\tilde{A} - \tilde{B}} \mid (A, B)\right)\right]\\
            &=
            \Probf\left(\absin{B - \tilde{B}} \leq \absin{A - B} + \absin{\tilde{A} - \tilde{B}}\right)\enspace.
        \end{align*}
        Furthermore, since $\ind{a \leq b + c} \leq \ind{2a \leq b} + \ind{2b \leq c}$ for all $a, b, c \geq 0$ we continue as
        \begin{align*}
            \Expf[\abs{F_A(A) - F_B(B)}^q]
            &\leq
            \Probf\left(2\absin{B - \tilde{B}} \leq \absin{A - B}\right) + \Probf\left(2\absin{B - \tilde{B}} \leq \absin{\tilde{A} - \tilde{B}}\right)\\
            &=2\Probf\left(2\absin{\tilde{B} - B} \leq \absin{A - B}\right)\\
            &=
            2\Expf\left[\Probf\left(2\absin{B - \tilde{B}} \leq \absin{A - B} \mid (A, B)\right)\right]\\
            &\leq 4C_B\Exp|A - B|\enspace,
        \end{align*}
        where the last inequality holds thanks to the assumption that the density of $B$ (and hence of $\tilde{B}$) is bounded by $C_B$, which in turn implies that $F_B(\cdot)$ is $C_B$-Lipschitz.\\
        \textbf{Case: $q \in (1, \infty)$.} At first we proceed similarly:
        \begin{align*}
            \Expf[\abs{F_A(A) - F_B(B)}^q]
            \leq 
            \Expf\left[\left\{\Probf\left(\absin{B - \tilde{B}} \leq \absin{A - B} + \absin{\tilde{A} - \tilde{B}} \mid (A, B)\right)\right\}^q\right]\enspace.
        \end{align*}
        Furthermore, for any $\delta > 0$, we can deduce by Markov's inequality
        \begin{align*}
            \Probf\left(\absin{B - \tilde{B}} \leq \absin{A - B} + \absin{\tilde{A} - \tilde{B}} \mid (A, B)\right)
            &\leq
            \Probf(\absin{B - \tilde{B}} \leq \delta \mid B)\\
            &\phantom{\leq}+ \Probf\left(\delta \leq \absin{A - B} + \absin{\tilde{A} - \tilde{B}} \mid (A, B)\right)\\
            &\leq
            2C_B\delta + \frac{\Exp[(\absin{A - B} + \absin{\tilde{A} - \tilde{B}})^{q - 1} \mid (A, B)]}{\delta^{q-1}}\enspace.
        \end{align*}
        Minimizing the above expression over $\delta > 0$ we deduce that
        \begin{align*}
            \Probf\left(\absin{B - \tilde{B}} \leq \absin{A - B} + \absin{\tilde{A} - \tilde{B}} \mid (A, B)\right) \leq 
            C_B^{1 - (\sfrac{1}{q})}C_q'\left(\Exp[(\absin{A - B} + \absin{\tilde{A} - \tilde{B}})^{q - 1} \mid (A, B)]\right)^{\sfrac{1}{q}}\enspace,
        \end{align*}
        where $C_q' = q(\sfrac {(q-1)} 2)^{\sfrac{1}{q}}$
        Thus, the claimed bound:
        \begin{align*}
            \Expf[\abs{F_A(A) - F_B(B)}^q]
            &\leq
            C_B^{q-1}(C_q')^q\Exp[(\absin{A - B} + \absin{\tilde{A} - \tilde{B}})^{q - 1}]
            \leq
            C_qC_B^{q-1}\Exp\absin{A - B}^{q-1}\enspace.
        \end{align*}
        \textbf{Case: $|A - B| \leq M$ almost surely.}
        In this case the result follows from the following chain of inequalities:
        \begin{align*}
            \Expf[\abs{F_A(A) - F_B(B)}^q]
            &\leq 
            \Expf\left[\left\{\Probf\left(\absin{B - \tilde{B}} \leq \absin{A - B} + \absin{\tilde{A} - \tilde{B}} \mid (A, B)\right)\right\}^q\right]\\
            &\leq \Expf\left[\left\{\Probf\left(\absin{B - \tilde{B}} \leq 2M \mid (A, B)\right)\right\}^q\right]
            \leq
            4^qC_B^qM^q
            \enspace.
        \end{align*}
        The proof is concluded.
    \end{proof}

}
\section{Relation between \texorpdfstring{$\class{U}_{\KS}$}{Lg} and \texorpdfstring{$\class{U}$}{Lg}}

\begin{lemma}
    \label{lem:ks_to_earth_mover}
    Let $\mu, \nu$ be two univariate measures such that $\mu$ admits a density \wrt the Lebesgue measure bounded by $C_{\mu}$, then
    \begin{align*}
        \KS(\mu, \nu) \leq 2\sqrt{C_{\mu}\sW_1(\mu, \nu)}\enspace.
    \end{align*}
\end{lemma}

\begin{proposition}
    Fix some measurable $f : \bbR^p \times [K] \to \bbR$.
    Assume that $a_s = f(\cdot, s)\#\mu_s \in \class{P}_2(\bbR)$ and it admits density bounded by $C_{f, s}$ for all $s \in [K]$, then
    \begin{align*}
        \class{U}_{\KS}(f) \leq  \|\sfrac{1}{\bsw}\|_{\infty}\sqrt{ 8\bar{C}_{f} }\cdot\class{U}^{1/4}(f)\enspace,
    \end{align*}
    where $\bar{C}_f = \sum_{s = 1}^Kw_sC_{f, s}$.
\end{proposition}

\begin{proof}
    We set $a_s =  \Law(f(\bsX, S) \mid S{=}s)$ and $a = \sum_{s = 1}^K w_s a_s$.
    Therefore, thanks to assumption of the proposition and Lemma~\ref{lem:ks_to_earth_mover} we can write
    \begin{align*}
        \class{U}_{\KS}(f)
        \eqdef
        \sum_{s = 1}^K \KS(a_s, a)
        \leq
        \|\sfrac{1}{\bsw}\|_{\infty}\sum_{s = 1}^K w_s\KS(a_s, a)
        \leq
        2\|\sfrac{1}{\bsw}\|_{\infty}\sum_{s = 1}^K w_s C_{f, s}^{\sfrac{1}{2}}W^{\sfrac{1}{2}}_1(a_s, a)\enspace.
    \end{align*}
    Furthermore we can write for any measure $\nu \in \class{P}_2(\bbR)$ that
    \begin{align*}
        \class{U}_{\KS}(f)
        &\stackrel{(a)}{\leq}
        2\|\sfrac{1}{\bsw}\|_{\infty}\sum_{s = 1}^K w_s C_{f, s}^{\sfrac{1}{2}}\left\{\sum_{s' = 1}^K w_{s'}\sW_1(a_s, a_{s'})\right\}^{\sfrac{1}{2}}\\
        &\stackrel{(b)}{\leq}
        2\|\sfrac{1}{\bsw}\|_{\infty}\sum_{s = 1}^K w_s C_{f, s}^{\sfrac{1}{2}}\left\{\sW_1(a_s, \nu) + \sum_{s' = 1}^K w_{s'}\sW_1(a_{s'}, \nu)\right\}^{\sfrac{1}{2}}\enspace.
    \end{align*}
    In the above inequalities $(a)$ follows from the convexity of $\sW_1(a_s, \cdot)$ \cite[see \eg][Section 4.1]{bobkov2019one} and $(b)$ uses the triangle inequality.
    Applying the Cauchy–Schwarz inequality we obtain
    \begin{align*}
        \class{U}_{\KS}(f)
        &\leq
        2\|\sfrac{1}{\bsw}\|_{\infty}\left\{\sum_{s = 1}^K w_s C_{f, s}\right\}^{\sfrac{1}{2}}\left\{\sum_{s = 1}^Kw_s \parent{\sW_1(a_s, \nu) + \sum_{s' = 1}^K w_{s'}\sW_1(a_{s'}, \nu)}\right\}^{\sfrac{1}{2}}\\
        &=
        2^{\sfrac{3}{2}}\|\sfrac{1}{\bsw}\|_{\infty}\left\{\sum_{s = 1}^K w_s C_{f, s}\right\}^{\sfrac{1}{2}}\left\{\sum_{s = 1}^Kw_s\sW_1(a_s, \nu)\right\}^{\sfrac{1}{2}}\\
        &\stackrel{(c)}{\leq}
        2^{\sfrac{3}{2}}\|\sfrac{1}{\bsw}\|_{\infty}\left\{\sum_{s = 1}^K w_s C_{f, s}\right\}^{\sfrac{1}{2}}\left\{\sum_{s = 1}^Kw_s\sW_1^2(a_s, \nu)\right\}^{\sfrac{1}{4}}\enspace,
    \end{align*}
    where  $(c)$ uses the Cauchy–Schwarz inequality one more time.
    Finally, setting $\nu$ as the Wasserstein-2 barycenter of $a_1, \ldots, a_K$ and using the fact that $\sW_1(\mu, \nu) \leq \sW_2(\mu, \nu)$ we deduce that
    \begin{align*}
        \class{U}_{\KS}(f) \leq 2^{\sfrac{3}{2}}\|\sfrac{1}{\bsw}\|_{\infty}\left\{\sum_{s = 1}^K w_s C_{f, s}\right\}^{\sfrac{1}{2}}\class{U}^{1/4}(f)\enspace.
    \end{align*}
    The proof is concluded.
\end{proof}

{

\section{Extension to \texorpdfstring{${\ell_{q}}$}{Lg} losses}
\label{sec:q_losses}

In this section we provide an extension of the derived theory for the case of $\ell_q$ losses with $q \in [1, +\infty)$.
In other words, instead of the $\ell_2$-based risk and unfairness we define
\begin{align*}
    \risk_q(f) &= \sum_{s = 1}^K w_s \Exp\left[\abs{f(\bsX, S) - f^*(\bsX, S)}^q \mid S = s\right]\enspace,\\
    \class{U}_q(f) &= \min_{\nu \in \class{P}_q(\bbR)}\sum_{s = 1}^K w_s \sW_q^q\left(\Law(f(\bsX, S) \mid S = s), \,\nu\right)\enspace.
\end{align*}

In this section we will work under the following assumption, which is a straightforward adaptation of Assumption~\ref{as:atomless} to handle $q \in [1, \infty) \setminus \{2\}$.
\begin{assumption}\label{as:atomless_q}
    The measures $\{\nu^*_s\}_{s \in [K]}$ are non-atomic and have finite $q$-moments.
\end{assumption}

\begin{theorem}[$q$-Optimal and the trade-off]
\label{thm:main_q}
Fix some $q \in [1, \infty)$ and let Assumption~\ref{as:atomless} be satisfied.
Then, for any $\alpha \in [0,1]$, a $f^*_{\alpha, q}$ solution of
\begin{align*}
    \min_{f: \bbR^p \times [K] \to \bbR}\enscond{\risk_q(f)}{\class{U}_q(f) \leq \alpha \cdot \class{U}_q(f^*)}\enspace,
\end{align*}
can be written for all $(\bsx, s) \in \bbR^p \times [K]$ as
\begin{align}
    \label{eq:all_alpha_q}
    f^*_{\alpha, q}(\bsx, s) = \alpha^{\sfrac{1}{q}}f^*(\bsx, s) + \parent{1{-}\alpha^{\sfrac{1}{q}}} \cdot \argmin_{y \in \bbR} \ens{\sum_{s' = 1}^Kw_{s'}\abs{F^{-1}_{\nu^*_{s'}} \circ F_{\nu^*_s} \circ f^*(\bsx, s) - y}^q}\enspace.
\end{align}
Furthermore, it holds that
\begin{align*}
    \risk_q(f^*_{\alpha, q}) = \parent{1 - \alpha^{\sfrac{1}{q}}}^q\cdot\class{U}_q(f^*)\qquad\text{and}\qquad\class{U}_q(f^*_{\alpha, q}) = \alpha\cdot \class{U}_q(f^*)\enspace.
\end{align*}
\end{theorem}
\begin{remark}
Note that this result is a strict generalization of Proposition~\ref{prop:optimal_alpha}. Furthermore, it should be noted that for $q = 1$ the optimization problem appearing in Eq.~\eqref{eq:all_alpha_q} (description of $f^*_{\alpha, 1}$) does not necessarily admit a unique minimizer (whenever $K$ is even). Meanwhile, for $q > 1$ the objective function is strictly convex and the minimizer is a singleton. We also observe that all the properties except the fourth one established in Section~\ref{SUBSEC:GENERAL} still hold for $f^*_{\alpha, q}$. The fourth property---mean stability---is intrinsic to the case $q=2$, as the minimizer appearing in Eq.~\eqref{eq:all_alpha_q} admit a closed form expression. Note also that the ad-hoc procedure developed in Section~\ref{sec:adhoc} for $q=2$ can be straightforwardly extended to any $q \in [1, \infty)$.
\end{remark}
\begin{proof}[Proof of Theorem~\ref{thm:main_q}]
    The proof is almost identical to that of Proposition~\ref{prop:optimal_alpha}. Thus, we only develop the modifications that need to be introduced. In particular, as for the case of $q=2$, our goal is to apply Lemma~\ref{lem:geometric_general_extended} below (instead of Lemma~\ref{lem:geometric_general}). To this end, we need a way to build geodesics in $(\class{P}_q(\bbR), \sW_q)$ and suitable expressions for $q$-barycenters.
    The first part (geodesics) is addressed by ~\cite[Theorem 5.27]{santambrogio2015optimal} (instead of~\cite[Section 2.2]{kloeckner2010geometric}) and gives identical description of geodesics for all $q \in [1, \infty)$. A closed form expressions for $q$-barycenters $(\class{P}_q(\bbR), \sW_q)$ is derived in Lemma~\ref{lem:barycenter_p} (used instead of Lemma~\ref{lem:A3}) which we prove below.
\end{proof}

\begin{definition}[$q$-barycenter]
    \label{def:barycenter_space_extended}
    Fix some $q \in [1, \infty)$.
    We say that a metric space $(\class{X}, d)$ satisfies the $q$-barycenter property if for any weights $\bsw \in \Delta^{K - 1}$ and tuple $\bsa = (a_1, \ldots, a_K) \in \class{X}^K$ there exists a barycenter
    \begin{align*}
        C_{\bsa_{\bsw}} \in \argmin_{C \in \class{X}} \sum_{s = 1}^K w_s d^q(a_s, C)\enspace.
    \end{align*}
    Moreover, for any tuple $\bsa = (a_1, \ldots, a_K) \in \class{X}^K$ we denote\footnote{When there is no ambiguity in the weights $\bsw$ we simply write $C_{\bsa}$.} by $C_{\bsa_{\bsw}}$ a barycenter of $\bsa$ weighted by $\bsw \in \Delta^{K - 1}$.
\end{definition}
\begin{lemma}[Abstract geometric lemma]\label{lem:geometric_general_extended}
    Fix some $q \in [1, +\infty)$.
    Let $(\class{X}, d)$ be a metric space satisfying the $q$-barycenter property.
    Let $\bsa = (a_1, \ldots, a_K) \in \class{X}^K$, $\bsw = (w_1, \ldots, w_K)^\top \in \Delta^{K - 1}$ and let $C_{\bsa}$ be a $q$-barycenter of $\bsa$ with respect to weights $\bsw$.
    For a fixed $\alpha \in [0, 1]$ assume that there exists $\bsb = (b_1, \ldots, b_K) \in \class{X}^K$ which satisfies
    \begin{alignat}{3}
        &d(a_s, C_{\bsa}) = d(a_s, b_s) + d(b_s, C_{\bsa})\enspace, \qquad &&s = 1,\dots,K\enspace, \tag{$P_1^q$}\label{eq:prop_1_extended}\\
        &d(b_s, a_s) = \big(1 {-} \alpha^{\sfrac{1}{q}}\big)d(a_s, C_{\bsa})\enspace, \qquad &&s = 1,\dots,K\enspace. \tag{$P_2^q$}\label{eq:prop_2_extended}
    \end{alignat}
    Then, $\bsb$ is a solution of
    \begin{align}
        \label{eq:geom_lemma_the_problem_extended}
        \inf_{\bsb \in \class{X}^K}\enscond{\sum_{s = 1}^K w_s d^q(b_s, a_s)}{\sum_{s = 1}^K w_s d^q(b_s, C_{\bsb}) \leq {\alpha} \sum_{s = 1}^K w_s d^q(a_s, C_{\bsa}) }\enspace.
    \end{align}
\end{lemma}
\begin{proof}
    The proof is identical to that of Lemma~\ref{lem:geometric_general}. Formally, it amounts to replacing all occurrences of $q=2$ by general $q \in [1, \infty)$.
\end{proof}

The next lemma is reminiscent to \cite[Theorem~8]{le2017existence}, where the authors study the existence of $q$-barycenters in rather general metric spaces. In our case, however, the underlying space is $\bbR$ and hence a finer characterisation can be obtained.
\begin{lemma}
    \label{lem:barycenter_p}
    Let $\mu_1, \ldots, \mu_K \in \class{P}_q(\bbR)$ and $\bsw \in \Delta^{K-1}$. Assume that $\mu_1, \ldots, \mu_K$ admit density \wrt the Lebesgue measure, then for any $q \in [1, \infty)$ the quantile function $F^{-1}_{\nu^*}$ of a minimizer $\nu^*$ of
    \begin{align*}
       \nu \mapsto \sum_{s = 1}^Kw_s\sW_{q}^q(\mu_s, \nu)\enspace,
    \end{align*}
    can be written for all $t \in (0, 1)$ as
    \begin{align*}
        F^{-1}_{\nu^*}(t) \in \argmin_{y \in \bbR}\sum_{s = 1}^Kw_s|F_{\mu_s}^{-1}(t) - y|^q\enspace.
    \end{align*}
    Furthermore, for all $s \in [K]$ the following mapping is an optimal transport map (in $\sW_q$ sense) from $\mu_s$ to $\nu^*$:
    \begin{align*}
        T_{\mu_s \mapsto \nu^*}(x) \in \argmin_{y \in \bbR} \sum_{s' = 1}^Kw_{s'}\abs{F^{-1}_{\mu_{s'}} \circ F_{\mu_s} (x) - y}^q\qquad\forall x \in \bbR\enspace.
    \end{align*}
\end{lemma}
\begin{proof}
It is clear that by definition $F_{\nu^*}^{-1}$ is monotone non-decreasing, hence, given its domain, it is indeed a quantile function.
Furthermore, note that $\nu^* \in \class{P}_q(\bbR)$. Indeed, for all $s \in [K]$ we have by definition of $T_{\mu_s \mapsto \nu^*}$
\begin{align*}
    \int_{\bbR} |x|^q \d \nu^*(x) = \int_{\bbR} |T_{\mu_s \mapsto \nu^*}(x)|^q\d\mu_s(x)
    &\leq 2^{q}\sum_{s' \in [K]}w_{s'}\int_{\bbR}|F^{-1}_{\mu_{s'}} \circ F_{\mu_s} (x)|^q \d \mu_s(x)\\
    &=2^q\sum_{s' \in [K]}w_{s'}\int_{\bbR}|x|^q \d\mu_{s'}(x) < \infty\enspace,
\end{align*}
since each $\mu_{s'} \in \class{P}_q(\bbR)$.
Furthermore, using~\cite[Theorem 2.10]{bobkov2019one}, it holds for all $\nu \in \class{P}_q(\bbR)$ that
\begin{align*}
    \sum_{s = 1}^Kw_s\sW_{q}^q(\mu_s, \nu)
    =
    \sum_{s = 1}^Kw_s\int_{0}^1|F_{\mu_s}^{-1}(t) - F_{\nu}^{-1}(t)|^q \d t
    &\geq 
     \int_{0}^1\min_{y}\sum_{s = 1}^Kw_s|F_{\mu_s}^{-1}(t) - y|^q \d t\\
     &=
     \sum_{s = 1}^Kw_s\int_{0}^1|F_{\mu_s}^{-1}(t) - F_{\nu^*}^{-1}(t)|^q \d t\\
     &=\sum_{s = 1}^Kw_s\sW_{q}^q(\mu_s, \nu^*)\enspace.
\end{align*}
The last assertion of the lemma follows from~\cite[Theorem 5.27]{santambrogio2015optimal}.
    
\end{proof}

\section{Additional empirical results}

\begin{table}[ht!]
\centering
{%
\begin{tabular}{@{}l|cc|cc|cc@{}}
\multicolumn{7}{c}{\textsc{Balanced}}\\
\midrule
&
\multicolumn{2}{c|}{Oracle}&
\multicolumn{2}{c|}{Proposed}&
\multicolumn{2}{c}{Naive}\\
$\alpha$&
$\risk(f^*_{\alpha})$&
$\class{U}(f^*_{\alpha})$&
$\risk(\hat f_{\hat\tau})$&
$\class{U}(\hat f_{\hat\tau})$&
$\risk(\hat f_{\alpha})$&
$\class{U}(\hat f_{\alpha})$\\
\midrule
$0$&
$2.0$&
$0.0$&
$2.13 \pm 0.03$&
$0.0 \pm 0.0$&
$2.13 \pm 0.03$&
$0.0 \pm 0.0$\\
$0.2$&
$0.61$&
$0.4$&
$0.87 \pm 0.05$&
$0.31 \pm 0.02$&
$0.74 \pm 0.04$&
$0.40 \pm 0.03$\\
$0.4$&
$0.27$&
$0.8$&
$0.52 \pm 0.05$&
$0.62 \pm 0.05$&
$0.40 \pm 0.04$&
$0.81 \pm 0.05$\\
$0.6$&
$0.10$&
$1.2$&
$0.34 \pm 0.04$&
$0.93 \pm 0.07$&
$0.24 \pm 0.04$&
$1.21 \pm 0.08$\\
$0.8$&
$0.02$&
$1.6$&
$0.23 \pm 0.04$&
$1.25 \pm 0.09$&
$0.16 \pm 0.03$&
$1.61 \pm 0.11$\\
$1$&
$0.0$&
$2.0$&
$0.17 \pm 0.03$&
$1.56 \pm 0.12$&
$0.14 \pm 0.03$&
$2.02 \pm 0.14$\\
\bottomrule
\end{tabular}%
}
\caption{Summary for $p = 10, K = 5, \nur = 0.5$. We report the mean and the standard deviation.}
\label{table1}
\end{table}

Table~\ref{table1} presents the numeric results for $p = 10$, $K = 5$, $\nur = 0.5$ for estimator developed in Section~\ref{SEC:LINEAR} and in the context of the simulated data. We remark the striking drop in the risk for $\alpha = 0.2$, indicating that a slight relaxation of the Demographic Parity constraint results in a significant improvement in terms of the risk. Of course, the justification of such a relaxation must be considered based on the application at hand.

\subsection{Real data description}
\texttt{Communities and Crime} dataset  combines socio-economic data from the 1990 US Census, law enforcement data from the 1990 US LEMAS survey, and crime data from the 1995 FBI UCR.
We removed columns with missing values as well as the only non-numerical variable, \texttt{communityname}. After this pre-processing step the dataset consists of $1994$ observations characterized each by $101$ $(d=101)$ numeric variables. All variables were normalized into the interval $[0, 1]$.
The target variable measures the total number of violent crimes per $100$K population.
For a given observation/community, we define the sensitive variable $(K=2)$ as the indicator that the percentage of population that is African American is above the average of this percentage over the whole dataset, similarly to \cite{berk2017convex}.

\subsection{Ad-hoc estimator from Section~\ref{sec:adhoc}: choices of weights}

We tested three choices of weights which are described in Section~\ref{sec:weights_choice}. Figure~\ref{fig:weights_empirical} can be seen as an empirical counterpart to Figure~\ref{fig:weights} discussed in details Section~\ref{sec:weights_choice}. Overall, the conclusions about the choices of weights remain the same and the actual choice is ultimately left to the statistician and domain experts.

\begin{figure*}[ht!]
    \centering
    \begin{subfigure}[b]{0.475\textwidth}
        \centering
     \includegraphics[width=\textwidth]{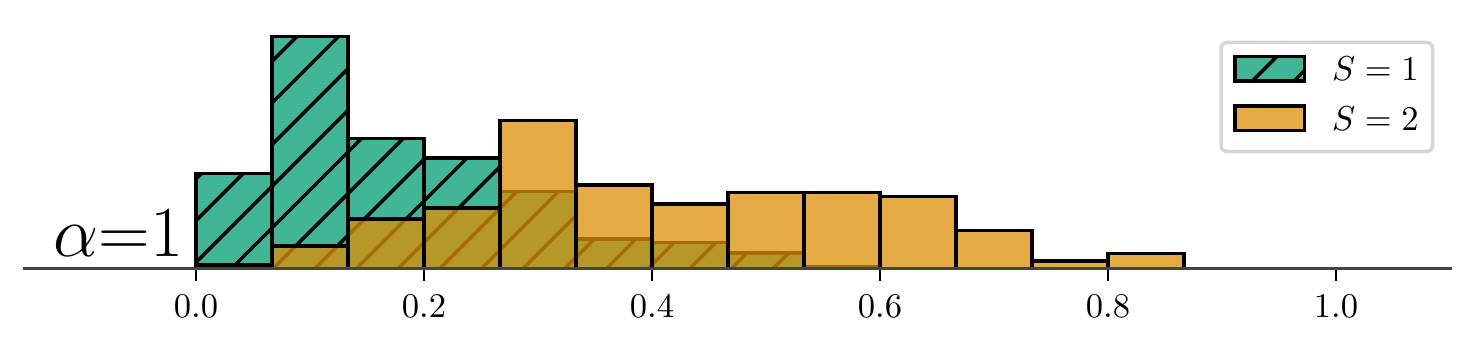}
        \caption{No fairness adjustment}   
        \label{fig:mean and std of net14}
    \end{subfigure}
    \hfill
    \begin{subfigure}[b]{0.475\textwidth}  
        \centering 
        \includegraphics[width=\textwidth]{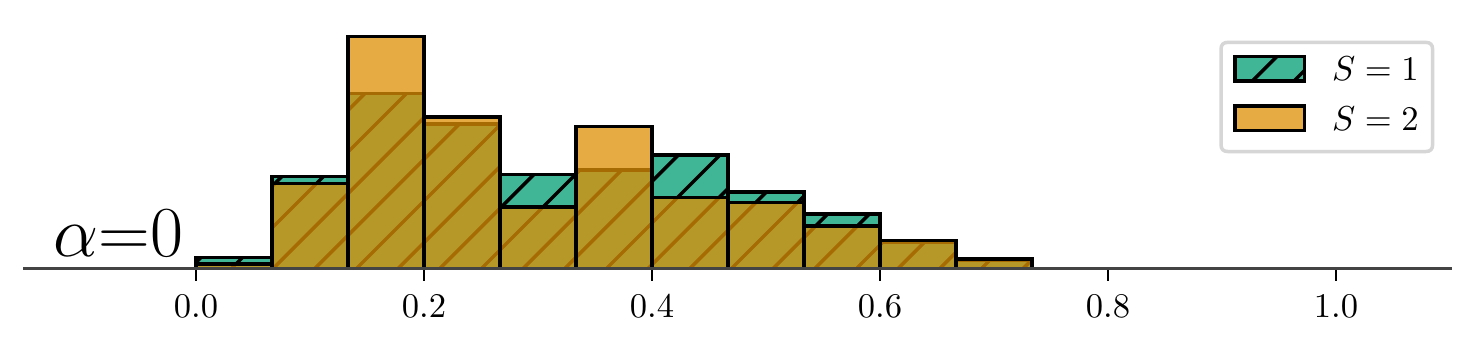}
        \caption{$w_s = \tfrac{1}{K}$}   
        \label{fig:mean and std of net24}
    \end{subfigure}
    \vskip\baselineskip
    \begin{subfigure}[b]{0.475\textwidth}   
        \centering 
        \includegraphics[width=\textwidth]{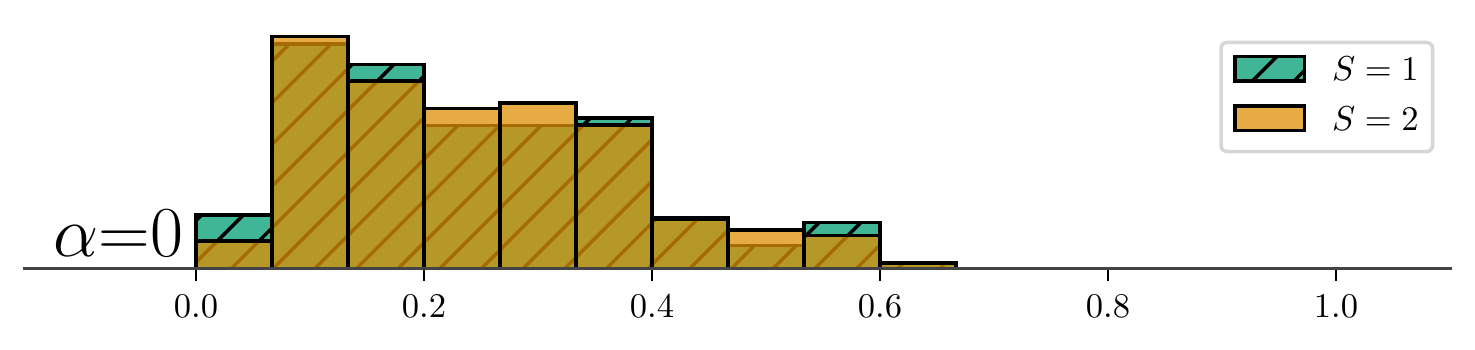}
        \caption{$w_s = \Prob(S = s)$}   
        \label{fig:mean and std of net34}
    \end{subfigure}
    \hfill
    \begin{subfigure}[b]{0.475\textwidth}   
        \centering 
        \includegraphics[width=\textwidth]{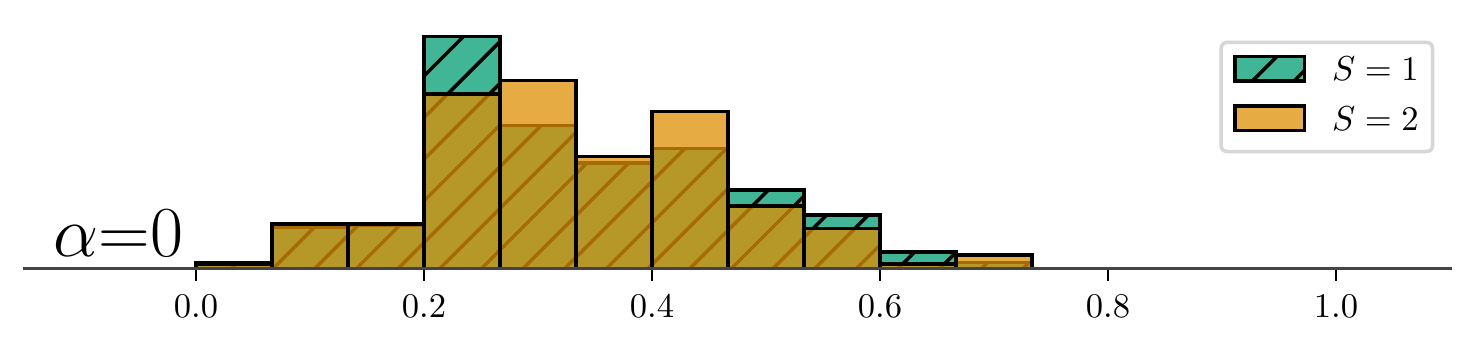}
        \caption{$w_s = 1/\Prob(S = s)$}      
        \label{fig:mean and std of net44}
    \end{subfigure}
    \caption{Three main choices of weights for the post-processing method applied on top of the random forest base estimator.}
    \label{fig:weights_empirical}
\end{figure*}

}

\stopcontents[appendices]

\end{document}